\documentclass[11pt]{article}
\usepackage[a4paper, width=16cm]{geometry}
\usepackage{authblk}
\usepackage{datetime}
\usepackage{amsmath,amsfonts,amssymb,amsthm}
\usepackage{mathtools}  
\usepackage{amsthm}         
\usepackage{breqn}
\usepackage[dvipsnames]{xcolor}
\usepackage{tcolorbox}
\usepackage{multirow}
\usepackage{hyperref}
\usepackage{fullpage}
\usepackage{enumerate}
\usepackage{subfigure}
\usepackage[title]{appendix}
\usepackage[textsize=tiny]{todonotes}
\usepackage{lineno}
\usepackage[normalem]{ulem}
\usepackage{lmodern}
\usepackage{cite}
\usepackage{listings}
\usepackage{enumitem}
\newlist{steps}{enumerate}{1}
\setlist[steps, 1]{label = Step \arabic*:}
\usepackage[normalem]{ulem}
\usepackage{caption} 
\usepackage{float}

\hypersetup{
	colorlinks=true, 
	linkcolor=blue, 
	urlcolor=red, 
	citecolor=magenta,
	linktoc=all 
}

\newcommand{\al}{\alpha}
\newcommand{\ub}{\mathbf{u}}

\newcommand{\Ub}{\mathbf{U}}
\newcommand{\Wb}{\mathbf{W}}
\newcommand{\fb}{\mathbf{f}}
\newcommand{\Db}{\mathbf{D}}
\newcommand{\Fb}{\mathbf{F}}
\newcommand{\Eb}{\mathbf{E}}
\newcommand{\Ab}{\mathbf{A}}
\newcommand{\Bb}{\mathbf{B}}
\newcommand{\Vb}{\mathbf{V}}
\newcommand{\Jb}{\mathbf{J}}
\newcommand{\pa}{\partial}
\newcommand{\na}{\nabla}

\newcommand{\Dx}{\Delta x}
\newcommand{\Dy}{\Delta y}
\newcommand{\Dt}{\Delta t}
\newcommand{\iph}{i+\frac{1}{2}}
\newcommand{\jph}{j+\frac{1}{2}}
\newcommand{\imh}{i-\frac{1}{2}}
\newcommand{\jmh}{j-\frac{1}{2}}

\usepackage{xcolor}  
\definecolor{lightyellow}{HTML}{fffcf1}

\newtheorem{remark}{Remark}[section]

\newtheorem{thm}{Theorem}[section]
\newtheorem{prop}[thm]{Proposition}

\newcommand{\mde}{\textbf{O2EXP-MultiD}}
\newcommand{\mdi}{\textbf{O2IMEX-MultiD}}
\newcommand{\phme}{\textbf{O2EXP-PHM}}
\newcommand{\phmi}{\textbf{O2IMEX-PHM}}
\newcommand{\ote}{\textbf{O2EXP}}
\newcommand{\oti}{\textbf{O2IMEX}}

\newcommand{\rev}[1]{{\color{red}#1\color{black}}}

\newcommand{\jump}[1]{[\![ #1]\!]}

\providecommand{\keywords}[1]
{
	\small	
	\textbf{{Keywords: }} #1
}

\title{Second order divergence constraint preserving schemes for two-fluid relativistic plasma flow equations}
\date{}
\author[a]{Jaya Agnihotri}
\author[b]{Deepak Bhoriya}
\author[a]{Harish Kumar}
\author[c]{Praveen Chandrashekar}
\author[b,d]{Dinshaw S. Balsara}
\affil[a]{Department of Mathematics, Indian Institute of Technology Delhi, India}
\affil[b]{Physics Department, University of Notre Dame, USA}
\affil[c]{Centre for Applicable Mathematics, TIFR, Bangalore, India}
\affil[d]{ACMS, University of Notre Dame, USA}
\begin{document}

	\date{}
	\maketitle
	\begin{abstract}
		Two-fluid relativistic plasma flow equations combine the equations of relativistic hydrodynamics with Maxwell's equations for electromagnetic fields, which involve divergence constraints for the magnetic and electric fields. When developing numerical schemes for the model, the divergence constraints are ignored, or Maxwell's equations are reformulated as Perfectly Hyperbolic Maxwell's (PHM) equations by introducing additional equations for correction potentials. In the latter case, the divergence constraints are preserved only as the limiting case.
		
		In this article, we present second-order numerical schemes that preserve the divergence constraint for electric and magnetic fields at the discrete level. The schemes are based on using a multidimensional Riemann solver at the vertices of the cells to define the numerical fluxes on the edges. The second-order accuracy is obtained by reconstructing the electromagnetic fields at the corners using a MinMod limiter. The discretization of Maxwell's equations can be combined with any consistent and stable discretization of the fluid parts. In particular, we consider entropy-stable schemes for the fluid part. The resulting schemes are second-order accurate, entropy stable, and preserve the divergence constraints of the electromagnetic fields. We use explicit and IMEX-based time discretizations. We then test these schemes using several one- and two-dimensional test cases. We also compare the divergence constraint errors of the proposed schemes with schemes having no divergence constraints treatment and schemes based on the PHM-based divergence cleaning.
		\end{abstract}
\section{Introduction}
In most astrophysical phenomena, fast-moving charged fluid components interact with electromagnetic forces and also generate these forces. To model such flows, the equations of relativistic magnetohydrodynamics (RMHD) are used~\cite{Komissarov1999,Balsara2001,DelZanna2003,Mignone2006,Komissarov2007}. Although highly successful in several applications, these equations consider the fluid as a single fluid and ignore the two-fluid effects, which enforces limitations on the RMHD model where it can be used~\cite{Amano2016, Balsara2016}. Hence, a two-fluid description of the fluid components is desirable. Several authors have considered two-fluid non-relativistic plasma flow models~\cite{Hakim2006,loverich2011,wang2020,kumar2012entropy,Abgrall2014,Meena2019}.  Similarly, for relativistic plasma flows, several authors have considered two-fluid relativistic plasma flow equations~\cite{zenitani2009relativistic,Zenitani2010,Amano2013,Barkov2014,Balsara2016,Bhoriya2023,bhoriya_entropydg_2023}. 

The two-fluid relativistic plasma flow equations are a system of hyperbolic balance laws. It models ion and electron flows via equations of relativistic hydrodynamics (RHD), and electromagnetic fields are modeled using Maxwell's equations. These three parts (two fluid parts and one electromagnetic part) are coupled via nonlinear source terms due to the Lorentz forces. Furthermore, the source terms in Maxwell's equations in the form of total current and total charge are described using the fluid variables. Several authors have proposed numerical methods for the relativistic two-fluid plasma flows~\cite{zenitani2009relativistic,Zenitani2009a,Amano2013,Barkov2014,Amano2016,Balsara2016,Bhoriya2023,bhoriya_entropydg_2023}.  In~\cite{zenitani2009relativistic, Zenitani2009a}, the authors have simulated relativistic magnetic reconnection using the two-fluid model. In~\cite{Amano2013}, the authors study the importance of the electric field in the case of strongly magnetized relativistic plasma. A multidimensional third-order accurate scheme for the model was proposed in~\cite{Barkov2014}. In particular, the issue of electromagnetic divergence constraints was addressed in~\cite{Amano2016}  and~\cite{Balsara2016}. More recently, \cite{Bhoriya2023,bhoriya_entropydg_2023} developed entropy-stable numerical schemes for this model. 

At the continuous level, the divergence-free condition on the magnetic field and the divergence constraint on the electric field, in the form of Guass's law, appear separately. However, they are automatically satisfied if the initial condition has a divergence-free magnetic field and the initial total current density and electric field satisfy Gauss's law. Hence, these two divergence constraints can be viewed as additional constraints on the initial conditions. However, at the discrete level, this is often not the case. Often, to clean the error in satisfying divergence constraints, a perfectly hyperbolic Maxwell's (PHM) equations are considered~\cite{Barkov2014,Bhoriya2023,bhoriya_entropydg_2023}, which do not completely satisfy the constraints. In~\cite{Amano2016}, electric and magnetic fields are evolved on the cell interfaces, resulting in a staggered scheme, which ensures the preservation of both divergence constraints. A higher-order scheme by evolving electric and magnetic fields on the cell interfaces, using the multidimensional Riemann cells at the cell vertices, is presented in~\cite{Balsara2016} while fluid variables are evolved at the cell centers. 

In this article, motivated by the continuous problem, we want to present two-dimensional, co-located, second-order discretizations, which ensure that the divergence constraints are automatically satisfied, provided the initial condition has a discrete divergence-free magnetic field and a discrete Gauss's law is satisfied. A similar approach was taken in \cite{jaya2024}, where divergence constraint preserving schemes for non-relativistic two-fluid ideal plasma flow equations are designed. The key idea is to define a new approximated Riemann solver for Maxwell's equations at the cell faces, using the multidimensional Riemann solver~\cite{Balsara2016,balsara2014,chandrashekar2020} at the cell vertices. To achieve the second-order accuracy of the schemes, the input variables of the multidimensional Riemann solver at the cell vertices are reconstructed using a MinMod limiter along the diagonal directions. The resulting second-order discretization can be coupled with any consistent and stable discretization of the fluid models. In particular, we use entropy stable~\cite{Bhoriya2023} numerical schemes for the fluid models. Finally, we use explicit and Implicit-Explicit (IMEX) time update schemes. We also analyze the effects of these time discretizations on the divergence constraints.

The rest of the article is organized as follows. In the next section, we describe the two-fluid relativistic plasma flow equations.  In Section~\ref{sec:semi_schemes}, following~\cite{Bhoriya2023}, we first recall the semi-discrete entropy stable numerical schemes for the fluid equations. Then, we present an approximate Riemann solver based on the multidimensional Riemann solver at the vertices. To obtain the second-order accuracy, we then describe the MinMod-based reconstruction process. We also present the divergence constraints analysis for the proposed semi-discrete schemes. In Section~\ref{sec:fully_dis}, we present the time discretizations and corresponding divergence constraint analysis. Numerical results in one and two dimensions are then presented in Section~\ref{sec:num_test}, followed by concluding remarks in Section~\ref{sec:con}. 

\section{Two-fluid relativistic plasma flow equations}
\label{sec:model}
Following~\cite{Barkov2014, Bhoriya2023,bhoriya_entropydg_2023}, the fluid part of the two-fluid equations can be written as,
\begin{subequations}\label{eq:TFRHD_fluid}
		\begin{align}
			%
			\frac{\partial D_\alpha}{\partial t} + \nabla \cdot (D_\alpha \mathbf{u}_\al) & = 0, \label{density}
			\\
			%
			\frac{\partial \mathbf{m}_\alpha}{\partial t} + \nabla \cdot (\mathbf{m}_\alpha \mathbf{u}_\al  + p_\alpha \mathbf{I}) & = r_\alpha \Gamma_\alpha \rho_\alpha (\mathbf{E}+\mathbf{u}_\al \times \mathbf{B}), \label{momentum} 
			\\
			\frac{\partial \mathnormal{E}_\al}{\partial t} + \nabla \cdot ((\mathnormal{E}_\alpha+p_\al)\mathbf{u}_\al) & = r_\al \Gamma_\al \rho_\al (\mathbf{u}_\al \cdot \mathbf{E}).	\label{energy}
	\end{align}
\end{subequations}
Here, $D_\al$ refers to mass density,  $\mathbf{m}_\al$ is the momentum density and $\mathnormal{E}_\al$ refers to the total energy density of the species $\al\in\{i,e\}$, denoting ion and electron flows. The flow of each species is modeled via equations of relativistic hydrodynamics (RHD). In particular, \eqref{density} represents the mass conservation, \eqref{momentum} is the equation for the momentum conservation, with the source terms representing the Lorentz forces due to the electromagnetic fields, and~\eqref{energy} is the equation for energy conservation. The conserved fluid quantities, $D_\al, \mathbf{m}_\al= \left(m^x_{\al},m^y_{\al}, m^z_{\al}\right)$, and $E_\al$ are related to the primitive fluid quantities, proper density $\rho_\al$,  velocity $\mathbf{u}_\al=\left(u_{\al}^x,u_{\al}^y, u_{\al}^z\right)$ and pressure $p_\al$, via relations,
	\begin{align*}
		D_\alpha           & = \rho_\alpha \Gamma_\alpha,                \\
		\mathbf{m}_\alpha      & = \rho_\alpha h_\alpha \Gamma^2_\alpha \mathbf{u}_\al, \\
		\mathnormal{E}_\alpha & = \rho_\alpha h_\alpha \Gamma^2_\alpha - p_\alpha.
	\end{align*}
Here $\Gamma_\al$ is the Lorentz factor given by,
	\begin{equation*}
	\Gamma_\al = \dfrac{1}{\sqrt{1 - |\mathbf{u}_\al|^2}}, \qquad \text{ with} \qquad \mathbf{u}_\al = \left(u_{\al}^x,u_{\al}^y, u_{\al}^z\right).
\end{equation*}
where we have assumed the speed of light to be unity. Furthermore, the $h_\al$ is the specific enthalpy, given by ideal equation of state,
\begin{equation*}
	h_\alpha = 1+ \frac{\gamma_\alpha}{\gamma_\alpha - 1} \frac{p_\alpha}{\rho_\alpha} \label{h},
\end{equation*}
where the specific heat ratio is $\gamma_\alpha=c_{p_\alpha}/c_{V_\alpha}$. Consequently, the sound speed $c_\alpha = c_{s_\alpha}$ and the polytropic index $n_\alpha$ can be expressed as
\begin{equation*}
	n_\alpha=k_\alpha-1 \qquad \text{  and  } \qquad c_\alpha^2=\frac{k_\alpha p_\alpha }{n_\alpha \rho_\alpha h_\alpha}, 
\end{equation*}
where $k_\alpha = \dfrac{\gamma_\alpha}{\gamma_\alpha-1}$ are constants. 

The charge to mass ratios are given by  $r_{\alpha}=\dfrac{q_\al}{m_\al}.$ The source terms in momentum and energy equations depend upon the electric $\mathbf{E}=(E_x, E_y, E_z)$ and magnetic $\mathbf{B}=(B_x, B_y, B_z)$ fields. The evolution of these electromagnetic fields is governed by Maxwell's equations, given by,
	\begin{subequations}\label{eq:maxwell_eqn}
	\begin{align}
		\dfrac{\partial \mathbf{B}}{\partial t} + \nabla \times \mathbf{E} &= 0, \label{eq:max_B}\\
		\dfrac{\partial \mathbf{E}}{\partial t} - \nabla \times \mathbf{B} &= -{\mathbf{J}},\label{eq:max_E}\\
		\nabla \cdot \mathbf{B} &= 0,\label{eq:div_B}\\
		\nabla \cdot \mathbf{E} &= {\rho_c }. \label{eq:gauss_E}
	\end{align}
\end{subequations}
The source term of Maxwell's equations~\eqref{eq:maxwell_eqn} contains the total charge density $\rho_c$ and total current density vector $\mathbf{J} = (J_x,J_y,J_z)$, which are given by,
\begin{align}
\rho_c &= \sum_\al{r_\al \rho_\al \Gamma_\al }, \label{eq:charge}\\
\mathbf{J} &= \sum_\al{r_\al \rho_\al \Gamma_\al \mathbf {u}_\al},\label{eq:current}\qquad \al \in \{i,\ e\}.
\end{align}
Ignoring the divergence constraints \eqref{eq:div_B} and \eqref{eq:gauss_E}, the system of equations \eqref{eq:TFRHD_fluid} and \eqref{eq:maxwell_eqn} can be written in the form of balance law:
\begin{equation}
	\frac{\partial \mathbf{U}}{\partial t}+\frac{\partial \mathbf{f}^x}{\partial x} + \frac{\partial \mathbf{f}^y}{\partial y}= \mathcal{S}.
	\label{eq:system}
\end{equation}
The conservative variable vector is given by, 
\[
\mathbf{U}=(\mathbf{U}_i^\top, \mathbf{U}_e^\top, \mathbf{U}_m^\top)^\top, 	
\]
where, $\mathbf{U}_i=(D_i, \mathbf{m}_i, \mathnormal{E}_i)^\top$ is the vector of ion-fluid conservative variables, $\mathbf{U}_e=(D_e, \mathbf{m}_e, \mathnormal{E}_e)^\top$ is the vector of electron-fluid conservative variables,  $\mathbf{U}_m = (\mathbf{B}, \ \mathbf{E})^\top$ consists of Maxwell's equations conservative variables. Here, $\mathcal{S}$ is the source term given by,
\begin{align*}
	\mathcal{S} =
	\begin{pmatrix}
		\mathcal{S}_i (\Ub_i,\Ub_m)\\
		\mathcal{S}_e (\Ub_e,\Ub_m)\\
		\mathcal{S}_m(\Ub_i,\Ub_e)
	\end{pmatrix}
\end{align*}
where the source terms corresponding to the fluid parts are given by,
\begin{align*}
	\mathcal{S}_\al (\Ub_\al,\Ub_m)=
	\begin{pmatrix}
		0\\
		r_\al D_\al (E_x+u_\al^yB_z-u_\al^zB_y)\\
		r_\al D_\al (E_y+u_\al^zB_x-u_\al^xB_z)\\
		r_\al D_\al (E_z+u_\al^xB_y-u_\al^yB_x)\\
		r_\al D_\al (u_\al^xE_x +u_\al^yE_y+u_\al^zE_z)
	\end{pmatrix}.
 \end{align*}
 The Maxwell's source terms are given by,
\begin{align*}
	\mathcal{S}_m (\Ub_i,\Ub_e)= \left(0,0,0,-J_x,-J_y,-J_z\right)^\top.
\end{align*}
The fluxes in the $x$ and $y$directions are denoted by $\mathbf{f}^x$ and $\mathbf{f}^y$, and are given by,
\[
\mathbf{f}^d =\begin{pmatrix}
	\mathbf{f}^d_i(\mathbf{U}_i)\\
	\mathbf{f}^d_e(\mathbf{U}_e)\\
	\mathbf{f}^d_m(\mathbf{U}_m)	
\end{pmatrix}, \qquad d \in\{x,y\}.
\]
Here, $\mathbf{f}^d_i(\mathbf{U}_i)$ is the flux for ions in direction $d$,  $\mathbf{f}^d_e(\mathbf{U}_e)$ is the flux for electrons in $d$ direction and  $\mathbf{f}^d_m(\mathbf{U}_m)$ is the flux for Maxwell's equations in direction $d$. As each of these individual fluxes depends on its own variables (i.e., ion flux depends on only the ion conservative variable, electron flux depends on the electron conservative variable, and Maxwell's flux depends only on electromagnetic fields.), flux for the whole equation has three independent parts. The interaction between these three parts is via source terms only. The fluid fluxes $\mathbf{f}_\al^x$ and $\mathbf{f}_\al^y$ are,
\begin{align}
        \mathbf{f}_\al^x = 
	\begin{pmatrix}
		D_\al u^{x}_{\al}\\
		m^{x}_{\al} u^{x}_{\al} + p_\al\\
		m^{y}_{\al} u^{x}_{\al}\\
		m^{z}_{\al} u^{x}_{\al}\\
		m^{x}_{\al}
	\end{pmatrix},
	\quad
	\mathbf{f}_\al^y = 
	\begin{pmatrix}
		D_\al u^{y}_{\al}\\
		m^{x}_{\al} u^{y}_{\al}\\
		m^{y}_{\al} u^{y}_{\al}+p_\al\\
		m^{z}_{\al} u^{y}_{\al}\\
		m^{y}_{\al}  
	\end{pmatrix}, \qquad \al\in\{i,e\} .
	\label{eq:fluid_flux}
\end{align}
Furthermore, the Maxwell's fluxes in $x$ and $y$ directions are given by,
\begin{align}
	\fb_m^x=\left(0,-E_z,E_y,0,B_z,-B_y\right)^\top \text{ and   }  
	\fb_m^y=\left(E_z,0,-E_x,-B_z,0,B_x\right)^\top
		\label{eq:max_flux}
\end{align}
respectively. The vector of primitive variables is $\Wb=(\Wb_i^\top,\Wb_e^\top,\Wb_m^\top)^\top,$ where $\Wb_\al=(\rho_\al,\ub_\al,p_\alpha)^\top$, and $\Wb_m=(\Bb,\Eb)^\top$. Obtaining the primitive variable $\mathbf{W}$ from the conservative variable $\mathbf{U}$ is not straightforward and requires solving a nonlinear equation; we follow the numerical algorithm given in~\cite{Bhoriya2020, Schneider1993}. 

The physically admissible solutions of this model should lie in the following set
\begin{equation*}
	\Omega = \{ \mathbf{U} \in\mathbb{R}^{16}: \ \rho_i>0, \rho_e>0, \ p_i>0,p_e>0, \ |\mathbf{u}_i|<1, |\mathbf{u}_e| < 1 \}.
\end{equation*}
For solutions in $\Omega$, the equations are hyperbolic with real eigenvalues of the flux Jacobian matrices $\mathnormal{A}^d=\dfrac{\partial \mathbf{f}^d}{\partial \mathbf{U}},$   $d \in \{x,y\}$. The expressions for the eigenvalues are given by,
\begin{equation}
\begin{aligned} 
	\mathbf{\Lambda}{(A^d)} & \equiv \left\{ \left(\mathbf{\Lambda}_{i}\right)_{1\times 5}, \left(\mathbf{\Lambda}_e\right)_{1\times 5}, \left(\mathbf{\Lambda}_m\right)_{1\times 6} \right\} \\
	&= 
	\biggl\{
	 	\frac{(1-c_i^2)u^{d}_i-(c_i/\Gamma_i) \sqrt{Q_i^d}}{1-c_i^2 |\mathbf{u}_i|^2},\
	u^{d}_i, \ u^{d}_i, \ u^{d}_i,
	\frac{(1-c_i^2)u^{d}_i+(c_i/\Gamma_i) \sqrt{Q_i^d}}{1-c_i^2 |\mathbf{u}_i|^2}, 
	\\  
	& \qquad	\frac{(1-c_e^2)u^{d}_e-(c_e/\Gamma_e) \sqrt{Q_e^d}}{1-c_e^2 |\mathbf{u}_e|^2},\
	u^{d}_e, \ u^{d}_e, \ u^{d}_e,
	\frac{(1-c_e^2)u^{d}_e+(c_e/\Gamma_e) \sqrt{Q_e^d}}{1-c_e^2 |\mathbf{u}_e|^2}, 
	\\ 
	& \qquad -1,-1,0,0,1,1
	\biggr\},
\end{aligned} \label{eq:eigenval}
\end{equation}
where $Q_\al^x=1-(u^{x}_{\al})^2-c_\al^2 ((u^{y}_{\al})^2+(u^{z}_{\alpha})^2)$,   $Q_\al^y=1-(u^{y}_{\al})^2-c_\al^2 ((u^{x}_{\al})^2+(u^{z}_{\alpha})^2)$ and $\al \in \{ i, \ e \}$.  
For $\mathbf{U} \in \Omega$ and $\gamma_\al \in (1,2]$ we have $c_\al<1$ which implies that $Q_\al^x > 0$ and $Q_\al^y > 0$; hence all eigenvalues are real.

Following~\cite{Bhoriya2023}, the entropy functions $\eta_\al$ and entropy flux $q_\al^d$ for the two-fluid relativistic plasma flow equations are defined as follows:

\begin{equation*}
	\eta_\al=-\frac{\rho_\al \Gamma_\al s_\al}{\gamma_\al -1} \qquad \text{ and } \qquad q^d_\al=-\frac{\rho_\al \Gamma_\al s_\al u^{d}_{\al}}{\gamma_\al -1}, \qquad  \text{ with } s_\al = \ln(p_\al \rho_\al^{-\gamma_\al}).  
\end{equation*} 
Here, $\al \in \{i,e\},$ and $d\in \{x,y\}$. We also introduce the {\em entropy variables}, 
$
\mathbf{V}_\alpha(\mathbf{U}_\alpha)=\frac{\partial  \eta_\al}{\partial \mathbf{U}_\alpha}$$
$
for $\alpha \in \{i,e\}$. A simple calculation results in the following expressions
\begin{equation}
	\mathbf{V}_\alpha=\begin{pmatrix}
		\dfrac{\gamma_\alpha- s_\alpha}{\gamma_\alpha -1} +{\beta_\alpha} \\
		{u_{x_\alpha} \Gamma_\alpha \beta_\alpha } \\ 
		{u_{y_\alpha} \Gamma_\alpha \beta_\alpha } \\
		{u_{z_\alpha} \Gamma_\alpha \beta_\alpha } \\
		-{\Gamma_\alpha \beta_\alpha}
	\end{pmatrix}, \qquad
	\text{with } \qquad
	\beta_\alpha = \dfrac{\rho_\alpha}{p_\alpha}.
	\label{eq:ent_var} 
\end{equation}
Then, following \cite{Bhoriya2023,bhoriya_entropydg_2023}, for the smooth solutions of \eqref{eq:TFRHD_fluid}, we get the entropy equality,
	\begin{equation}
		\dfrac{\partial \eta_\al}{\partial t} + \dfrac{\partial q^x_\al}{\partial x}  + \dfrac{\partial q^y_\al}{\partial y}=0. \label{entropy_equality}
	\end{equation}
For the non-smooth solutions, the entropy equality~\eqref{entropy_equality} is replaced with the entropy inequality
\begin{equation}
		\dfrac{\partial \eta_\al}{\partial t} + \dfrac{\partial q^x_\al}{\partial x}  + \dfrac{\partial q^y_\al}{\partial y} \le 0. \label{ent_inq}
\end{equation}   

\subsection{Electromagnetic divergence constraints} 
\label{subsec:div_cons}
The full set of Maxwell's equations consists of Faraday's induction equation \eqref{eq:max_B}, governing the evolution of the magnetic field, and  Ampere's law \eqref{eq:max_E}, which describes the evolution of the electric field. In addition, the magnetic field must be evolved in a divergence-free manner~\eqref{eq:div_B} for all time. Although this divergence constraint may appear to be an extra constraint, it can be seen as a constraint on the initial condition. Considering the divergence of the \eqref{eq:max_B}, we get,
\begin{equation}
\na \cdot \frac{\pa \Bb }{\pa t} + \nabla \cdot (\nabla \times \Eb) =  \frac{\pa ( \na \cdot \Bb )}{\pa t} =0.	
\label{eq:divB_evo_cont}
\end{equation} 
This ensures that the magnetic field will be divergence-free for all time provided the initial field is divergence-free. Similarly, for the electric field, the divergence of the Ampere's law \eqref{eq:max_E} results in,
\begin{equation}
	\na \cdot \frac{\pa \Eb }{\pa t} - \nabla \cdot (\nabla \times \Bb) =  \frac{\pa ( \na \cdot \Eb )}{\pa t} = - \na\cdot \Jb .	
\label{eq:divE_evo_cont}
\end{equation} 
Note that 
$$
-\na \cdot \Jb = - (r_i \nabla\cdot(\Gamma_i\rho_i \ub_i) + r_e \nabla\cdot(\Gamma_e \rho_e\ub_e)) =  \left(r_i \frac{\pa (\rho_i \Gamma_i)}{\pa t} + r_e \frac{\pa (\rho_i \Gamma_i)}{\pa t }\right) =\frac{\pa \rho_c}{\pa t }.
$$
This results in
\begin{equation}
	 \frac{\pa (\na \cdot \Eb) }{\pa t} = \frac{\pa \rho_c}{\pa t }.	
\end{equation} 
Hence, if Gauss's law is satisfied initially, it will be satisfied for all time. Thus, the divergence constraints \eqref{eq:div_B} and \eqref{eq:gauss_E} hold for the continuous case. At the discrete level, however, this is not the case as the typical schemes violate these constraints \cite{jiang1996origin}. Ignoring these constraints will also result in non-physical oscillations. A typical approach~\cite{Abgrall2014,kumar2012entropy,bond2016plasma,Meena2019,Bhoriya2023,bhoriya_entropydg_2023,wang2020} to overcome these difficulties is to consider the {\em perfectly hyperbolic Maxwell's equations} (PHM), proposed in~\cite{munz2000},
\begin{subequations}\label{eq:phm_eqn}
	\begin{align}
		\dfrac{\partial \mathbf{B}}{\partial t} + \nabla \times \mathbf{E} +\kappa \nabla \psi &= 0, \label{eq:phm_B}\\
		\dfrac{\partial \mathbf{E}}{\partial t} - \nabla \times \mathbf{B}  + \xi  \nabla \phi&= -{\mathbf{J}},\label{eq:phm_E}\\
		\frac{\partial \psi}{\pa t}+ \kappa \nabla \cdot \mathbf{B} &= 0,\label{eq:phm_div_B}\\
		\frac{\partial \phi}{\pa t} +\xi  \nabla \cdot \mathbf{E} &= \xi {\rho_c }. \label{eq:phm_gauss_E}
	\end{align}
\end{subequations}
Here, $\phi$ and $\psi$ are the {\em correction potentials}, and $\kappa$ and $\xi$ are the divergence errors propagation speed. The divergence constraints are then preserved in the limit when $\kappa$ and $\xi$ tend to infinity. One of the benefits of using the PHM formulation is that the extended system is still hyperbolic. 

In the next section, we aim to design numerical schemes, such that the discrete version of the divergence constraints is satisfied by the discretization of the equations \eqref{eq:max_B} and \eqref{eq:max_E}, similar to the situation for the continuous case.

\section{Semi-discrete  schemes}
\label{sec:semi_schemes}
Consider a two-dimensional rectangular domain  $[a,b]\times[c,d]$. We discretize the domain using a uniform mesh with a cell size of $\Delta x \times \Delta y$. The cell centers are given by $(x_i,y_j)$ with $x_i=a+(i+1/2)\Dx$, $y_j=c+(j+1/2)\Dy$, $0\le i< N_x$ and $0 \le j < N_y$. The cell vertices are defined as $(x_\iph,y_\jph)$, with $x_\iph=\frac{x_i +x_{i+1}}{2}$ and $y_\jph=\frac{y_j +y_{j+1}}{2}$. 

The semi-discrete scheme for the Eqns.~\eqref{eq:TFRHD_fluid}, \eqref{eq:max_B} and \eqref{eq:max_E} has the following form:
\begin{equation}
	\label{eq:semi_discrete}
	\frac{d \Ub_{i,j}}{d t} +\frac{1}{\Dx} \left(\Fb^x_{\iph,j} - \Fb^x_{\imh,j}\right) + \frac{1}{\Dy} \left(\Fb^y_{i,\jph} - \Fb^y_{i,\jmh}\right)  =  \mathcal{S}(\Ub_{i,j}),
\end{equation}  
where $\Fb^x$ and $\Fb^y$ are the numerical fluxes consistent with $\fb^x$ and $\fb^y$, respectively. Following~\cite{Bhoriya2023}, we will first describe the entropy-stable spatial discretization of the fluid equations \eqref{eq:TFRHD_fluid}, which is consistent with the entropy inequality \eqref{ent_inq}.

\subsection{Entropy stable discretization of the fluid equations}
The fluid part of the schemes~\eqref{eq:semi_discrete} can be written as,
\begin{equation}
	\label{eq:semi_discrete_fluid}
	\frac{d \Ub_{\al,i,j}}{d t} +\frac{1}{\Dx} \left(\Fb^x_{\al,\iph,j} - \Fb^x_{\al,\imh,j}\right) + \frac{1}{\Dy} \left(\Fb^y_{\al,i,\jph} - \Fb^y_{\al,i,\jmh}\right)  =  \mathcal{S}_\al(\Ub_{\al,i,j},\Ub_{m,i,j}).
\end{equation}  
where $\Fb_{\alpha}^x$ and $\Fb^y_\al$ are the numerical fluid flux consistent with the continuous fluid fluxes $\fb^x_\alpha$ and $\fb^y_\al$ (see Eqn. \eqref{eq:fluid_flux}), respectively.

Let us introduce the following notations for any grid function $z$,
$$
\overline{z}_{\iph,j} = \frac{z_{i+1,j}+z_{i,j}}{2},\qquad \overline{z}_{i,\jph}=\frac{z_{i,j+1}+z_{i,j}}{2},
$$
 and
 $$
 \jump{z}_{\iph,j} = z_{i+1,j}-z_{i,j},\qquad \jump{z}_{i,\jph}=z_{i,j+1}-z_{i,j}. 
 $$
Then the first-order entropy stable flux is given by,
\begin{equation}
	\begin{aligned}
		\mathbf{F}_{\alpha,i+\frac{1}{2},j}^{x,1} =\tilde{\mathbf{{F}}}_{\alpha,i+\frac{1}{2},j}^x - \frac{1}{2} \textbf{D}_{\alpha,i+\frac{1}{2},j}^x[\![ \mathbf{V}_\alpha]\!]_{i+\frac{1}{2},j},
		\\
		\mathbf{F}_{\alpha,i,j+\frac{1}{2}}^{y,1} = \tilde{\mathbf{{F}}}_{\alpha,i,j+\frac{1}{2}}^y - \frac{1}{2} \textbf{D}_{\alpha,i,j+\frac{1}{2}}^y[\![ \mathbf{V}_\alpha]\!]_{i,j+\frac{1}{2}}.
		\label{es_numflux}
	\end{aligned}
\end{equation}
where  $\tilde{\mathbf{{F}}}_{\alpha,i+\frac{1}{2},j}^x$ and  $\tilde{\mathbf{{F}}}_{\alpha,i,j+\frac{1}{2}}^y$ in Eqn.~\eqref{es_numflux} are entropy conservative numerical fluxes given by~\cite{Bhoriya2020},
\begin{gather}
	\mathbf{\tilde F}^x_{\alpha,i+\frac{1}{2},j}  = \begin{pmatrix}
		{\rho_\alpha^{\ln}} \overline{m_{x_\alpha} } \vspace{0.2cm} \\
		\frac{1}{\overline{\beta_\alpha}} \big( \frac{\overline{\beta_\alpha} \overline{m_{x_\alpha}}}{\bar{\Gamma_\alpha}}F_{\alpha,5}^x + \overline{\rho_\alpha}
		\big)	\\
		\frac{\overline{m_{y_\alpha}}}{\overline{\Gamma_\alpha}}F_{\alpha,5}^x \\ 
		\frac{\overline{m_{z_\alpha}}}{\overline{\Gamma_\alpha}}F_{\alpha,5}^x 
		\vspace{0.2cm} \\ 
		\frac{-\overline{\Gamma_\alpha}{\big( k_{\alpha} {\rho_\alpha^{\ln}}\overline{m_{x_\alpha}}
				+\frac{\overline{m_{x_\alpha}}\overline{\rho_\alpha}}{\overline{\beta_\alpha}}}
			\big)} {\big(    
			\overline{m_{x_\alpha}}^2 
			+ \overline{m_{y_\alpha}}^2 
			+ \overline{m_{z_\alpha}}^2 
			-{(\overline{\Gamma_\alpha})}^2  \big)}
	\end{pmatrix}_{i + \frac{1}{2},j}, \  	
 \mathbf{\tilde F}^y_{\alpha,i,j+\frac{1}{2}} = \begin{pmatrix}
			{\rho_\alpha^{\ln}} \overline{m_{y_\alpha} } \vspace{0.2cm} \\
			%
			%
			\frac{\overline{m_{x_\alpha}}}{\overline{\Gamma_\alpha}}F_{{\alpha,5}}^y	\\
			\frac{1}{\overline{\beta_\alpha}} \bigg( \frac{\bar{\beta_\alpha} \overline{m_{y_\alpha}}}{\overline{\Gamma_\alpha}}F_{{\alpha,5}}^y + \overline{\rho_\alpha}
			\bigg)\\ 
			\frac{\overline{m_{z_\alpha}}}{\overline{\Gamma_\alpha}}F_{{\alpha,5}}^y
			\vspace{0.2cm} \\ 
			\frac{-\overline{\Gamma_\alpha}{\left( k_{\alpha} {\rho_\alpha^{\ln}} \overline{m_{y_\alpha} } +
					\frac{\overline{m_{y_\alpha}} \ \overline{\rho_\alpha}}{\bar{\beta_\alpha}}\right)}
			} {\left(    \overline{m_{x_\alpha}}^2 + \overline{m_{y_\alpha}}^2 + \overline{m_{z_\alpha}}^2 -(\overline{\Gamma_\alpha})^2  \right)}
		\end{pmatrix}_{i, j+\frac{1}{2}},
 \label{eq:num_flux_ent_conser}
\end{gather}
where, $\beta_\alpha=\frac{\rho_\alpha}{p_\alpha}$,  $k_{\alpha}=\left(\frac{1}{\gamma-1} \dfrac{1}{{\beta^{\ln}_\alpha}}+1  \right)$, and $(\cdot)^{\ln}=\frac{[\![ (\cdot) ]\!]}{[\![ \log (\cdot)]\!]}$ is the logarithmic average. 
One can easily verify that $\tilde{{\mathbf F}}_\alpha^x$ is consistent with the flux $\mathbf{f}_{\alpha}^x$ and  $\tilde{{\mathbf F}}_\alpha^y$ is consistent with ${\mathbf{f}}^y_\alpha$, see~\cite{Bhoriya2023} for more details. Furthermore, these numerical fluxes satisfy Tadmor's conditions,
\begin{align}
	\jump{\Vb_\alpha}_{\iph,j} \cdot \mathbf{\tilde F}^x_{\alpha,i+\frac{1}{2},j}=\jump{\phi_\al^x},\\
	\jump{\Vb_\alpha}_{i,\jph} \cdot \mathbf{\tilde F}^y_{\alpha,i,j+\frac{1}{2}}=\jump{\phi_\al^y}
\end{align}
for {\em entropy potentials} $\phi_\al^d = \rho_\al\Gamma_\al u^d_\alpha$.  \rev{Hence, they are entropy conservative as shown in \cite{Fjordholm2012,Tadmor1987} and also second order accurate since the entropy conservative fluxes $\mathbf{\tilde F}^x, \mathbf{\tilde F}^y$ are symmetric. E.g., the flux $\mathbf{\tilde F}^x_{\alpha,i+\frac{1}{2},j}$ depends on $\mathbf U_{\alpha,i,j}$, $\mathbf U_{\alpha,i+1,j}$ and the flux does not change if the states are interchanged.} 

To make the scheme entropy stable, numerical diffusion is added by considering matrices $\textbf{D}^x_{\alpha,i+\frac{1}{2},j}$ and $\textbf{D}^y_{\alpha, i,j+\frac{1}{2}}$, which are symmetric positive definite and are given by
\begin{equation}  \label{diffusiontype}
	\mathbf{D}_{\alpha,i+\frac{1}{2},j}^x = {\mathbf{\tilde{R}}}_{\alpha,i+\frac{1}{2},j}^x \Lambda_{\alpha,i+\frac{1}{2},j}^x {\mathbf{\tilde{R}}}_{\alpha,i+\frac{1}{2},j}^{x \top}, \text{ \ \ and \ \ } \textbf{D}_{\alpha,i,j+\frac{1}{2}}^y = {\mathbf{\tilde{R}}}_{\alpha,i,j+\frac{1}{2}}^y \Lambda_{\alpha,i,j+\frac{1}{2}}^y {\mathbf{\tilde{R}}}_{\alpha,i,j+\frac{1}{2}}^{y \top},
\end{equation}
where, ${\mathbf{\tilde{R}}^d}_\alpha$  are matrices of the entropy scaled right eigenvectors~\cite{Bhoriya2020} and ${\Lambda^d_\alpha} = \{\Lambda_{\alpha_k}^d: 1 \leq k \leq 5 \},\, d \in \{x,y\}$ are $5 \times 5$ non-negative diagonal matrices given by, 
\[
{\Lambda^d_\alpha}=  
\left( \max_{1 \leq k \leq 5} |\Lambda_{\alpha_k}^d|\right) \mathbf{I}_{5 \times 5}. 
\]
where $\Lambda_{\alpha_k}^d$ is the $k$-th eigenvalue of the fluid $\al$ in direction $d$ given in Eqn.~\eqref{eq:eigenval}. The scheme with the numerical flux ~\eqref{es_numflux} is entropy stable but only first-order accurate. To achieve second order accuracy, we reconstruct the jumps $\jump{\Vb_\al}_{\iph,j}$ and $\jump{\Vb_\al}_{i,\jph}$. Following~\cite{Bhoriya2023}, we proceed as follows in $x$-direction:

\begin{steps}
   \item We first define {\em scaled entropy variables}  
\[
\mathcal{W}_{\alpha,{m,j}}^{x,\pm}\,=\, (\mathbf{\tilde{R}}^{x}_{\alpha,{i\pm\frac{1}{2},j}} )^\top\mathbf{V}_{\alpha,{m,j}}, \qquad \textrm{$m$ are neighbours of cell $(i,j)$ along $x$}
\]
\item We perform the second-order MinMod-based reconstruction procedure on $\mathcal{W}_{\alpha,{m,j}}$ to obtain the reconstructed value $\tilde{\mathcal{W}}_{\al,m,j}^{x,\pm}$
\[
\mathcal{\tilde W}_{\alpha,i,j}^{x,\pm} = P_{i,j}^{x,\pm}\left(x_{i \pm \frac{1}{2}}\right).
\]
In general, any second-order {\em sign preserving} reconstruction process can be used.
\item Compute the traces $\tilde{\Vb}^{x,\pm}_{\al,\iph,j}$ as follows:
$$\mathbf{\tilde{V}}_{\alpha,{i,j}}^{x,\pm}\,=\, \left\lbrace (\mathbf{\tilde{R}}^{x}_{\alpha,{i\pm\frac{1}{2},j}})^\top\right\rbrace ^{(-1)}\tilde{\mathcal{W}}_{\alpha,{i,j}}^{x,\pm}$$ 
\end{steps}
Now replacing the jumps $\jump{\Vb_\al}_{\iph,j}$ with 
$$[\![ \mathbf{\tilde{V}}^x_\alpha]\!]_{i+\frac{1}{2},j}\,=\,\mathbf{\tilde{V}}^{x,-}_{\alpha,{i+1,j}}\,-\,\mathbf{\tilde{V}}^{x,+}_{\alpha,{i,j}}.$$
in \eqref{es_numflux}, we get, the second-order entropy stable flux in the $x-$direction
\begin{equation}
	\label{eq:ent_stable_numflux_x_o2}
		\mathbf{F}_{\alpha,i+\frac{1}{2},j}^{x,2} = \tilde{\Fb}^x_{\al,\iph,j}-\frac{1}{2}\Db^x_{\al,\iph,j}\jump{\tilde{\Vb}^x_{\al}}_{\iph,j}
\end{equation}
Here, we only illustrate the procedure for the $x-$direction. We can proceed in the $y$-direction to define,
\begin{equation}
	\label{eq:ent_stable_numflux_y_o2}
		\mathbf{F}_{\alpha,i,j+\frac{1}{2}}^{y,2} = \tilde{\Fb}^y_{\al,i,\jph}-\frac{1}{2}\Db^y_{\al,i,\jph}\jump{\tilde{\Vb}^y_{\al}}_{i,\jph}
\end{equation}

Following \cite{Bhoriya2023}, we now have the following result:
\begin{thm} The numerical scheme~\eqref{eq:semi_discrete_fluid} with the numerical fluxes~\eqref{eq:ent_stable_numflux_x_o2} and~\eqref{eq:ent_stable_numflux_y_o2} is second-order accurate and entropy stable, i.e., it satisfies the discrete version of the entropy inequality~\eqref{ent_inq},
\begin{equation}
	\frac{d}{dt}  \eta_\al(\mathbf{U}_{\alpha,i,j})  +\frac{1}{\Delta x} \left( \hat{q}_{\alpha,i+\frac{1}{2},j}^x - \hat{q}_{\alpha,i-\frac{1}{2},j}^x\right)+\frac{1}{\Delta y}\left( \hat{q}_{\alpha,i,j+\frac{1}{2}}^y - \hat{q}_{\alpha,i,j-\frac{1}{2}}^y\right) \le 0, \ \ \ \alpha \in \{i,e\},
\label{eq:discrete_ent}
\end{equation}
where $\hat{q}^{x}_{\alpha,i+\frac{1}{2},j}$, and $\hat{q}^{y}_{\alpha, i,j+\frac{1}{2}}$ are the numerical entropy flux functions consistent with the continuous entropy fluxes $q^x_\al$, and $q^y_\al$, respectively. 
\end{thm}

\begin{remark}
The above result holds for any consistent second-order discretization of Maxwell's equations.
\end{remark}
This completes our description of the second-order entropy stable scheme for the fluid part. 


\subsection{Multidimensional approximate Riemann solver for Maxwell's equations}
\label{subsec:multid_riem_solver}
The Maxwell parts of the semi-discrete scheme \eqref{eq:semi_discrete} can be written as,
\begin{equation}
	\label{eq:semi_discrete_max}
	\frac{d \Ub_{m,i,j}}{d t} +\frac{1}{\Dx} \left(\Fb^x_{m,\iph,j} - \Fb^x_{m,\imh,j}\right) + \frac{1}{\Dy} \left(\Fb^y_{m,i,\jph} - \Fb^y_{m,i,\jmh}\right)  =  \mathcal{S}_m(\Ub_{i,i,j},\Ub_{e,i,j}).
\end{equation}  
Here $\Fb^x_m$ and $\Fb^y_m$ are the numerical fluxes consistent with $\fb_m^x$ and $\fb^y_m$. Unlike the standard Riemann solver, where only two states (left and right) are considered, here we will use a multidimensional Riemann solver at each vertex and then use these to define fluxes on the faces. Given a vertex $(x_{\iph},y_\jph)$, let us define,
\begin{align*}
	\label{direction_def}
	&\Ub^{SW}_{m,i+\frac{1}{2},j+\frac{1}{2}} =  \Ub_{m,i,j},  &\Ub^{SE}_{m,i+\frac{1}{2},j+\frac{1}{2}} 	=  \Ub_{m,i+1,j},\\
	&\Ub^{NE}_{m,i+\frac{1}{2},j+\frac{1}{2}} =  \Ub_{m,i+1,j+1}, 	&\Ub^{NW}_{m,i+\frac{1}{2},j+\frac{1}{2}} = \Ub_{m,i,j+1}.
\end{align*}
which are the four states around the vertex.
\begin{figure}[!htbp]
	\begin{center}
		\includegraphics[width=0.5\textwidth,clip=]{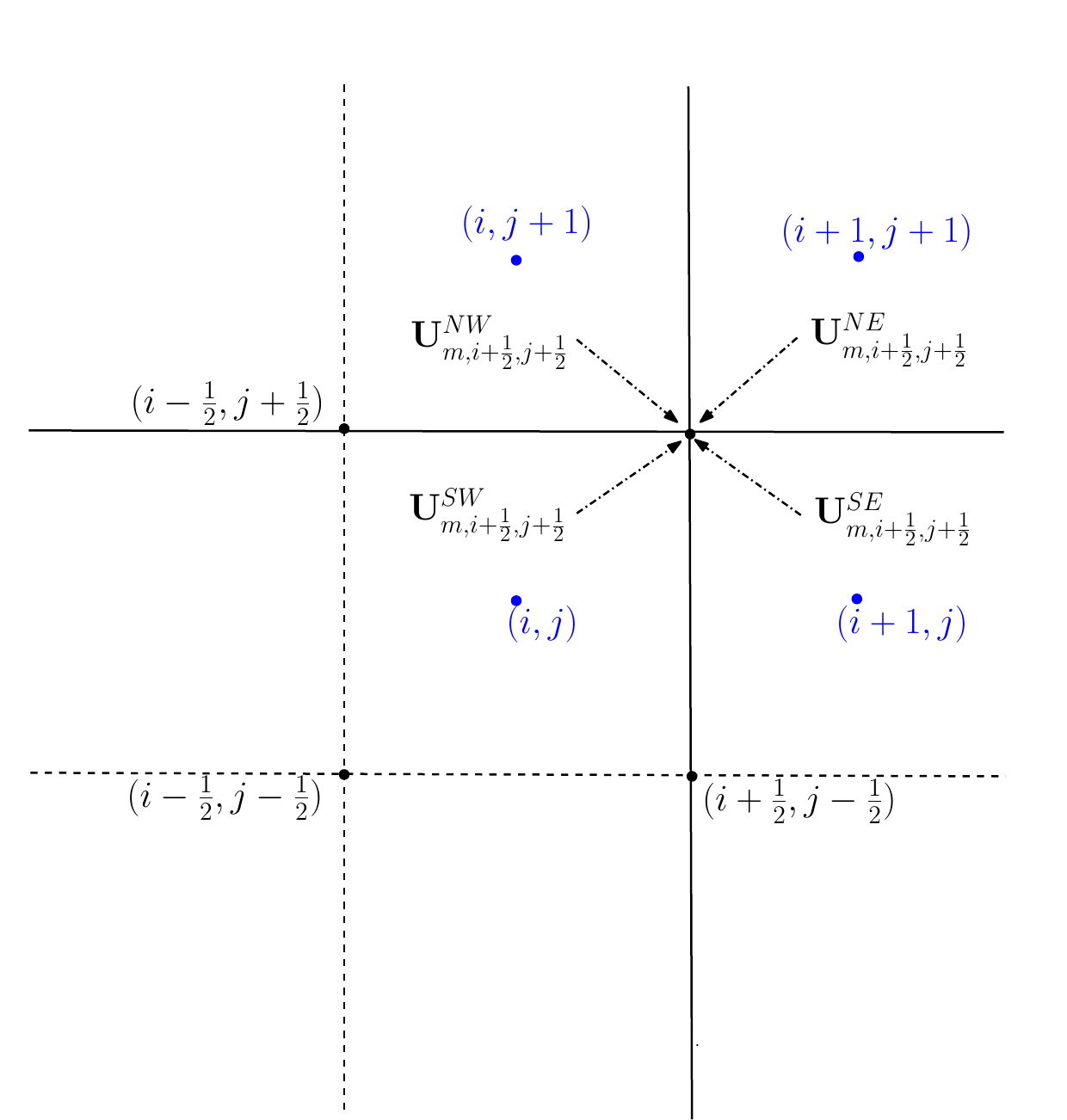}

		\caption{Stencils for multidimensional Riemann solver.}
	\label{fig:grid_tfrhd}
	\end{center}
\end{figure}
Using these, we define the vertex values (See Figure~\ref{fig:grid_tfrhd}),
and then use these vertex values to define the flux across the edges.  Let us first define the following notations:
\begin{equation}\label{eq:max_four_face_averages}
	\overline{{\Ub}}_{m,\iph,\jph}=\frac{\Ub^{SW}_{m,i+\frac{1}{2},j+\frac{1}{2}}+\Ub^{SE}_{m,i+\frac{1}{2},j+\frac{1}{2}}  +\Ub^{NE}_{m,i+\frac{1}{2},j+\frac{1}{2}}+ \Ub^{NW}_{m,i+\frac{1}{2},j+\frac{1}{2}}}{4},
\end{equation}
\begin{equation}
	\label{eq:max_face_averages}
	\overline{{\Ub}}_{m,\iph,j}=\frac{\Ub^{SW}_{m,i+\frac{1}{2},j+\frac{1}{2}} + \Ub^{SE}_{m,i+\frac{1}{2},j+\frac{1}{2}} }{2},\quad	\overline{{\Ub}}_{m,i,\jph} = \frac{\Ub^{NW}_{m,i+\frac{1}{2},j+\frac{1}{2}}  +\Ub^{SW}_{m,i+\frac{1}{2},j+\frac{1}{2}} }{2}.
\end{equation}
Following \cite{balsara2014,balsara2010,Balsara2016aderweno,chandrashekar2020}, the multidimensional local Lax-Friedrich scheme for the vertex values of ${E}_z$ and ${B}_z$ are defined as,
\begin{equation}\label{eq:multid_E}
	\tilde{E}_{z,\iph,\jph}= \bar{E}_{z,\iph,\jph} +\frac{1}{2}\left( (\bar{B}_{y,i+1,\jph}-\bar{B}_{y,i,\jph})  - (\bar{B}_{x,\iph,j+1}-\bar{B}_{x,\iph,j})\right), 
\end{equation}
and
\begin{equation}\label{eq:multid_B}
	\tilde{B}_{z,\iph,\jph}= \bar{B}_{z,\iph,\jph}-\frac{1}{2}\left((\bar{E}_{x,\iph,j+1}-\bar{E}_{x,\iph,j} )	-   (\bar{E}_{y,i+1,\jph}-\bar{E}_{y,i,\jph} )\right).
\end{equation}
Using these values, the $x$-directional flux on the edge with vertices $(x_{\iph},y_\jph)$ and $(x_\iph,y_\jmh)$ is defined as
\begin{equation}
	\label{eq:multiD_flux_x}
	\Fb_{m,\iph,j}^x= \begin{pmatrix}
		0\\
		-\dfrac{1}{2} \left(\tilde{E}_{z,\iph,\jph} +\tilde{E}_{z,\iph,\jmh} \right)  \\
		\tilde{F}_{m,\iph,j}^{x,3}  \\
		0\\
		\dfrac{1}{2}\left(\tilde{B}_{z,\iph,\jph} +\tilde{B}_{z,\iph,\jmh}\right)\\
		{\tilde{F}}_{m,\iph,j}^{x,6}
	\end{pmatrix}
\end{equation}
where, $\tilde{F}_{m,\iph,j}^{x,3}$ and $\tilde{F}_{m,\iph,j}^{x,6}$ are obtained using one-dimensional Rusanov's solver as follows:
\begin{align}
   {F}_{m,\iph,j}^{x,3}  & =\left(  \dfrac{{E}_{y,i,j} + {E}_{y,i+1,j}}{2}\right)
   - 
   \dfrac{1}{2} \left( {B}_{z,i+1,j} - {B}_{z,i,j} \right),
   \label{eq:oned_rus_x_Bz}
   \\
    \tilde{F}_{m,\iph,j}^{x,6} &=
   -\left(
   \dfrac{{B}_{y,i,j} + {B}_{y,i+1,j}}{2}
   \right)
   - 
   \dfrac{1}{2} \left( {E}_{z,i+1,j} - {E}_{z,i,j} 
   \right).
   \label{eq:oned_rus_x_Ez}
\end{align}
We note that the extreme eigenvalues of Maxwell's part are $1$ and $-1$. Hence, the numerical flux presented here is equivalent to multidimensional HLL numerical flux at each vertex. Furthermore, following \cite{chandrashekar2020, Balsara2016}, we comment that the numerical flux \eqref{eq:multiD_flux_x} is consistent with the one-dimensional numerical flux.

Similarly, the $y$-directional numerical flux  $\Fb_{m,i,\jph}^y$ is given by, 
\begin{equation}
	\label{eq:multiD_flux_y}
\Fb_{m,i,\jph}^y=\begin{pmatrix}
	 \dfrac{1}{2} \left(\tilde{E}_{z,\iph,\jph} +\tilde{E}_{z,\imh,\jph} \right)\\
	0
    \\
	\tilde{F}_{m,i,\jph}^{y,3}\\
	-\dfrac{1}{2} \left(\tilde{B}_{z,\iph,\jph} +\tilde{B}_{z,\imh,\jph} \right)\\
	0\\
	\tilde{F}_{m,i,\jph}^{y,6}
\end{pmatrix}.	
\end{equation}
where, $\tilde{F}_{m,i,\jph}^{y,3}$ and $\tilde{F}_{m,i,\jph}^{y,6}$ are obtained using one-dimensional Rusanov's solver as follows:
\begin{align}
   \tilde{F}_{m,i,\jph}^{y,3} &= 
   - \left(
   \dfrac{{E}_{x,i,j} + {E}_{x,i,j+1}}{2}
   \right)
   - 
   \dfrac{1}{2} \left( {B}_{z,i,j+1} - {B}_{z,i,j} \right),
   \label{eq:oned_rus_y_Bz}
   \\
 \tilde{F}_{m,i,\jph}^{y,6} &= 
    \left(
   \dfrac{{B}_{x,i,j} + {B}_{x,i,j+1}}{2}
   \right)
   - 
   \dfrac{1}{2} \left( {E}_{z,i,j+1} - {E}_{z,i,j} \right).
   \label{eq:oned_rus_y_Ez}
\end{align}

\subsection{Second-order reconstruction for Maxwell's equations}
\label{subsec:max_2nd_order}
To achieve second-order accuracy, we need to reconstruct the values at each vertex $(x_\iph,y_\jph)$ using piecewise linear polynomials. We achieve this by performing the reconstruction diagonally (see Figure \ref{fig:o2_recon}), i.e., we define,
\begin{align*}
	\hat{\Ub}^{SW}_{m,i+\frac{1}{2},j+\frac{1}{2}} 
	&=  
	\Ub_{m,i,j} + \frac{1}{2} \text{MinMod}
	\Big(
	\Ub_{m,i-1,j-1}, \Ub_{m,i,j}, \Ub_{m,i+1,j+1} 
	\Big),
	\\
	\hat{\Ub}^{SE}_{m,i+\frac{1}{2},j+\frac{1}{2}} 
	&=  
	\Ub_{m,i+1,j} - \frac{1}{2} \text{MinMod}
	\Big(
	\Ub_{m,i+2,j-1},\Ub_{m,i+1,j}, \Ub_{m,i,j+1},  
	\Big),
	\\
	\hat{\Ub}^{NE}_{m,i+\frac{1}{2},j+\frac{1}{2}} 
	&=  
	\Ub_{m,i+1,j+1} - \frac{1}{2} \text{MinMod}
	\Big(
	\Ub_{m,i+2,j+2},\Ub_{m,i+1,j+1}, \Ub_{m,i,j} 
	\Big),
	\\
	\hat{\Ub}^{NW}_{m,i+\frac{1}{2},j+\frac{1}{2}} 
	&=  
	\Ub_{m,i,j+1} + \frac{1}{2} \text{MinMod}
	\Big(
	\Ub_{m,i-1,j+2},\Ub_{m,i,j+1}, \Ub_{m,i+1,j}  
	\Big).
\end{align*}\label{hat_def}
where, 
\begin{align}
   \text{MinMod} (a,b,c) = \begin{cases}
       \text{sign}(b-a)   \min \{|c-b|,|b-a|,\} & \text{if} \quad \text{sign}(b-a) = \text{sign}(c-b),
       \\
       0, & \text{otherwise.}
   \end{cases}
   \label{eq:minmod}
\end{align}
\begin{figure}[ht]
\begin{center}
	\includegraphics[width=0.5\textwidth]{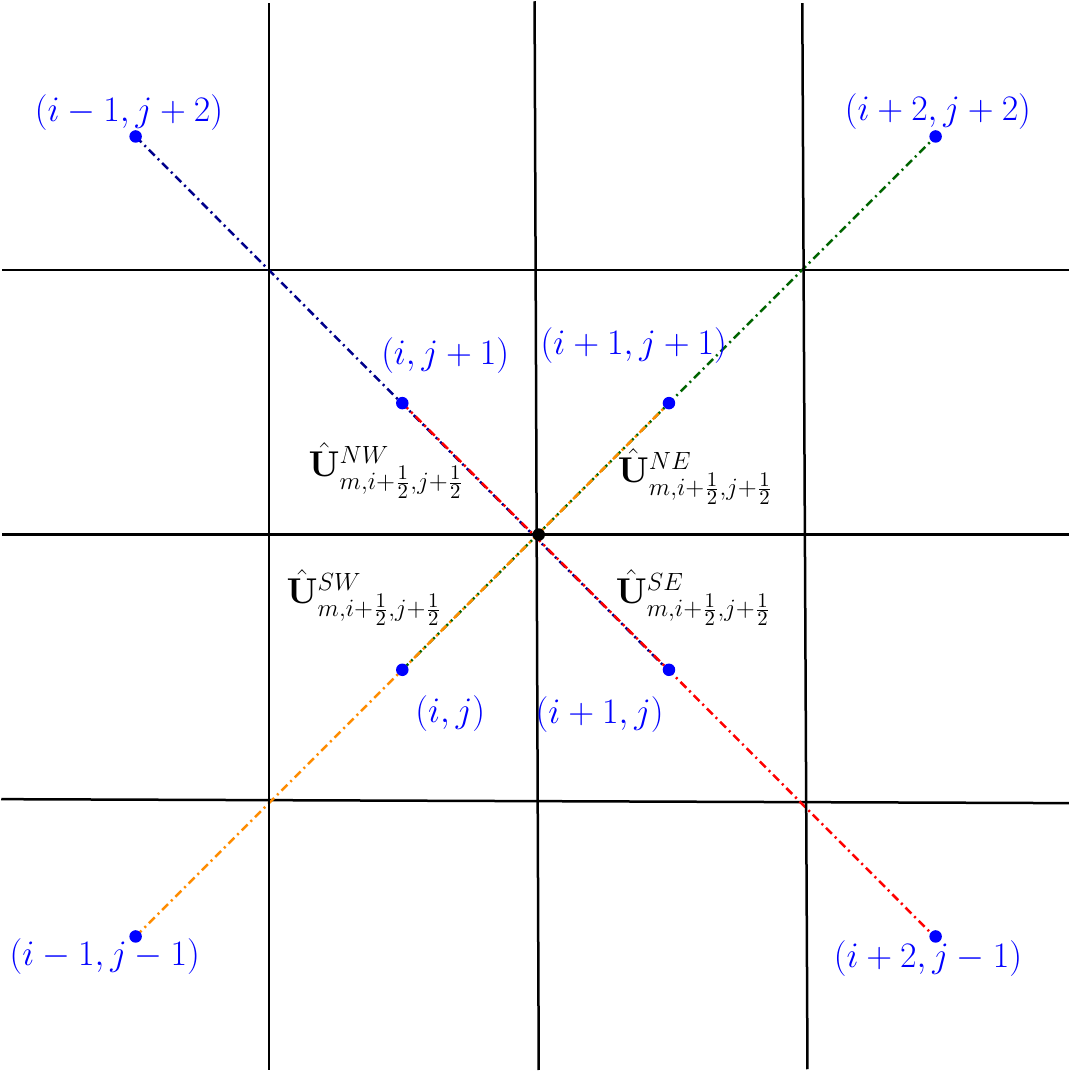}
	\caption{Stencil of second order reconstruction at the vertex $\left( i+\frac{1}{2}, j+\frac{1}{2} \right)$.}
	\label{fig:o2_recon}
\end{center}
\end{figure}
Combining  these diagonal traces with multidimensional Riemann solver based values~\eqref{eq:multid_B} and~\eqref{eq:multid_E}, we define the vertex values $\tilde{E}_{z,\iph,\jph}$ 
\begin{equation}
\begin{aligned}
    \tilde{E}_{z,\iph,\jph}
      & = \dfrac{
        \hat{E}^{SW}_{z,i+\frac{1}{2},j+\frac{1}{2}}
      + \hat{E}^{SE}_{z,i+\frac{1}{2},j+\frac{1}{2}}
      + \hat{E}^{NE}_{z,i+\frac{1}{2},j+\frac{1}{2}}
      + \hat{E}^{NW}_{z,i+\frac{1}{2},j+\frac{1}{2}}
        }{4}
      \\ 
      &+ 
      \frac{1}{2}
         \left(
            \dfrac{\hat{B}^{SE}_{y,i+\frac{1}{2},j+\frac{1}{2}}
                +  \hat{B}^{NE}_{y,i+\frac{1}{2},j+\frac{1}{2}}}{2}
                -
            \dfrac{\hat{B}^{SW}_{y,i+\frac{1}{2},j+\frac{1}{2}}
                +  \hat{B}^{NW}_{y,i+\frac{1}{2},j+\frac{1}{2}}}{2}  
         \right)
      \\ 
      &- 
      \frac{1}{2}
         \left(
            \dfrac{\hat{B}^{NW}_{x,i+\frac{1}{2},j+\frac{1}{2}}
                +  \hat{B}^{NE}_{x,i+\frac{1}{2},j+\frac{1}{2}}}{2}
                -
            \dfrac{\hat{B}^{SW}_{x,i+\frac{1}{2},j+\frac{1}{2}}
                +  \hat{B}^{SE}_{x,i+\frac{1}{2},j+\frac{1}{2}}}{2}  
         \right),	
\label{eq:multid_Ez_2nd}
\end{aligned}
\end{equation}
and $\tilde{B}_{z,\iph,\jph}$,
\begin{equation}
\begin{aligned}
    \tilde{B}_{z,\iph,\jph}
      & = \dfrac{1}{4} \left( 
        \hat{B}^{SW}_{z,i+\frac{1}{2},j+\frac{1}{2}}
      + \hat{B}^{SE}_{z,i+\frac{1}{2},j+\frac{1}{2}}
      + \hat{B}^{NE}_{z,i+\frac{1}{2},j+\frac{1}{2}}
      + \hat{B}^{NW}_{z,i+\frac{1}{2},j+\frac{1}{2}} \right)
      \\ 
      & -
      \frac{1}{2}
         \left(
            \dfrac{\hat{E}^{NW}_{x,i+\frac{1}{2},j+\frac{1}{2}}
                +  \hat{E}^{NE}_{x,i+\frac{1}{2},j+\frac{1}{2}}}{2}
                -
            \dfrac{\hat{E}^{SW}_{x,i+\frac{1}{2},j+\frac{1}{2}}
                +  \hat{E}^{SE}_{x,i+\frac{1}{2},j+\frac{1}{2}}}{2}  
         \right)
      \\ 
      & + 
      \frac{1}{2}
         \left(
            \dfrac{\hat{E}^{SE}_{y,i+\frac{1}{2},j+\frac{1}{2}}
                +  \hat{E}^{NE}_{y,i+\frac{1}{2},j+\frac{1}{2}}}{2}
                -
            \dfrac{\hat{E}^{SW}_{y,i+\frac{1}{2},j+\frac{1}{2}}
                +  \hat{E}^{NW}_{y,i+\frac{1}{2},j+\frac{1}{2}}}{2}  
         \right).	
\label{eq:multid_Bz_2nd}
\end{aligned}
\end{equation}
To compute second-order accurate $\tilde{F}_{\iph,j}^{x,3} $ and $\tilde{F}_{\iph,j}^{x,6}$, we use standard {\em MinMod} limiter in $x$-direction and one-dimensional Rusanov's solver. We first compute the traces in $x$- direction,
\begin{align}
   {\Ub}^{-}_{m,\iph,j}
   &=
   \Ub_{m,i,j} + \dfrac{1}{2} 
   \text{MinMod} \Big(
     \Ub_{m,i-1,j},\Ub_{m,i,j}, \Ub_{m,i+1,j} 
    \Big),
    \\
    \Ub^{+}_{m,\iph,j}
    &=
   \Ub_{m,i+1,j} - \dfrac{1}{2} 
   \text{MinMod} \Big(
    \Ub_{m,i,j},\Ub_{m,i+1,j},  \Ub_{m,i+2,j}
    \Big).
\end{align}
Now using one-dimensional Rusanov's solver \eqref{eq:oned_rus_x_Bz} and \eqref{eq:oned_rus_x_Ez}, to define, 
\begin{align}
\tilde{F}_{m,\iph,j}^{x,3}  &= \dfrac{1}{2} \left({E}^{-}_{y,\iph,j} + {E}^+_{y,\iph,j}\right)
   - 
   \dfrac{1}{2} \left( {B}^+_{z,\iph,j} - {B}^-_{z,\iph,j} \right),
   \label{eq:oned_rus_x_Bz_2nd}
   \\
   \tilde{F}_{m,\iph,j}^{x,6}  &=  -\dfrac{1}{2}
   \left( {B}^-_{y,\iph,j} + {B}^+_{y,\iph,j} \right)
   - 
   \dfrac{1}{2} \left( {E}^+_{z,\iph,j} - {E}^-_{z,\iph,j} \right).
   \label{eq:oned_rus_x_Ez_2nd}
\end{align}
Similarly, we can proceed in $y$-direction. This completes the second-order description of the flux terms.
%
\subsection{Discrete divergence constraints}
\label{subsec:semi_div_cons}
To discuss the evolution of the discrete divergence constraints, let us first define the discrete 2-D divergence of a grid vector function $\Ab=(A_x, A_y, A_z)$ at vertex $(x_\iph,y_\jph)$ as,

	\begin{align*}
	\left(\nabla\cdot\Ab\right)_{i+\frac{1}{2},j+\frac{1}{2}} 
	& = 
	\left(\partial_x A_x\right)_{i+\frac{1}{2},j+\frac{1}{2}} + \left(\partial_y A_y\right)_{i+\frac{1}{2},j+\frac{1}{2}} \\
	& =
	\dfrac{1}{2} 
	\Bigg(
	\frac{{A}_{x,i+1,j+1}-{A}_{x,i,j+1} }{\Delta x}
	+
	\frac{{A}_{x,i+1,j}-{A}_{x,i,j} }{\Delta x}
	\Bigg)\\
	& +
	\dfrac{1}{2} 
	\Bigg(
	\frac{{A}_{y,i+1,j+1}-{A}_{y,i+1,j} }{\Delta y}
	+
	\frac{{A}_{y,i,j+1}-{A}_{y,i,j} }{\Delta y}
	\Bigg). 
\end{align*}
We now have the following result.

\begin{thm}[Evolution of the divergence of magnetic field]
	\label{thm:divB}
	 The semi-discrete scheme~\eqref{eq:semi_discrete} with Maxwell's fluxes \eqref{eq:multiD_flux_x} and \eqref{eq:multiD_flux_y} and second order reconstruction given in Section \ref{subsec:max_2nd_order}, satisfies,
	\begin{equation}
		\label{eq:exp_divB_error}
			\dfrac{d}{d t} \left(\na \cdot \Bb\right)_{\iph,\jph} = 0.
	\end{equation}
	for the magnetic field $\Bb$ which is consistent with \eqref{eq:divB_evo_cont}. 
\end{thm}
\begin{proof}

The evolution of the magnetic field components $B_x$ and $B_y$ can now be written as,
	$$
	\dfrac{d B_{x,i,j}}{d t} =	 
	-  \dfrac{1}{2\Delta y } 
	\bigg[
	\left(\tilde{E}_{z,i+\frac{1}{2},j+\frac{1}{2}} + \tilde{E}_{z,i-\frac{1}{2},j+\frac{1}{2}}\right)
	-\left(\tilde{E}_{z,i+\frac{1}{2},j-\frac{1}{2}} + \tilde{E}_{z,i-\frac{1}{2},j-\frac{1}{2}}\right)
	\bigg]
	$$
and
$$
\dfrac{dB_{y,i,j}}{d t} = \dfrac{1}{2\Delta x } 
	\bigg[
	\left(\tilde{E} _{z,i+\frac{1}{2},j+\frac{1}{2}} + \tilde{E} _{z,i+\frac{1}{2},j-\frac{1}{2}}\right)
	-\left(\tilde{E} _{z,i-\frac{1}{2},j+\frac{1}{2}} + \tilde{E} _{z,i-\frac{1}{2},j-\frac{1}{2}}\right)
	\bigg]
$$
then, the evolution of the discrete divergence is given by,
		\begin{align*}
			\frac{d}{dt}\left(\na \cdot \Bb\right)_{\iph,\jph}
			= &
			\dfrac{1}{2} 
			\Bigg(
			\frac{ \frac{d ({B}_{x,i+1,j+1})}{dt}-\frac{d ({B}_{x,i,j+1})}{dt}  }{\Delta x}
			+
			\frac{\frac{d ({B}_{x,i+1,j})}{dt} -\frac{d ({B}_{x,i,j})}{dt} }{\Delta x}
			\Bigg)
			\\ & \qquad +
			\dfrac{1}{2} 
			\Bigg(
			\frac{\frac{d ({B}_{y,i+1,j+1})}{dt}-\frac{d ({B}_{y,i+1,j})}{dt} }{\Delta y}
			+
			\frac{\frac{d ({B}_{y,i,j+1})}{dt} -\frac{d ({B}_{y,i,j})}{dt} }{\Delta y}
			\Bigg) \\
			 = &
			-\dfrac{1}{4\Delta x \Delta y} 
			\bigg[
			\left(\tilde{E}_{z,i+\frac{3}{2},j+\frac{3}{2}} + \tilde{E}_{z,i+\frac{1}{2},j+\frac{3}{2}}\right)
			-\left(\tilde{E}_{z,i+\frac{3}{2},j+\frac{1}{2}} + \tilde{E}_{z,i+\frac{1}{2},j+\frac{1}{2}}\right)
			\bigg]\\
			& 
			+\dfrac{1}{4\Delta x \Delta y}
			\bigg[
			\left(\tilde{E}_{z,i+\frac{1}{2},j+\frac{3}{2}} + \tilde{E}_{z,i-\frac{1}{2},j+\frac{3}{2}}\right)
			-\left(\tilde{E}_{z,i+\frac{1}{2},j+\frac{1}{2}} + \tilde{E}_{z,i-\frac{1}{2},j+\frac{1}{2}}\right)
			\bigg]\\
			&-\dfrac{1}{4\Delta x \Delta y} 
			\bigg[
			\left(\tilde{E}_{z,i+\frac{3}{2},j+\frac{1}{2}} + \tilde{E}_{z,i+\frac{1}{2},j+\frac{1}{2}}\right)
			-\left(\tilde{E}_{z,i+\frac{3}{2},j-\frac{1}{2}} + \tilde{E}_{z,i+\frac{1}{2},j-\frac{1}{2}}\right)
			\bigg]\\
			&+\dfrac{1}{4\Delta x \Delta y}
			\bigg[
			\left(\tilde{E}_{z,i+\frac{1}{2},j+\frac{1}{2}} + \tilde{E}_{z,i-\frac{1}{2},j+\frac{1}{2}}\right)
			-\left(\tilde{E}_{z,i+\frac{1}{2},j-\frac{1}{2}} + \tilde{E}_{z,i-\frac{1}{2},j-\frac{1}{2}}\right)
			\bigg]\\
			&+\dfrac{1}{4\Delta y\Delta x}
			\bigg[
			\left(\tilde{E}_{z,i+\frac{3}{2},j+\frac{3}{2}} + \tilde{E}_{z,i+\frac{3}{2},j+\frac{1}{2}}\right)
			-\left(\tilde{E}_{z,i+\frac{1}{2},j+\frac{3}{2}} + \tilde{E}_{z,i+\frac{1}{2},j+\frac{1}{2}}\right)
			\bigg]\\
			&-\dfrac{1}{4\Delta y\Delta x}
			\bigg[
			\left(\tilde{E}_{z,i+\frac{3}{2},j+\frac{1}{2}} + \tilde{E}_{z,i+\frac{3}{2},j-\frac{1}{2}}\right)
			-\left(\tilde{E}_{z,i+\frac{1}{2},j+\frac{1}{2}} + \tilde{E}_{z,i+\frac{1}{2},j-\frac{1}{2}}\right)
			\bigg]\\
			&+\dfrac{1}{4\Delta y\Delta x} 
			\bigg[
			\left(\tilde{E}_{z,i+\frac{1}{2},j+\frac{3}{2}} + \tilde{E}_{z,i+\frac{1}{2},j+\frac{1}{2}}\right)
			-\left(\tilde{E}_{z,i-\frac{1}{2},j+\frac{3}{2}} + \tilde{E}_{z,i-\frac{1}{2},j+\frac{1}{2}}\right)
			\bigg]\\
			&-\dfrac{1}{2\Delta y\Delta x} 
			\bigg[
			\left(\tilde{E}_{z,i+\frac{1}{2},j+\frac{1}{2}} + \tilde{E}_{z,i+\frac{1}{2},j-\frac{1}{2}}\right)
			-\left(\tilde{E}_{z,i-\frac{1}{2},j+\frac{1}{2}} + \tilde{E}_{z,i-\frac{1}{2},j-\frac{1}{2}}\right)
			\bigg]\\
   = & \ 0
				\end{align*}
		with the right hand side becoming zero after rearranging and cancelling all the terms. Hence the proposed schemes will ensure the discrete divergence free evolution of the magnetic field if the initial magnetic field is  discretely divergence free. Note that this cancellation was possible because of the use of the multi-dimensional Riemann solver used to define the fluxes.
\end{proof}
Similar to the magnetic field, we have the following result for the electric field.
\begin{thm}[Evolution of the divergence of electric field]
		\label{thm:divE}
	The semi-discrete scheme~\eqref{eq:semi_discrete} with Maxwell's fluxes \eqref{eq:multiD_flux_x} and \eqref{eq:multiD_flux_y} and second order reconstruction given in Section \ref{subsec:max_2nd_order}, satisfies,
	\begin{equation}
		\label{eq:exp_divE_error}
	\dfrac{d}{d t}\left(\na \cdot \Eb\right)_{\iph,\jph} =  - \left( \na \cdot \Jb\right)_{\iph,\jph},
	\end{equation}
	for the electric field $\Eb$ which is consistent with \eqref{eq:divE_evo_cont}. 
\end{thm}
\begin{proof}
The evolution of the first two components of the electric field can be written as,

	\begin{subequations} \label{E_div_update}
		\begin{align}
			\dfrac{d}{d t}E_{x,i,j} &= - \dfrac{1}{2\Delta y } 
	\bigg[
	\left(\tilde{B}_{z,i+\frac{1}{2},j+\frac{1}{2}} + \tilde{B}_{z,i-\frac{1}{2},j+\frac{1}{2}}\right)
	-\left(\tilde{B}_{z,i+\frac{1}{2},j-\frac{1}{2}} + \tilde{B}_{z,i-\frac{1}{2},j-\frac{1}{2}}\right)
	\bigg] - ( J_x)_{i,j},  \\
			\dfrac{d}{d t}E_{y,i,j} &=  
			  \dfrac{1}{2\Delta x } 
	\bigg[
	\left(\tilde{B} _{z,i+\frac{1}{2},j+\frac{1}{2}} + \tilde{B} _{z,i+\frac{1}{2},j-\frac{1}{2}}\right)
	-\left(\tilde{B} _{z,i-\frac{1}{2},j+\frac{1}{2}} + \tilde{B} _{z,i-\frac{1}{2},j-\frac{1}{2}}\right)
	\bigg]- ( J_y)_{i,j}. 
		\end{align}
  \end{subequations}
Then,

	\begin{align*}
	\frac{d}{dt}\left(\na \cdot \Eb\right)_{\iph,\jph}
	& =
	\dfrac{1}{2} 
	\Bigg(
	\frac{ \frac{d ({E}_{x,i+1,j+1})}{dt}-\frac{d ({E}_{x,i,j+1})}{dt}  }{\Delta x}
	+
	\frac{\frac{d ({E}_{x,i+1,j})}{dt} -\frac{d ({E}_{x,i,j})}{dt} }{\Delta x}
	\Bigg)
	\\ & \qquad +
	\dfrac{1}{2} 
	\Bigg(
	\frac{\frac{d ({E}_{y,i+1,j+1})}{dt}-\frac{d ({E}_{y,i+1,j})}{dt} }{\Delta y}
	+
	\frac{\frac{d ({E}_{y,i,j+1})}{dt} -\frac{d ({E}_{y,i,j})}{dt} }{\Delta y}
	\Bigg)\\
	& =
	-\dfrac{1}{2} 
	\Bigg(
	\frac{ {J}_{x,i+1,j+1}-{J}_{x,i,j+1} }{\Delta x}
	+
	\frac{{J}_{x,i+1,j} -{J}_{x,i,j} }{\Delta x}
	\Bigg)
	\\ & \qquad -
	\dfrac{1}{2} 
	\Bigg(
	\frac{{J}_{y,i+1,j+1}- {J}_{y,i+1,j}}{\Delta y}
	+
	\frac{{J}_{y,i,j+1} -{J}_{y,i,j} }{\Delta y}
	\Bigg)	
\end{align*}
as the flux terms cancelled out, similar to the magnetic field case, leading to \eqref{eq:exp_divE_error}.
\end{proof}
\section{Fully discrete numerical schemes}
\label{sec:fully_dis}

The semi-discrete scheme \eqref{eq:semi_discrete} described in the previous section can be rewritten as, 
\begin{equation}
	\frac{d \Ub_{i,j}}{dt} = \mathcal{L}_{i,j}(\Ub(t)) + \mathcal{S}(\Ub_{i,j}(t))
\end{equation}
where
$$
\mathcal{L}_{i,j}(\Ub(t))=-\frac{1}{\Dx} \left(\Fb^x_{\iph,j} - \Fb^x_{\imh,j}\right) - \frac{1}{\Dy} \left(\Fb^y_{i,\jph} - \Fb^y_{i,\jmh}\right)  
$$
We now describe a second-order explicit and IMEX scheme and analyze the corresponding evolution of the divergence constraints.
\subsection{Explicit scheme}
\label{subsec:exp}
We consider the second-order strong stability preserving Runge-Kutta explicit method (see  \cite{Gottlieb2001}). We want to evolve the solution $\mathbf{U}^n$ at the $t^n$ and  to $\mathbf{U}^{n+1}$, the solution at time $t^{n+1} = t^n + \Delta t$. We proceed as follows:
\begin{enumerate}
	\item Set $\Ub^{(0)}=\Ub^n$.
	\item Compute the first step as follows:
	\begin{equation}
		\label{eq:exp_1st_step}
		\Ub^{(1)}_{i,j}=\Ub^{(0)}_{i,j} + \Dt \ \mathcal{L}_{i,j}\left(\Ub^{(0)}\right) +\Dt \ \mathcal{S}\left(\Ub_{i,j}^{(0)}\right).
	\end{equation}
	
	\item 
	Compute the second step as follows:
\begin{equation}
		\label{eq:exp_u2}
			\Ub^{(2)}_{i,j}=\Ub^{(1)}_{i,j}+\Dt \ \mathcal{L}_{i,j}\left(\Ub^{(1)}\right) +\Dt \ \mathcal{S}\left(\Ub_{i,j}^{(1)}\right).
\end{equation}	
	Now define the updated solution as,
		\begin{equation}
		\label{eq:exp_2nd_step}
	\Ub^{n+1}_{i,j}=\frac{1}{2}\left(\Ub^{(0)}_{i,j} + \Ub^{(2)}_{i,j}\right).
	\end{equation}
\end{enumerate}
The corresponding scheme is denoted by {\bf O2EXP-MultiD}. The evolution of the divergence constraint using the above scheme is given by the following result.
\begin{prop}[Evolution of the divergence constraint for the explicit scheme {\bf O2EXP-MultiD}]
	\label{prop:exp}
    The explicit scheme update \eqref{eq:exp_2nd_step} for the semi-discrete scheme  \eqref{eq:semi_discrete} satisfies,
	\begin{equation}
		\label{eq:exp_discrete_divB_error}
			\left(\na \cdot \Bb^{n+1}\right)_{\iph,\jph} = 	\left(\na \cdot \Bb^n\right)_{\iph,\jph},
	\end{equation}
	for the magnetic field $\Bb$ which is consistent with \eqref{eq:divB_evo_cont}. Similarly, for the electric field $\Eb$, the evolution of the discrete divergence of the electric field is given by, 
	\begin{equation}
		\label{eq:exp_discrete_divE_error}
		\left(\na \cdot \Eb^{n+1}\right)_{\iph,\jph} = \left(\na \cdot \Eb^n\right)_{\iph,\jph} -\frac{\Dt}{2} \left( \left(\na \cdot \Jb^{n}\right)_{\iph,\jph} + \left(\na \cdot \Jb^{(1)}\right)_{\iph,\jph}\right),
	\end{equation}
which is consistent with \eqref{eq:divE_evo_cont}.
\end{prop}
\begin{proof}
   Following the explicit scheme notation, define $\Bb^{(0)}=\Bb^n$. Then using Theorem \ref{thm:divB} with explicit time discretization, we get,
   $$
   \left(\na \cdot \Bb^{(1)}\right)_{\iph,\jph} = 	\left(\na \cdot \Bb^{(0)}\right)_{\iph,\jph},
   $$
   and
   $$
   \left(\na \cdot \Bb^{(2)}\right)_{\iph,\jph} = 	\left(\na \cdot \Bb^{(1)}\right)_{\iph,\jph}.
   $$
   Using these two equations, we get
   $$
   	\left(\na \cdot \Bb^{n+1}\right)_{\iph,\jph} = \frac{1}{2}\left( 	\left(\na \cdot \Bb^{(0)}\right)_{\iph,\jph} + \left(\na \cdot \Bb^{(2)}\right)_{\iph,\jph}  \right) = \left(\na \cdot \Bb^n\right)_{\iph,\jph}.
   $$
   Similarly for electric field, denote $\Eb^{(0)}=\Eb^n$; then using Theorem \ref{thm:divE}, we obtain
     $$
    \left(\na \cdot \Eb^{(1)}\right)_{\iph,\jph} = 	\left(\na \cdot \Eb^{(0)}\right)_{\iph,\jph} -\Dt   \left(\na \cdot \Jb^{(0)}\right)_{\iph,\jph}
    $$
    and
    $$
\left(\na \cdot \Eb^{(2)}\right)_{\iph,\jph} = 	\left(\na \cdot \Eb^{(1)}\right)_{\iph,\jph} -\Dt   \left(\na \cdot \Jb^{(1)}\right)_{\iph,\jph}
    $$
   Finally
   \begin{align*}
   	\left(\na \cdot \Eb^{n+1}\right)_{\iph,\jph} &= \frac{1}{2}\left( \left(\na \cdot \Eb^{(2)}\right)_{\iph,\jph} + \left(\na \cdot \Eb^{(0)}\right)_{\iph,\jph} \right)\\
   	&= \frac{1}{2}\left(  \left(\na \cdot \Eb^{(1)}\right)_{\iph,\jph} -\Dt   \left(\na \cdot \Jb^{(1)}\right)_{\iph,\jph} + \left(\na \cdot \Eb^{(0)}\right)_{\iph,\jph} \right)\\
   	&=  \frac{1}{2}\left(  \left(\na \cdot \Eb^{(0)}\right)_{\iph,\jph} -\Dt   \left(\na \cdot \Jb^{(0)}\right)_{\iph,\jph} -\Dt   \left(\na \cdot \Jb^{(1)}\right)_{\iph,\jph} + \left(\na \cdot \Eb^{(0)}\right)_{\iph,\jph} \right) \\
   	&= \left(\na \cdot \Eb^{n}\right)_{\iph,\jph} -\frac{\Dt}{2} \left( \left(\na \cdot \Jb^{n}\right)_{\iph,\jph} + \left(\na \cdot \Jb^{(1)}\right)_{\iph,\jph}\right).
   \end{align*}
\end{proof}

\subsection{IMEX scheme}
\label{subsec:imex}
The source terms contain charge-to-mass ratios which can be very large, leading to stiffness. To overcome the time-stepping restriction imposed by this stiffness, we consider IMEX schemes, where we treat the source terms implicitly~\cite{Bhoriya2023}. We will use a L-stable second-order accurate IMEX scheme from~\cite{Pareschi2005}. It is based on the evaluation of the two implicit stages. The scheme is given as follows:
\begin{eqnarray}
			\label{eq:imex_1st_step}
	\Ub_{i,j}^{(1)}&=& \Ub_{i,j}^n + \Dt \left(\beta \mathcal{S} \left(\Ub_{i,j}^{(1)} \right) \right) \\
			\label{eq:imex_2nd_step}	
	\Ub_{i,j}^{(2)} &=& \Ub_{i,j}^n + \Dt \left[ \mathcal{L}_{i,j} \left(\Ub^{(1)}\right) + (1-2\beta) \mathcal{S} \left(\Ub_{i,j}^{(1)} \right) + \beta \mathcal{S} \left(\Ub_{i,j}^{(2)} \right) \right]\\
			\label{eq:imex_3rd_step}
\Ub^{n+1}_{i,j} &=&\Ub_{i,j}^n + \frac{1}{2}\Dt\left[ \mathcal{L}_{i,j} \left(\Ub^{(1)}\right) + \mathcal{L}_{i,j} \left(\Ub^{(2)}\right)  + \mathcal{S} \left(\Ub_{i,j}^{(1)}\right) +\mathcal{S} \left(\Ub_{i,j}^{(2)} \right) \right]
\end{eqnarray} 
Here, $\beta = 1 - \frac{1}{\sqrt{2}}$. To solve the nonlinear implicit equations in \eqref{eq:imex_1st_step} and \eqref{eq:imex_2nd_step}, we use Newton's method based on the backtracing line search~\cite{Dennis1996}. The corresponding scheme is denoted by {\bf O2IMEX-MultiD}. 

For the IMEX scheme  {\bf O2IMEX-MultiD}, we have the following result for the evolution of divergence constraints.

\begin{prop}[Evolution of the divergence constraint for the explicit scheme {\bf O2IMEX-MultiD}]
		\label{prop:imex}
	 The IMEX scheme update~\eqref{eq:imex_1st_step}-\eqref{eq:imex_3rd_step} for the semi-discrete scheme  \eqref{eq:semi_discrete} satisfies,
	\begin{equation}
	\label{eq:imex_divB_error}
	(\na \cdot \Bb^{n+1})_{\iph,\jph} = (\na \cdot \Bb^n)_{\iph,\jph},
\end{equation}
	for the magnetic field $\Bb$ which is consistent with \eqref{eq:divB_evo_cont}. Similarly, for the electric field $\Eb$, the evolution of the discrete divergence is given by, 
	\begin{equation}
		\label{eq:imex_divE_error}
		\left(\na \cdot \Eb^{n+1}\right)_{\iph,\jph} = \left(	\na \cdot \Eb^n\right)_{\iph,\jph} -\frac{\Dt}{2} \left( \left(\na \cdot \Jb^{(1)}\right)_{\iph,\jph} + \left(\na \cdot \Jb^{(2)}\right)_{\iph,\jph}\right)
	\end{equation}
	which is consistent with \eqref{eq:divE_evo_cont}.
\end{prop}
\begin{proof}
	The proof is similar to the case of {\bf O2EXP-MultiD} explicit scheme.
\end{proof}

\section{Numerical results} \label{sec:num_test}
We will now present the numerical results for the proposed schemes for the various test cases. In few test cases we need additional resistive effect ~\cite{Amano2016,Bhoriya2023}, which are incorporated by modifying Eqns.\eqref{momentum} and \eqref{energy} to,

\begin{subequations}
	\begin{align}
	\frac{\partial \mathbf{m}_\alpha}{\partial t} + \nabla \cdot (\mathbf{m}_\alpha \mathbf{u}_\al  + p_\alpha \mathbf{I}) & = r_\alpha \Gamma_\alpha \rho_\alpha (\mathbf{E}+\mathbf{u}_\al \times \mathbf{B}) + \boldsymbol{R}_\alpha  , \label{resist_momentum}	\\ 
		%
	\frac{\partial \mathnormal{E}_\al}{\partial t} + \nabla \cdot ((\mathnormal{E}_\alpha+p_\al)\mathbf{u}_\al) & = r_\al \Gamma_\al \rho_\al (\mathbf{u}_\al \cdot \mathbf{E}) + R_\alpha^0,	\label{resist_energy} 
	\end{align}
\label{eq:resistive}
\end{subequations}
for $\alpha \in \{i,e\}$. Following\cite{Amano2016}, the anti-symmetry relationship, $(\boldsymbol{R}_e, {R}_e^0)=(-\boldsymbol{R}_i, -R_i^0)$, needs to be imposed to ensure conservation of momentum and energy. The terms $\boldsymbol{R}_i$ and $R_i^0$ are then given by 
\begin{align*}
	\boldsymbol{R}_i =  -\eta \dfrac{\omega_p^2}{r_i-r_e} (\mathbf{J}- \rho_0 \boldsymbol{\Phi}), \qquad
	{R}_i^0 =  -\eta \dfrac{\omega_p^2}{r_i-r_e} (\rho_c- \rho_0 {\Lambda}),
\end{align*}
with,
\begin{eqnarray} \label{source_var}
	\omega_p^2 = r_i^2 \rho_i + r_e^2 \rho_e, \qquad
	\boldsymbol{\Phi} = \dfrac{r_i \rho_i \Gamma_i \mathbf {u}_i + r_e \rho_e \Gamma_e \mathbf{u}_e}{\omega_p^2}, \qquad
	\Lambda = \dfrac{r_i^2 \rho_i \Gamma_i  + r_e^2 \rho_e \Gamma_e }{\omega_p^2}, \qquad
	\rho_0 = \Lambda \rho_c - \mathbf{J} \cdot \boldsymbol{\Phi}. \nonumber
\end{eqnarray}
Here,  $\eta$ is the resistivity constant, and $\omega_p$ is the total plasma frequency. Also, the total plasma skin depth $d_p$ is defined as $d_p = \frac{1}{\omega_p}$. For the time update using IMEX scheme \mdi, we treat the resistive terms implicitly and couple them with source $\mathcal{S}$.

To ensure consistent comparison with the numerical results in \cite{Balsara2016}, for the test cases presented here, we multiply the source term of the Maxwell's equations $\mathcal{S}_m$ with a factor of $4\pi$. Also, for the test cases from \cite{Balsara2016}, the total plasma skin depth is given by $d_p=\frac{\sqrt{2}}{\omega_p}$ with $\omega_p^2=4\pi(r_i^2 \rho_i + r_e^2 \rho_e).$  

\subsection{One-dimensional test cases}
\rev{The 1-D equations are obtained from the 2-D equations by removing the derivatives with respect to $y$. In the numerical scheme and code, we solve for all the quantities including $B_x$.} In the case of one-dimensional tests, the time step $\Delta t$ is given by, 	
\[
{\Delta t} = \text{CFL} \cdot \min \left\{ \dfrac{\Delta x}{\Lambda_{max}^x(\mathbf{U}_i)} : 1 \le i \le N_x\right\}, \textrm{ where } \Lambda_{max}^x(\mathbf{U}_i) = \max\{|\Lambda^x_k(\mathbf{U}_i)|: 1 \le k \le 16\}.
\]
We take the CFL number of 0.8 for both  \mde \ and \mdi~schemes unless stated otherwise. In one dimension, any consistent scheme automatically satisfies the divergence constraints, \rev{which reduces to $B_x = $ constant}. Therefore, there is no need to address divergence errors for one-dimensional test cases. Hence, we will limit our discussion to the accuracy and wave-capturing abilities of the proposed schemes. The performance of the schemes related to the divergence constraints will be discussed for the two-dimensional test cases. 

\subsubsection{Accuracy test}
\label{test:1d_smooth}
To demonstrate the accuracy of the schemes, we consider the test case from~\cite{Bhoriya2023,bhoriya_entropydg_2023}, which is motivated by the modified equations approach of~\cite{kumar2012entropy, Abgrall2014}. We consider,
\begin{equation*}
	\frac{\partial \mathbf{U}}{\partial t}+\frac{\partial \mathbf{f}^x}{\partial x} = \mathcal{S} + \mathbf{\mathcal{R}}(x,t)
\end{equation*}
where,
\begin{align*}
	\mathbf{\mathcal{R}}(x,t)=\Big(\mathbf{0}_{13},
	-\dfrac{1}{\sqrt{3}}(2+\sin(2 \pi (x-0.5t))),
	0,
	-3\pi\cos(2 \pi (x-0.5t))\Big)^\top.
\end{align*}
 The initial densities are taken to be $\rho_i=\rho_e=2+\sin(2 \pi x)$ with velocities $u_i^x=u_{e}^x = 0.5$ and constant pressures, $p_i=p_e=1$.  Also,  initially, we take $B_y=2 \sin(2 \pi (x))$ and $E_z = -\sin(2\pi (x)$. All other variables are taken to zeros. The computational domain is $[0,1]$ with periodic boundary conditions. The ion-electron mass ratio is taken to be $2.0$ and heat constants are taken to be $\gamma_i=\gamma_e=5/3.$ The exact solution of the problem is given by $\rho_i = \rho_e = 2+\sin(2 \pi (x-0.5 t))$.

\begin{table}[ht]
	\centering
	\begin{tabular}{|c|c|c|c|c|}
		\hline
		Number of cells & \multicolumn{2}{|c}{ \textbf{O2EXP-MultiD} } &
		\multicolumn{2}{|c|}{{\textbf{O2IMEX-MultiD}}} \\
		\hline
		-- & $L^1$ error & Order & $L^1$ error & Order \\
		\hline
		32 & 1.29481e-01 & -- & 1.29676e-01 & -- \\
		64 & 4.83174e-02 & 1.4221 & 4.83361e-02 & 1.4237 \\
		128 & 1.52820e-02 & 1.6607 & 1.52902e-02 & 1.6605 \\
		256 & 4.25288e-03 & 1.8453 & 4.25412e-03 & 1.8457 \\
		512 & 1.16120e-03 & 1.8728 & 1.16151e-03 & 1.8729 \\
		1024 & 3.13057e-04 & 1.8911 & 3.13114e-04 & 1.8912 \\
		2048 & 8.27499e-05 & 1.9196 & 8.27657e-05 & 1.9196 \\
		\hline
	\end{tabular}
	\caption[h]{\nameref{test:1d_smooth}: $L^1$ errors and order of convergence for $\rho_i$ using the explicit scheme \textbf{O2EXP-MultiD} and implicit scheme  \textbf{O2IMEX-MultiD}.}
	\label{table:order}
\end{table}	
Table~\eqref{table:order} consists of $L^1$ errors for the variable $\rho_i$, at computational time $t=2.0$ for \mde~and \mdi \ schemes. We note that for~\mdi \ scheme, $\mathcal{R}$ is evaluated to be consistent with the source $\mathcal{S}$ time discretization. We note that both the schemes achieve theoretical second-order accuracy and have comparable errors at all the resolutions.

\subsubsection{Relativistic Brio-Wu test problem with finite plasma skin depth} 
\label{test:1d_brio} 

In this test case, we consider a shock tube problem from \cite{Amano2016, Balsara2016, Bhoriya2023,bhoriya_entropydg_2023}, which is motivated by the Brio-Wu shock tube problem for classical MHD \cite{brio1988}. A non-relativistic two-fluid version is considered in \cite{Hakim2006,loverich2011,kumar2012entropy, Abgrall2014}. The RMHD version of the problem was considered in~\cite{Balsara2001}.  Here, we scale the source term of Maxwell's equation with a factor of $4\pi$, to have a consistent comparison with the results in \cite{Balsara2016}. No resistive effects are considered. The domain of the problem is $[-0.5,0.5]$ with Neumann boundary conditions. The initial solution has two constant states, given by,
\begin{align*}
	\mathbf{W_L}=\begin{cases}
		\rho_i =0.5\\ p_i=0.5 \\ \rho_e=0.5 \\ p_e=0.5 \\ B_x=\sqrt{\pi} \\ B_y = \sqrt{4 \pi}
		\end{cases},
	\qquad
	\mathbf{W_R}=\begin{cases}
		\rho_i=0.0625 \\ p_i=0.05 \\ \rho_e=0.0625 \\ p_e=0.05 \\ B_x=\sqrt{\pi} \\ B_y-\sqrt{4 \pi}
	\end{cases},
\end{align*}
separated at $x=0.0$. All other variables have zero values.  We consider the specific heat ratios to be $\gamma_i=\gamma_e=2.0$. The simulations are performed using  $r_i = -r_e$ for different values of $r_i=\frac{10^3}{\sqrt{4\pi}}, \frac{10^4}{\sqrt{4\pi}}, \frac{10^5}{\sqrt{4\pi}}$ with the corresponding plasma skin depths of $\frac{10^{-3}}{\sqrt{\rho_i}}$, $\frac{10^{-4}}{\sqrt{\rho_i}}$. and $\frac{10^{-5}}{\sqrt{\rho_i}}$, respectively. When cell resolution $\Delta x$ is bigger than the plasma skin depth, we expect the solution to be close to the ``RMHD" solution. 

The simulations are performed till time $t=0.4$ using $400$, $1600$, and $6400$ cells for the schemes \mde~and \mdi. The ``RMHD" solution used for comparison, is computed using \mdi~with $6400$ cells and $r_i=\frac{10^5}{\sqrt{4\pi}}$.

\begin{figure}[!htbp]
	\begin{center}
		\subfigure[Plots of total density $\rho_i+\rho_e$ using $400$, $1600$ and $6400$ cells.]{
			\includegraphics[width=3.0in, height=2.3in]{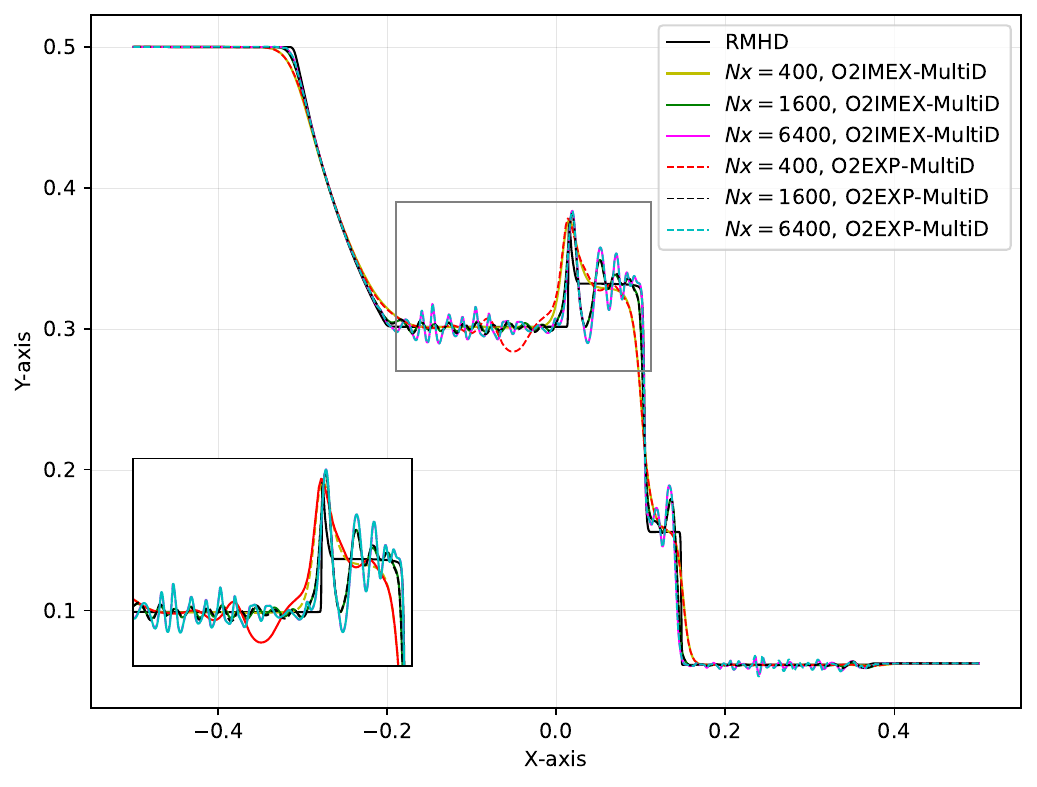}
			\label{fig:brio_rho_o2}}
		\subfigure[Plots of total density $\rho_i+\rho_e$ using $400$ cells.]{
			\includegraphics[width=3.0in, height=2.3in]{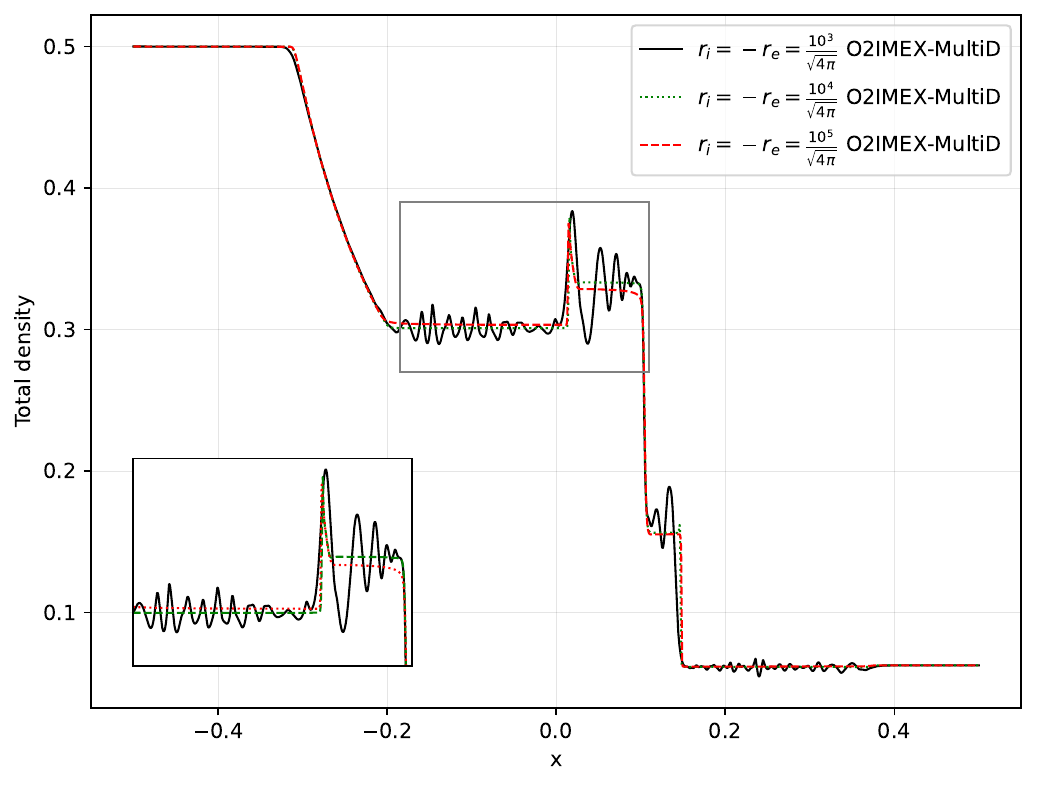}
			\label{fig:bw_6400}}
		\caption{Numerical results for \nameref{test:1d_brio} using different schemes and grid resolutions.}
		\label{fig:brio}
	\end{center}
\end{figure}

\begin{table}[ht]
	\centering
	\begin{tabular}{|c|c|c|c| }
		\hline
		& \multicolumn{3}{c|}{$r_i = -r_e = 10^4/ \sqrt{4 \pi}$} \\
		\hline
		Number of cells & $256$ & $512$ & $1024$  \\
		\hline
		{\bf O2IMEX-MultiD} & 2.85s (0.8) & 9.39s (0.8) & 38.93s (0.8) 
        \\
		\hline
		{\bf O2EXP-MultiD}  & 11.07s (0.02) & 20.34s (0.04) & 39.99s (0.08) 
        \\
		\hline
        & \multicolumn{3}{c|}{$r_i = -r_e = 10^5/ \sqrt{4 \pi}$} \\
		\hline
		Number of cells & $256$ & $512$ & $1024$  \\
		\hline
		{\bf O2IMEX-MultiD} & 2.81s (0.8)& 9.34s (0.8)& 34.35s (0.8)
        \\
		\hline
		{\bf O2EXP-MultiD}  & 201.68s (0.0011)& 331.81s (0.0025)& 1172.85s (0.0049)
        \\
		\hline
	\end{tabular}
	\caption[h]{\nameref{test:1d_brio}: Computational times (and CFL) used for  $r_i = -r_e = 10^4/ \sqrt{4 \pi}$ and $r_i = -r_e = 10^5/ \sqrt{4 \pi}$ using the second order({\bf O2EXP-MultiD} and {\bf O2IMEX-MultiD}) schemes.}
	\label{table:time_st1}
\end{table}
In Figure~\ref{fig:brio_rho_o2}, we consider the case with  $r_i=\frac{10^3}{\sqrt{4\pi}}.$  We have plotted the total density using~\mde~and~\mdi~schemes. When 400 cells are used, the resolution is bigger than the plasma skin depth, hence the solution is close to the ``RMHD" solution. When the resolution is increased to $1600$ and $6400$ cells, both schemes can resolve small-scale dispersive oscillations. Particularly when $6400$ cells are used, both schemes are able to resolve small-scale dispersive waves and also capture the discontinuities. 

To study the effects of different values of $r_i$, we have plotted the total density for $r_i=\frac{10^3}{\sqrt{4\pi}},\frac{10^4}{\sqrt{4\pi}},\frac{10^5}{\sqrt{4\pi}}$ in Figure~\ref{fig:bw_6400}. We note that for $r_i=\frac{10^3}{\sqrt{4\pi}}$, the cell size is smaller than the plasma skin depth. As a result, we are able to resolve all the dispersive waves. Nevertheless, the plasma skin depth is smaller for various values of $r_i$.  Hence, we are not able to resolve the additional waves, and the solution is close to the ``RMHD" solution.

\rev{To demonstrate the computational efficiency of IMEX schemes, we have presented the computational time and CFL for~\mde~and~\mdi~schemes with $r_i=\frac{10^4}{\sqrt{4\pi}}$ and $r_i=\frac{10^5}{\sqrt{4\pi}}$ for various resolutions. As the source term is stiff in these cases, the time step for the ~\mde~  is governed by the source term. Hence, we decrease the CFL for ~\mde~ till we get stable numerical results. This is unnecessary for the ~\mdi~ scheme, as the source is treated implicitly. Hence, the time step is given by the CFL of $0.8$ in all the cases. However, the IMEX update needs to solve nonlinear equations, which is more expansive than the explicit treatment.

For the case of $r_i=\frac{10^4}{\sqrt{4\pi}}$, at the lower resolutions, \mdi~scheme is much more computationally efficient as the ~\mde~ scheme needs a very small CFL number to make the source term stable. At higher resolution of $1024$ cells, the time step is still governed by the source term. Hence, ~\mde~ needs to take CFL of $0.08$, but the computational times are still comparable with ~\mdi~ with CFL of $0.8$ because of the computational cost for solving nonlinear equations.

For the case of $r_i=\frac{10^5}{\sqrt{4\pi}}$, even at the higher resolution of $1024$, the~\mdi~ scheme is more computationally effective as the CFL needed for a stable run of ~\mde~ is too small to overcome the computational cost of solving nonlinear equations in ~\mdi. Of course, ~\mdi~ is much more cost-effective on the coarser mesh.

 We further note a negligible difference in computational time for~\mdi~scheme for both values of $r_i$; however, for~\mde~scheme computational time increases significantly with $r_i$.}

\subsubsection{Self-similar current sheet with finite resistivity} 
\label{test:1d_sheet} 
To demonstrate the effects of resistive terms, we take into consideration this test from~\cite{Amano2016}, which is inspired by the RMHD test case, presented in~\cite{Komissarov2007}. We consider the modified Eqns.~\eqref{resist_momentum} and~\eqref{resist_energy} to account for additional resistive effects. Only the $B_y$ component has non-zero variation, and its time evolution is governed by,

\begin{equation} \label{current_diff}
	\dfrac{\partial B_y}{\partial t} - D \dfrac{\partial^2 B_y}{\partial x^2} =0,
\end{equation}
with $D=\eta c^2$. The exact solution (see \cite{Komissarov2007}) is then given by,
\begin{equation}
	\label{eq:Byrmhd}
	B_y(x,t) = B_0 \text{  erf}\left(\dfrac{x}{2 \sqrt{Dt}} \right),
\end{equation}
where ``erf" is the error function. For our model, we consider the initial $B_y$ at the time $t=1.0$, with  $B_0=1.0$ and $\eta c^2 =0.01$. We consider the computational domain of $[-1.5,1.5]$ with Neumann boundary conditions. The charge-to-mass ratios are $r_i = -r_e=10^3$. We also consider the ion and electron density as 0.5 and the ion and electron pressures as 25.0. The $z$-component of the ion and electron velocity is given by
\[
u^z_i=-u^z_e=\dfrac{B_0}{r_i \rho_i \sqrt{\pi D}} \exp\left(-\dfrac{x^2}{4D}\right).
\]

All other variables are assumed to be vanishing. We simulate till time $t=9.0$ using \mde~and \mdi~schemes with $400$ cells. The results are presented in Figure~\ref{fig:sheet}. We note that the computed solution for both schemes matches with the analytical solution derived from resistive RMHD. Furthermore, both schemes have similar accuracy.

\begin{figure}[!htbp]
	\begin{center}
		\subfigure[]{
			\includegraphics[width=3.0in, height=2.6in]{current_sheet.pdf}}
			\caption{\nameref{test:1d_sheet}: Comparison of the $B_y$ profile for explicit \mde~ and IMEX scheme \mdi~using $400$ cells.}
		\label{fig:sheet}
	\end{center}
\end{figure}

\subsection{Two-dimensional test cases}
In this section, we will present two-dimensional test cases. The focus will be on the performance of the schemes in preserving the divergence constraints. We will compare the divergence constraints preservation performance of the proposed scheme with no divergence treatment and PHM-based divergence treatment. In no divergence treatment, we ignore the divergence constraints (Eqns. \eqref{eq:div_B} and \eqref{eq:gauss_E}) and just discretize Eqns. \eqref{eq:max_B} and \eqref{eq:max_E} using one-dimensional Rusanov solver and one-dimensional MinMod reconstructions. The corresponding explicit and IMEX schemes are denoted by \ote~and \oti, respectively. For PHM-based schemes, we consider \eqref{eq:phm_eqn} instead of \eqref{eq:maxwell_eqn}. They are then discretized using a standard one-dimensional Rusanov solver and one-dimensional MinMod reconstruction (see \cite{Bhoriya2023} for complete details). The corresponding explicit and IMEX schemes are denoted by \phme~and \phmi, respectively. The time step for two-dimensional test cases is computed using,
 \[
{\Delta t} = \text{CFL}  \min_{i,j} \left\{ \dfrac{1}{ \dfrac{\Lambda_{max}^x(\mathbf{U}_{i,j})} {\Delta x} + \dfrac{\Lambda_{max}^y(\mathbf{U}_{i,j})}{\Delta y}}: 1 \le i \le N_x, \ 1 \le j \le N_y \right\}
\]

	where $\Lambda_{max}^x(\mathbf{U}_{i,j}) =  \max \{|\Lambda^x_k(\mathbf{U}_{i,j})| : 1 \le k \le 16 \}$ and $\Lambda_{max}^y(\mathbf{U}_{i,j}) = \max\{|\Lambda^y_k(\mathbf{U}_{i,j})|: 1 \le k \le 16\}$. We take CFL to be $0.45$ for the \mdi~and $0.2$ for \mde.

\par To compare the divergence constraints preserving the performance of the schemes, following the Theorem we compute the $L^1$ and $L^2$ norms of the discrete divergence of the magnetic field as follows:
	\begin{itemize}
	\item \textbf{$L^1$ error of magnetic field divergence constraint}: \qquad
	$$
	\| \nabla \cdot \Bb\|_1= 	\frac{1}{N_x N_y}\sum_{i=1}^{N_x} \sum_{j=1}^{N_y} |  (\nabla\cdot\mathbf{B})_{i+\frac{1}{2},j+\frac{1}{2}}|
	$$
	\item \textbf{$L^2$ error of magnetic field divergence constraint}: \qquad
	$$\| \nabla \cdot \Bb\|_2=\left[
	\frac{1}{N_x N_y}\sum_{i=1}^{N_x} \sum_{j=1}^{N_y} |  (\nabla\cdot\mathbf{B})_{i+\frac{1}{2},j+\frac{1}{2}}  |^2
	\right]^{1/2} $$
\end{itemize}
Similarly, following Proposition \ref{prop:exp} and \ref{prop:imex}, we define explicit electric field divergence constraint error as,
\begin{align*}
	R_E^{n} = \left((\nabla\cdot\mathbf{E}^{n+1})_{i+\frac{1}{2},j+\frac{1}{2}}-\left(
	(\nabla\cdot\mathbf{E}^{n})_{i+\frac{1}{2},j+\frac{1}{2}}
	- 
	\dfrac{\Delta t}{2}
	\left[
	(\nabla\cdot\mathbf{J}^{(1)})_{i+\frac{1}{2},j+\frac{1}{2}}
	+
	(\nabla\cdot\mathbf{J}^{n})_{i+\frac{1}{2},j+\frac{1}{2}}
	\right]
	\right)\right)
\end{align*}
and IMEX electric field divergence constraint error as,
\begin{align*}
	R_I^{n} = \left((\nabla\cdot\mathbf{E}^{n+1})_{i+\frac{1}{2},j+\frac{1}{2}}-\left(
	(\nabla\cdot\mathbf{E}^{n})_{i+\frac{1}{2},j+\frac{1}{2}}
	- 
	\dfrac{\Delta t}{2}
	\left[
	(\nabla\cdot\mathbf{J}^{(1)})_{i+\frac{1}{2},j+\frac{1}{2}}
	+
	(\nabla\cdot\mathbf{J}^{(2)})_{i+\frac{1}{2},j+\frac{1}{2}}
	\right]
	\right)\right)
\end{align*}
Using these two expressions the corresponding $L^1$ and $L^2$ errors are defined as,
\begin{itemize}
	\item \textbf{$L^1$ error of electric field divergence constraint for explicit scheme}: \qquad
	$$
	\|\na\cdot \Eb\|_{1}^{E}	 =  \frac{1}{N_x N_y}\sum_{i=1}^{N_x} \sum_{j=1}^{N_y} 
	\left|
	R_E^n
	\right|
	$$
	\item \textbf{$L^2$ error of electric field divergence constraint for explicit scheme}: \qquad
	$$\|\na\cdot\Eb\|_{2}^{E}=\left[
	\frac{1}{N_x N_y}\sum_{i=1}^{N_x} \sum_{j=1}^{N_y} 
	\left| 
	R_E^n
	\right|^2
	\right]^{1/2} $$
	\item \textbf{$L^1$error of electric field divergence constraint for IMEX scheme}: \qquad
	$$
	\|\na\cdot\Eb\|_{1}^{I}= 	\frac{1}{N_x N_y}\sum_{i=1}^{N_x} \sum_{j=1}^{N_y} 
	\left|  
	R_I^n
	\right|
	$$
	\item \textbf{$L^2$ error of electric field divergence constraint for IMEX scheme}: \qquad
	$$\|\na\cdot\Eb\|_{2}^{I}=\left[
	\frac{1}{N_x N_y}\sum_{i=1}^{N_x} \sum_{j=1}^{N_y} 
	\left|  
	R_I^n
	\right|^2
	\right]^{1/2}$$
\end{itemize}

\subsubsection{Two-dimensional smooth problem} 
\label{test:smooth_2d}
This case is a two-dimensional version of the test case considered in Section \ref{test:1d_smooth}. We consider
\begin{equation*}
	\frac{\partial \mathbf{U}}{\partial t}
	+ 
	\frac{\partial \mathbf{f}^x}{\partial x} 
	+
	\frac{\partial \mathbf{f}^y}{\partial y} 
	= \mathcal{S} + 
	\mathbf{\mathcal{R}}(x,y,t) 
\end{equation*}
with 
\begin{align*}
	\mathbf{\mathcal{R}}(x,y,t)=\Big(\mathbf{0}_{13},-\dfrac{1}{\sqrt{14}}(2+\sin(2 \pi (x+y-0.5t))),-\dfrac{1}{\sqrt{14}}(2+\sin(2 \pi (x+y-0.5t))),
 \\
 -7\pi\cos(2 \pi (x+y-0.5t)) \Big)^\top.	
\end{align*}
We take initial densities as $\rho_i = \rho_e = 2+\sin(2 \pi (x+y))$, with initial velocities $\ub_i=(0.25,0.25,0)$ and $\ub_{e} = (0.25,0.25,0)$ and initial pressures $p_i=p_e=1$. The magnetic field is taken to be $\Bb= (-2 \sin(2 \pi (x+y)),2 \sin(2 \pi (x+y)),0)$ and the electric field is $\Eb = (0,0,-\sin(2\pi (x+y))$. All other variables are set to zero. We consider the computational domain of $I=[0,1] \times [0,1]$ with the periodic boundary conditions. We also set charge to mass ratios $r_i = 1$ and $r_e = -2$. We use ion-electron adiabatic index $\gamma=5/3$. Using the above conditions, it is easy to verify that the exact solution is $\rho_i = \rho_e = 2+\sin(2 \pi (x+y-0.5 t))$.
\begin{table}[ht]
	\centering
	\begin{tabular}{|c|c|c|c|c|}
		\hline
		Number of cells & \multicolumn{2}{|c}{ \textbf{O2EXP-MultiD} } &
		\multicolumn{2}{|c|}{{\textbf{O2IMEX-MultiD}}} \\
		\hline
		-- & $L^1$ error & Order & $L^1$ error & Order \\
		\hline
		32 & 1.29620e-01 & -- & 1.29620e-01 & -- \\
		64 & 4.83174e-02 & 1.4222 & 4.83262e-02 & 1.4234 \\
		128 & 1.52820e-02 & 1.6608 &  1.52890e-02&  1.6603\\
		256 & 4.25288e-03 & 1.8451 & 4.25412e-03 & 1.8456 \\
		512 & 1.16120e-03 & 1.8728 & 1.16150e-03 & 1.8728 \\
		1024 & 3.13057e-04 & 1.8910 & 3.13113e-04 &  1.8912\\
		2048 & 8.27499e-05 & 1.9196 &  8.27656e-05& 1.9196 \\
		\hline
	\end{tabular}
	\caption[h]{\nameref{test:smooth_2d}: $L^1$ errors and order of convergence for $\rho_i$ using the explicit scheme \textbf{O2EXP-MultiD} and implicit scheme  \textbf{O2IMEX-MultiD}.}
	\label{table:order_2d}
\end{table}	

In Table~\eqref{table:order_2d}, we have presented the $L^1$-errors for explicit \textbf{O2EXP-MultiD} and IMEX scheme \textbf{O2IMEX-MultiD} at the final time $t=10.0$. We have presented $ L^1$ errors for $\rho_i$ at different resolutions. For both explicit and IMEX schemes, we observe a consistent order of accuracy at various resolutions. Furthermore, we also note that the errors for both explicit and IMEX schemes are comparable at a given resolution.

One key feature of this test case is that $\nabla \cdot \Eb\neq \rho_c$ due to the presence of source $\mathbf{\mathcal{R}}$, unlike the other two-dimensional case we have considered. Hence, $R^n_I$ and $R^n_E$ are not vanishing. So, $\|\nabla \cdot E\|_1^I, \|\nabla \cdot E\|_2^I, \|\nabla \cdot E\|_1^E$ and $\|\nabla \cdot E\|_2^E$ are not going to be zeros. But we expect \mdi~and \mde~ schemes to keep the initial nonzero values constant for all time. In Figure~\ref{fig:smooth2d_norms}, we have plotted the magnetic and electric field errors for \mdi,~\mde,~\ote,~\oti,~\phme~and ~\phmi~ schemes. We note that the norms of divergence of the magnetic field are very close to the machine errors for all the schemes. For the norms related to the Guass's law constraint, we see that the \mde,~, and \mdi~ schemes maintain the initial constant norms, whereas all other schemes see some increase. 

\begin{figure}[!htbp]
	\begin{center}
	\subfigure[$\|\na \cdot \mathbf{B}^n\|_1$,  and $\|\na \cdot \mathbf{B}^n\|_2$  errors for explicit schemes  \textbf{O2EXP-MultiD}, \textbf{O2EXP-PHM} and \textbf{O2EXP}.]{
			\includegraphics[width=2.9in, height=2.5in]{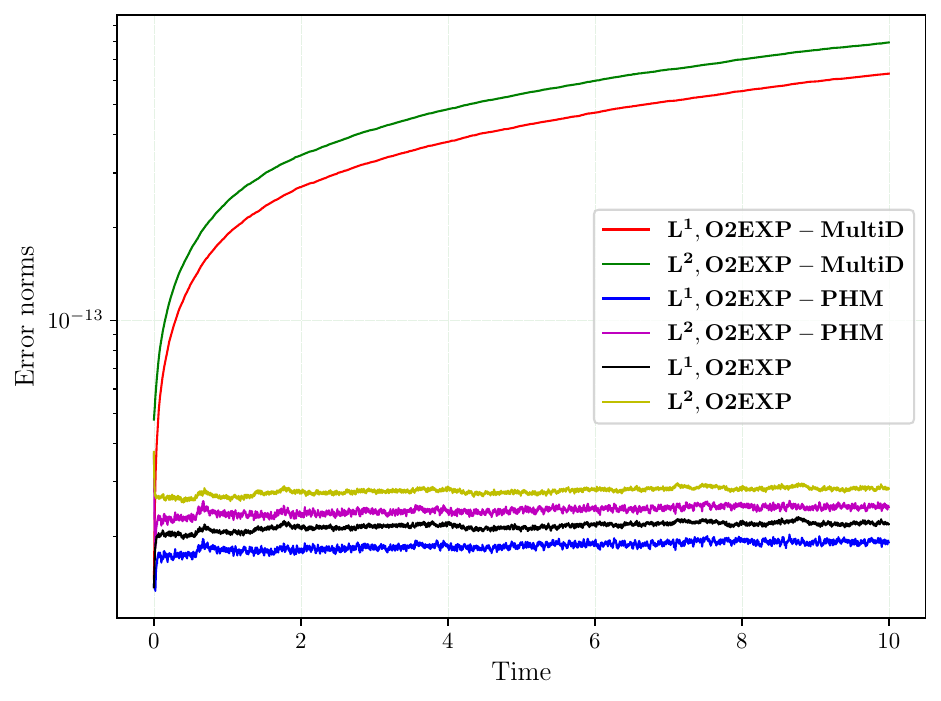}
			\label{fig:smooth2d_B_exp}}
		\subfigure[$\|\na \cdot \mathbf{B}^n\|_1$,  and $\|\na \cdot \mathbf{B}^n\|_2$  errors for IMEX schemes \textbf{O2IMEX-MultiD}, \textbf{O2IMEX-PHM} and \textbf{O2IMEX}.]{
			\includegraphics[width=2.9in, height=2.5in]{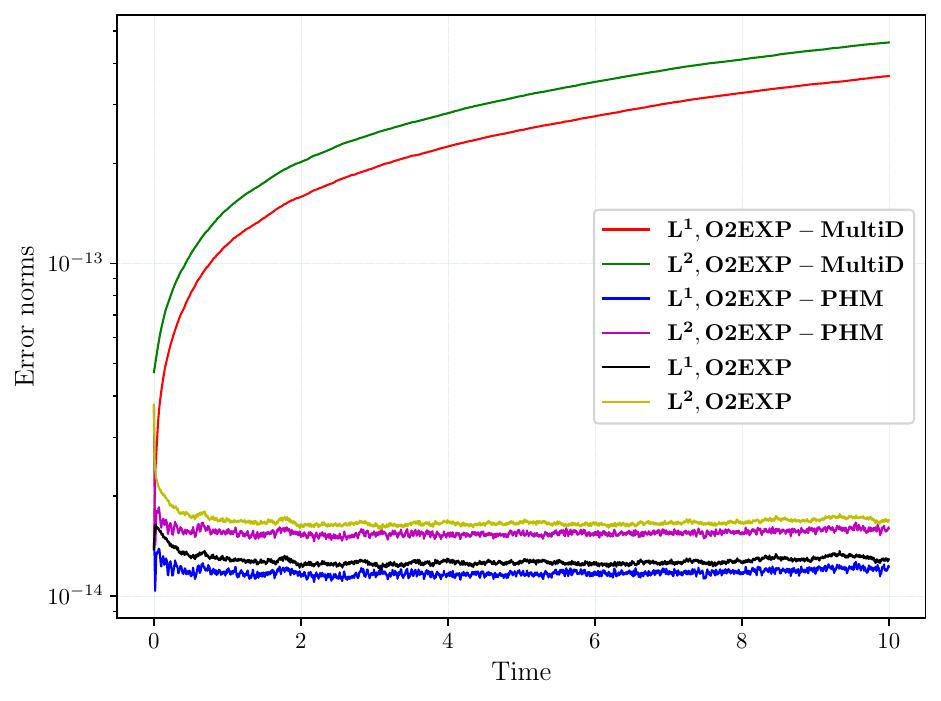}
			\label{fig:smooth2d_B_imp}}
			\subfigure[$\|\na\cdot\Eb\|_{1}^{E},$ and $\|\na\cdot\Eb\|_{2}^{E}$ errors, for explicit schemes  \textbf{O2EXP-MultiD}, \textbf{O2EXP-PHM} and \textbf{O2EXP}.]{
		\includegraphics[width=2.9in, height=2.5in]{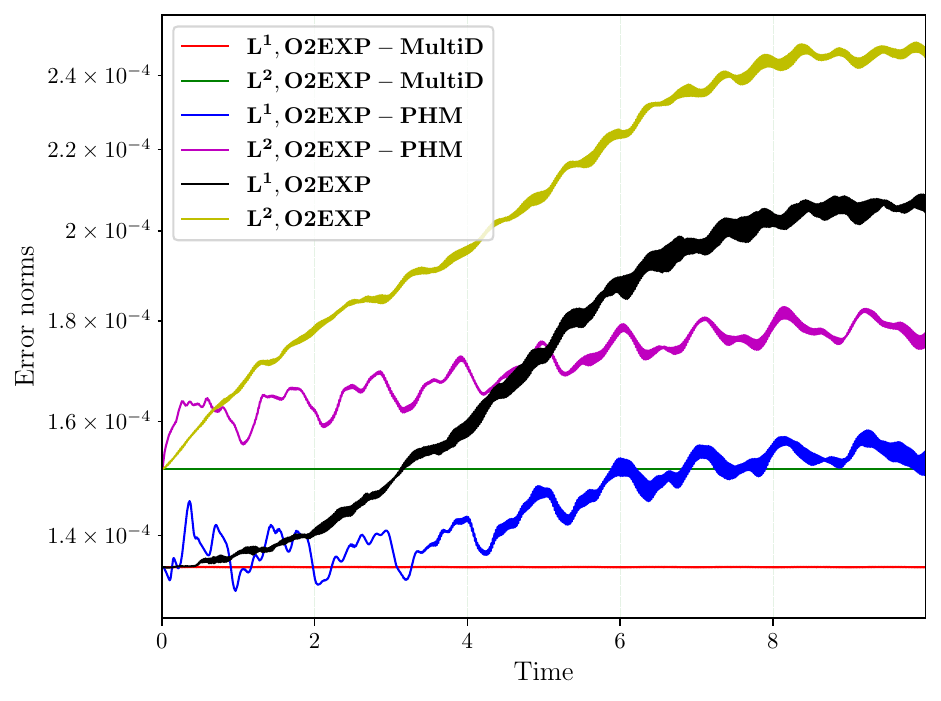}
		\label{fig:smooth2d_E_exp}}
	\subfigure[$\|\na\cdot\Eb\|_{1}^{I},$ and $\|\na\cdot\Eb\|_{2}^{I}$ errors, for IMEX schemes \textbf{O2IMEX-MultiD}, \textbf{O2IMEX-PHM} and \textbf{O2IMEX}.]{
	\includegraphics[width=2.9in, height=2.5in]{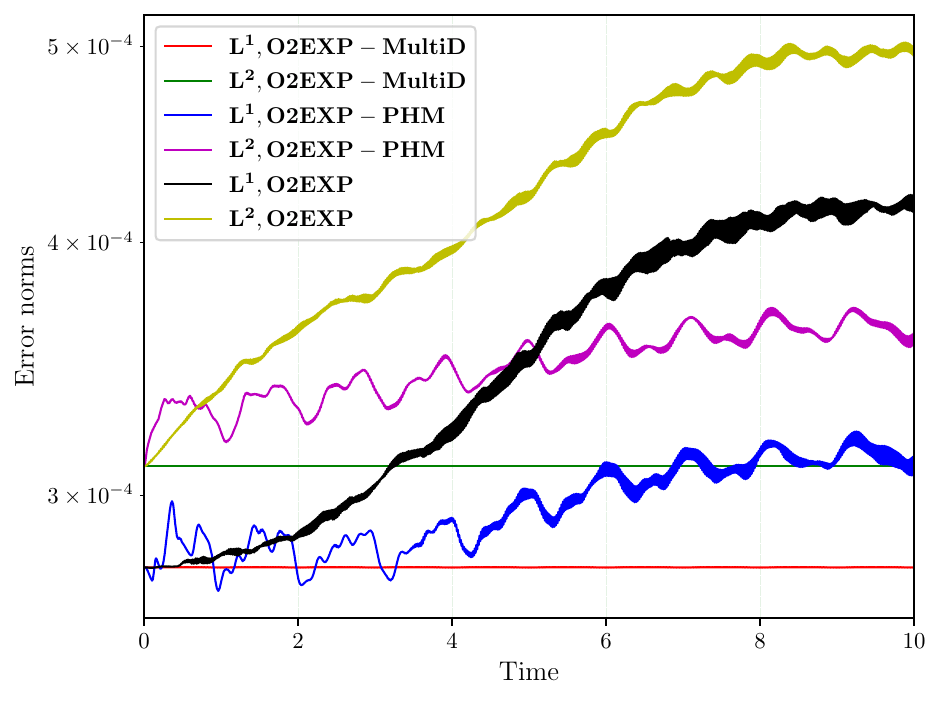}
	\label{fig:smooth2d_E_imp}}
		\caption{\nameref{test:smooth_2d}: Evolution of the divergence constraints errors of magnetic and electric fields for different schemes using $100 \times 100$ cells.}
		\label{fig:smooth2d_norms}
	\end{center}
\end{figure}

\subsubsection{Relativistic Orzag-Tang vortex} 
\label{test:2d_ot}
We consider the test cases from \cite{Balsara2016,bhoriya_entropydg_2023, Bhoriya2023}, which is motivated by the MHD test case proposed in \cite{Orszag1979}. To compare the results with those presented in \cite{Balsara2016}, we consider Maxwell's equations with scaled source terms. The computational domain is taken to be $[0,1] \times [0,1]$ with periodic boundary conditions. The initial setup is given by,
\begin{align*}
	\begin{pmatrix*}[c]
		\rho_i \\ u_{i}^x \\ u_{i}^y \\ p_i
	\end{pmatrix*} =
	\begin{pmatrix*}[c]
		\frac{25}{72 \pi} \\ - \frac{\sin(2 \pi y)}{2} \\ \frac{\sin(2 \pi x)}{2} \\ \frac{5}{24 \pi}
	\end{pmatrix*}, \qquad
	\begin{pmatrix*}[c]
		\rho_e \\u_{e}^x \\ u_{e}^y \\ p_e
	\end{pmatrix*} =
	\begin{pmatrix*}[c]
		\frac{25}{72 \pi} \\ - \frac{\sin(2 \pi y)}{2} \\  \frac{\sin(2 \pi x)}{2} \\  \frac{5}{24 \pi}
	\end{pmatrix*}, \qquad
	\begin{pmatrix*}[c]
		B_x \\ B_y
	\end{pmatrix*} =
	\begin{pmatrix*}[c]
		-\sin(2 \pi y) \\  \sin (4 \pi x)
	\end{pmatrix*}
\end{align*}
The expression for the initial electric field is derived using $-\mathbf{u}_i \times \mathbf{B}$. All the remaining variables are assumed to be vanishing. We take $\gamma_i=\gamma_e = 5/3$ and use the charge to mass ratios of $r_i=-r_e=10^3/\sqrt{4 \pi}$, which corresponds to the plasma skin depth of $3.0 \times 10^{-3}$, approximately. Simulations are performed till time $t = 1.0$ with $200\times 200$ cells.

Results are presented in Figure \ref{fig:ot}. Figures \ref{fig:ot_o2_exp_multiD_rho}-\ref{fig:ot_o2_exp_multiD_MagB}, contains the plots of total density, total pressure, Lorentz factor, and $|\Bb|$ using \mde~scheme. Figures  \ref{fig:ot_o2_imp_multiD_rho}-\ref{fig:ot_o2_imp_multiD_MagB}, contain the plots of total density, total pressure, Lorentz factor and $|\Bb|$ using \mdi~scheme. We observe that both schemes have the same performance and are able to capture the different solution structures. Furthermore, results match with those presented in \cite{Balsara2016}.

In Figures \ref{fig:ot_B_norm_exp_es_Rus}-\ref{fig:ot_E_norm_imp_es_Rus}, we have plotted the errors in electrogmagnetic constraints for \mde~and \mdi~schemesa and also compared these schemes with \ote, \oti,\phme~and \phmi~schemes. We observe that for the magnetic field divergence constraint, the proposed schemes (\mde~and \mdi) keep the errors to the machine's precision. Whereas, the other schemes have high divergence errors. For the electric field divergence constraint error, both PHM schemes and proposed schemes have similar performance but are better than the \ote~and\oti~schemes, which have errors increasing with time.

\begin{figure}[!htbp]
	\begin{center}
		\subfigure[Plot of total density $(\rho_i +\rho_e)$ for \mde.]{
			\includegraphics[width=1.4in, height=1.25in]{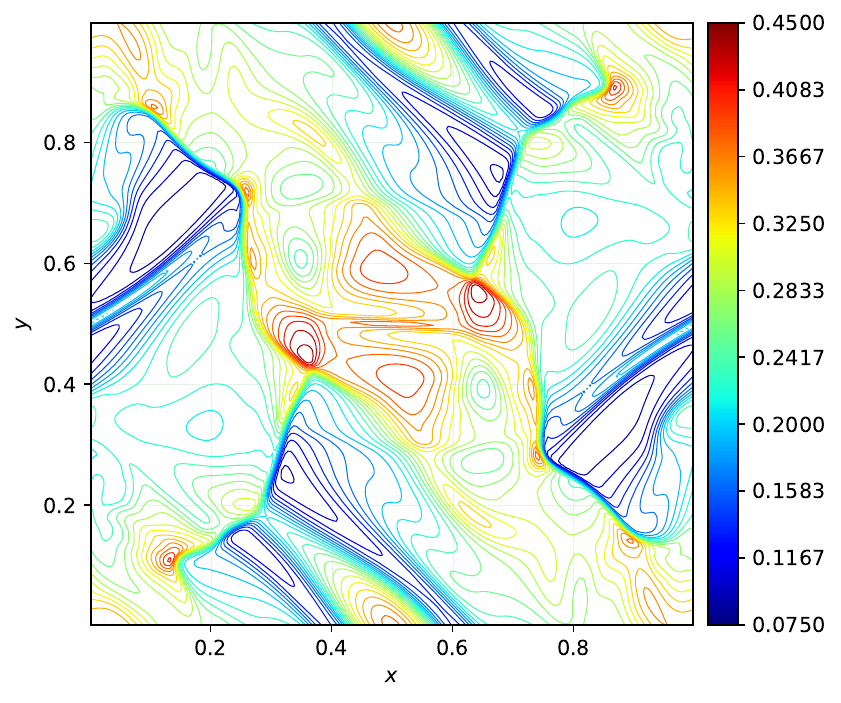}
			\label{fig:ot_o2_exp_multiD_rho}}
		\subfigure[Plot of total pressure $(p_{i} + p_e)$ for \mde]{
			\includegraphics[width=1.4in, height=1.25in]{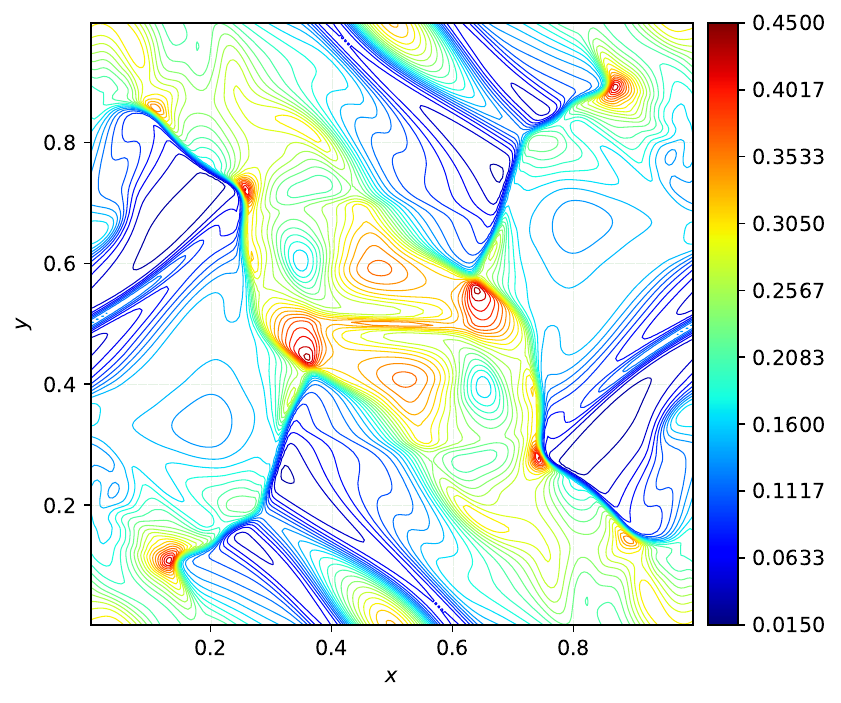}
			\label{fig:ot_o2_exp_multiD_p}}
		\subfigure[Plot of ion Lorentz factor $\Gamma_i$ for \mde]{
			\includegraphics[width=1.4in, height=1.25in]{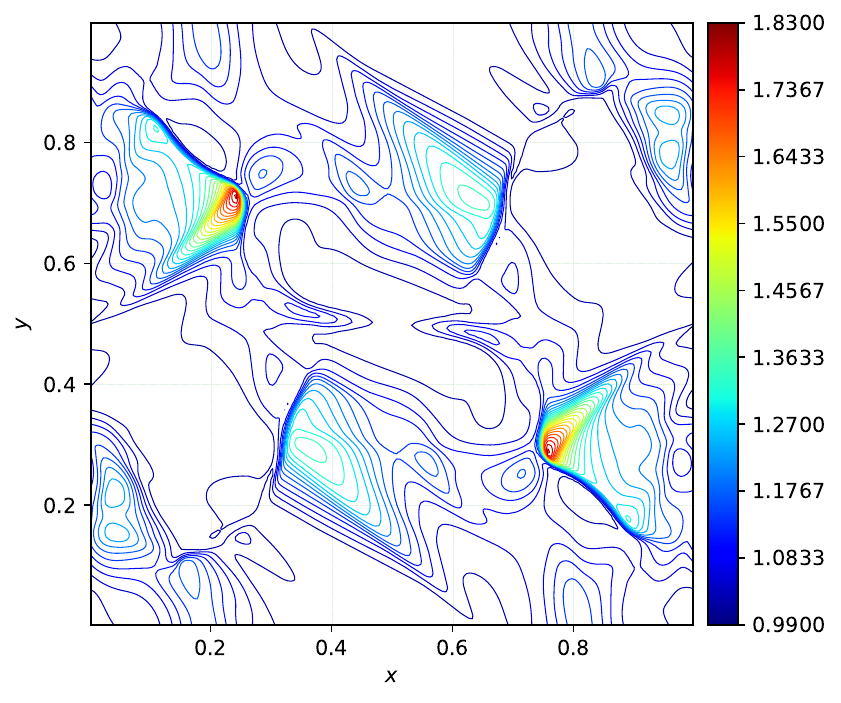}
			\label{fig:ot_o2_exp_multiD_lorentz}}
		\subfigure[Plot of magnitude of the Magnetic field, $|\mathbf{B}|$ for \mde]{
			\includegraphics[width=1.4in, height=1.25in]{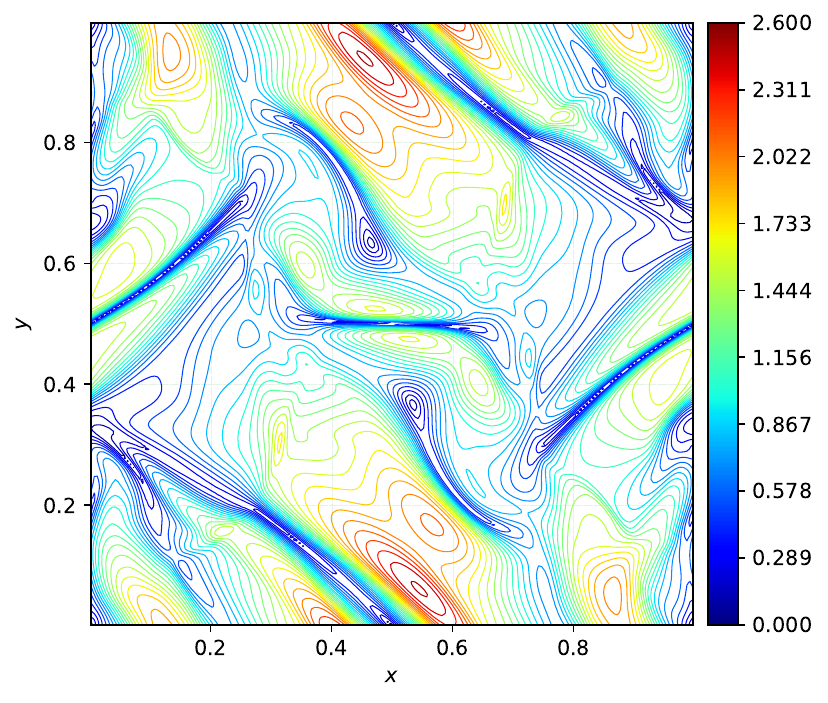}
			\label{fig:ot_o2_exp_multiD_MagB}}
 		\subfigure[Plot of total density $(\rho_i +\rho_e)$ for \mdi]{
			\includegraphics[width=1.4in, height=1.25in]{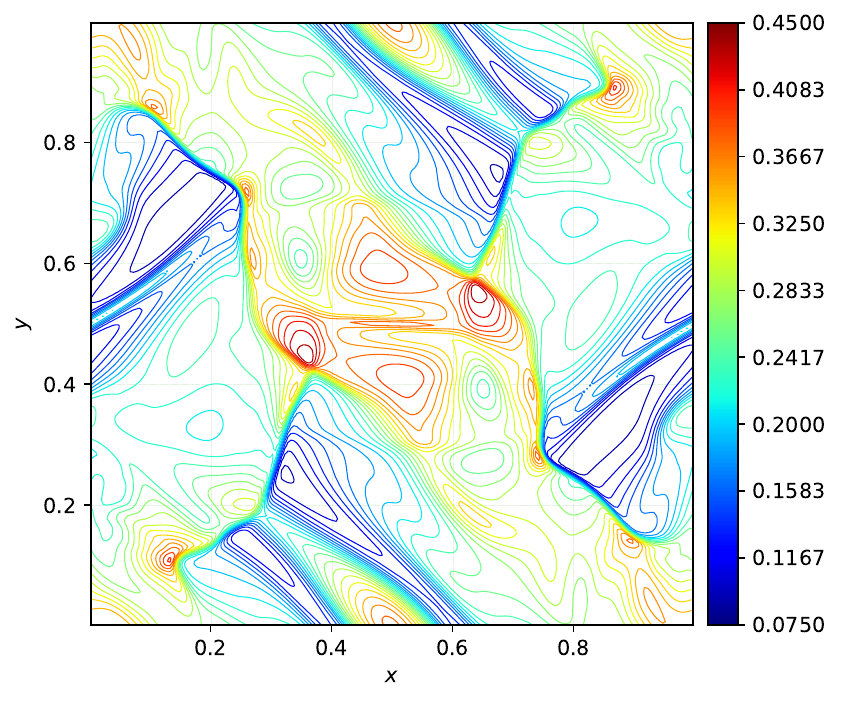}
			\label{fig:ot_o2_imp_multiD_rho}}
		\subfigure[Plot of total pressure $(p_{i} + p_e)$ for \mdi.]{
			\includegraphics[width=1.4in, height=1.25in]{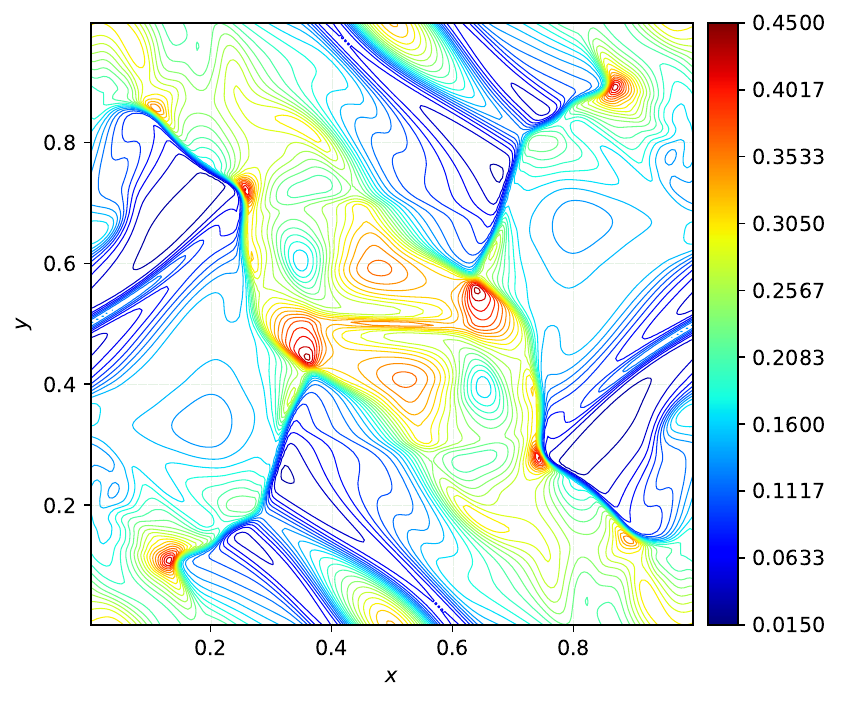}
			\label{fig:ot_o2_imp_multiD_p}}
		\subfigure[Plot of ion Lorentz factor $\Gamma_i$ for \mdi.]{
			\includegraphics[width=1.4in, height=1.25in]{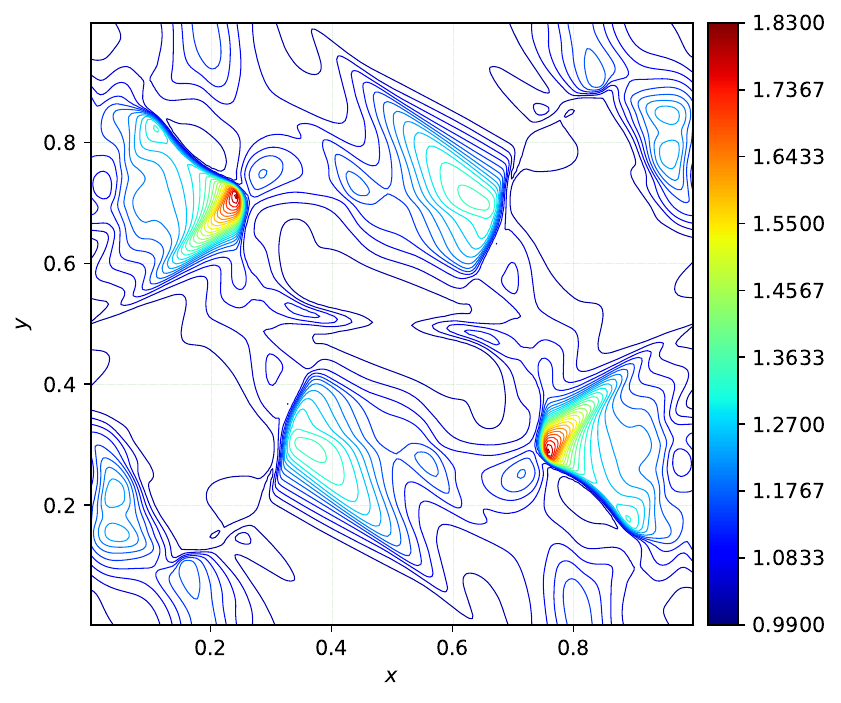}
			\label{fig:ot_o2_imp_multiD_lorentz}}
		\subfigure[Plot of magnitude of the Magnetic field, $|\mathbf{B}|$ for \mdi.]{
			\includegraphics[width=1.4in, height=1.25in]{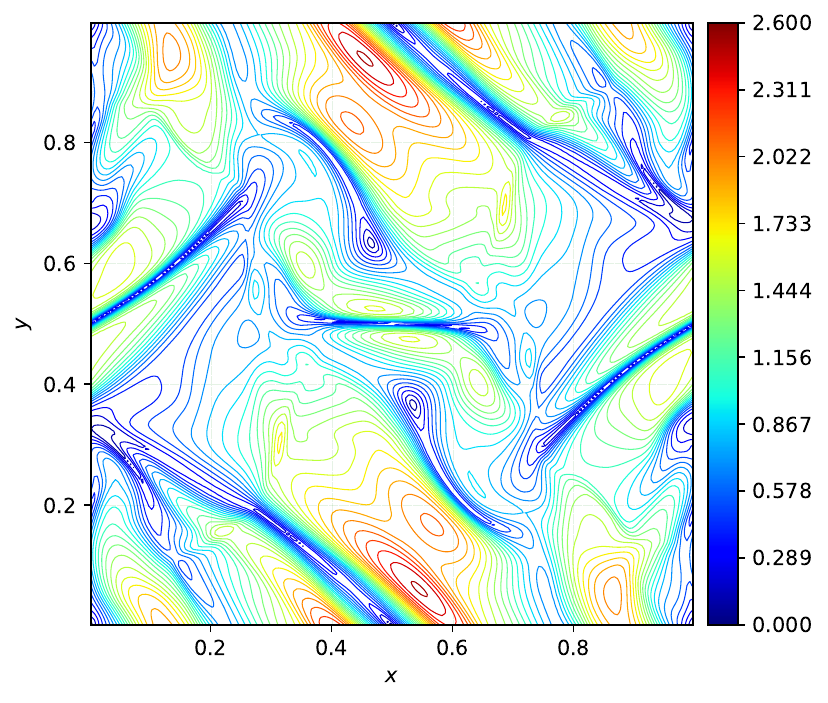}
			\label{fig:ot_o2_imp_multiD_MagB}}
	\subfigure[$\|\na \cdot \mathbf{B}^n\|_1$,  and $\|\na \cdot \mathbf{B}^n\|_2$  errors for explicit schemes.]{
	\includegraphics[width=1.4in, height=1.25in]{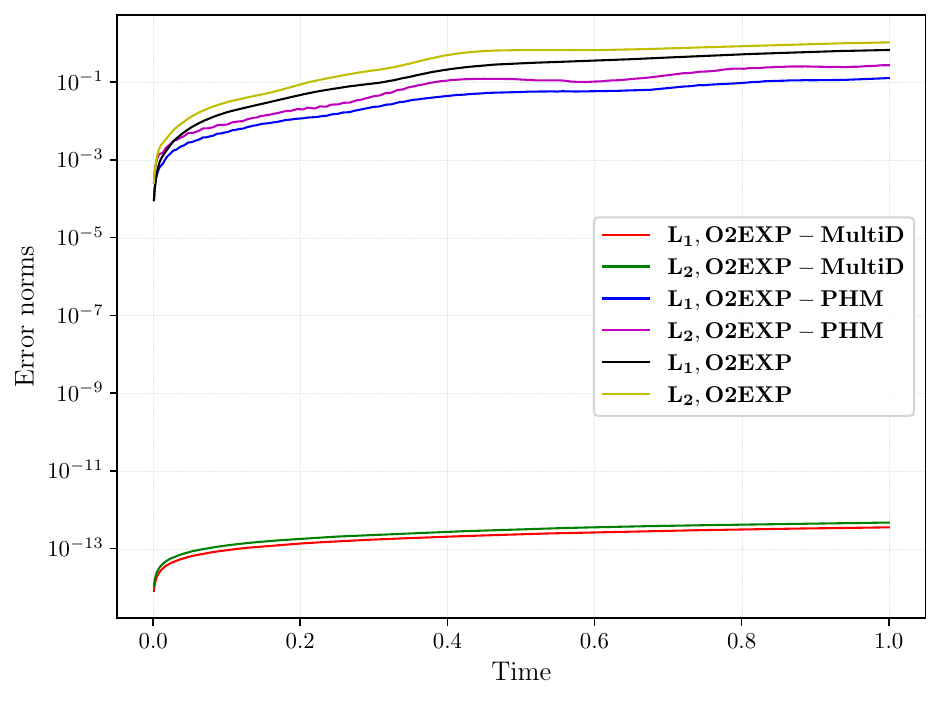}
	\label{fig:ot_B_norm_exp_es_Rus}}
\subfigure[$\|\na \cdot \mathbf{B}^n\|_1$,  and $\|\na \cdot \mathbf{B}^n\|_2$  errors for IMEX schemes.]{
	\includegraphics[width=1.4in, height=1.25in]{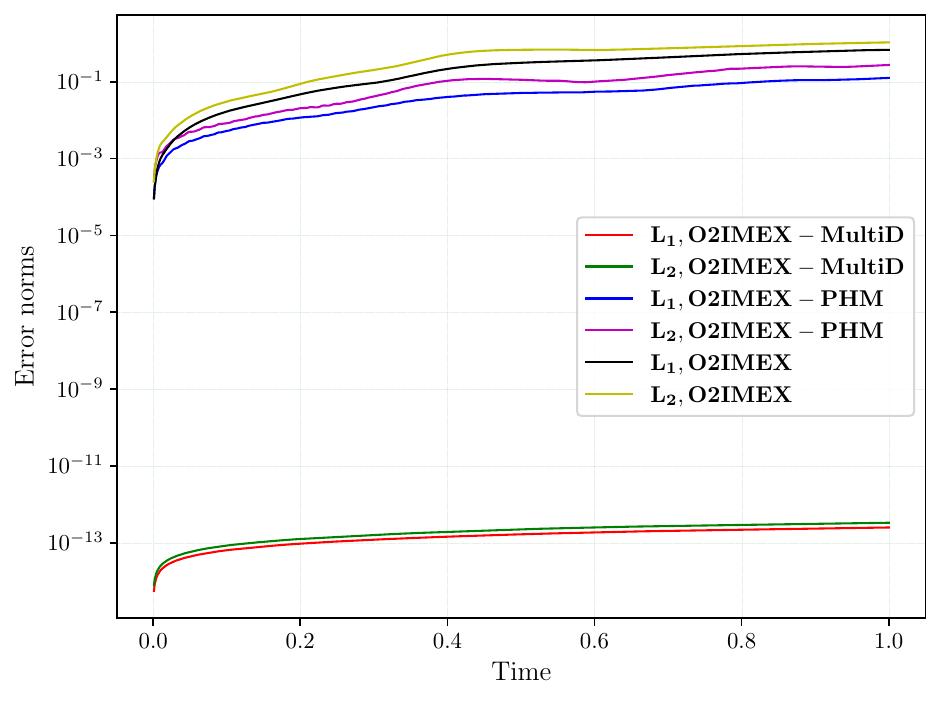}
	\label{fig:ot_B_norm_imp_es_Rus}}
\subfigure[$\|\na\cdot\Eb\|_{1}^{E},$ and $\|\na\cdot\Eb\|_{2}^{E}$ errors, for explicit schemes.]{
	\includegraphics[width=1.4in, height=1.25in]{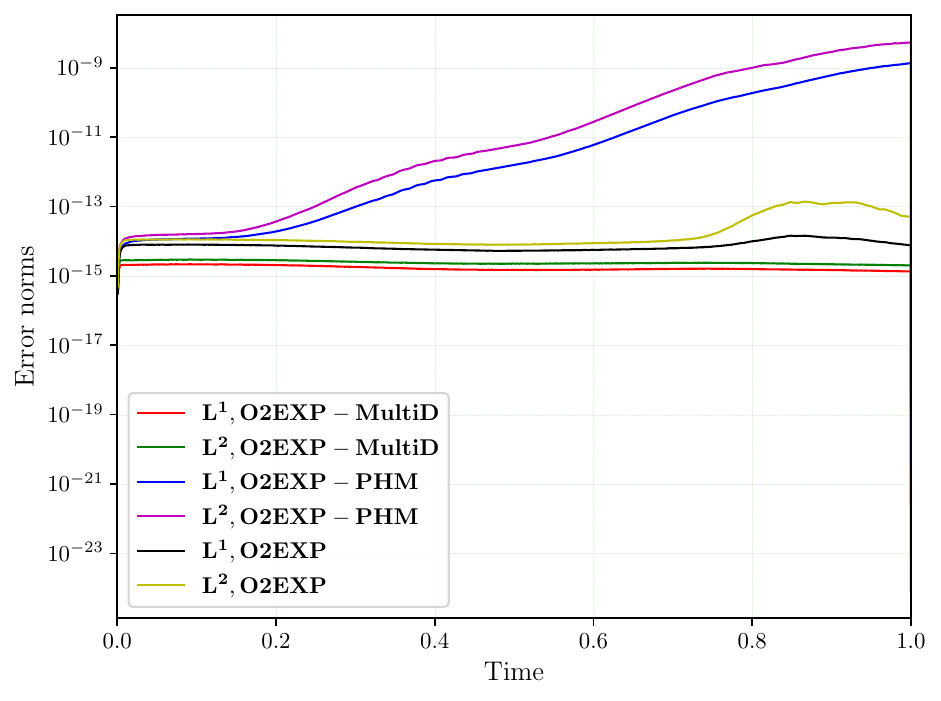}
	\label{fig:ot_E_norm_exp_es_Rus}}
\subfigure[$\|\na\cdot\Eb\|_{1}^{I},$ and $\|\na\cdot\Eb\|_{2}^{I}$ errors, for IMEX schemes.]{
	\includegraphics[width=1.4in, height=1.25in]{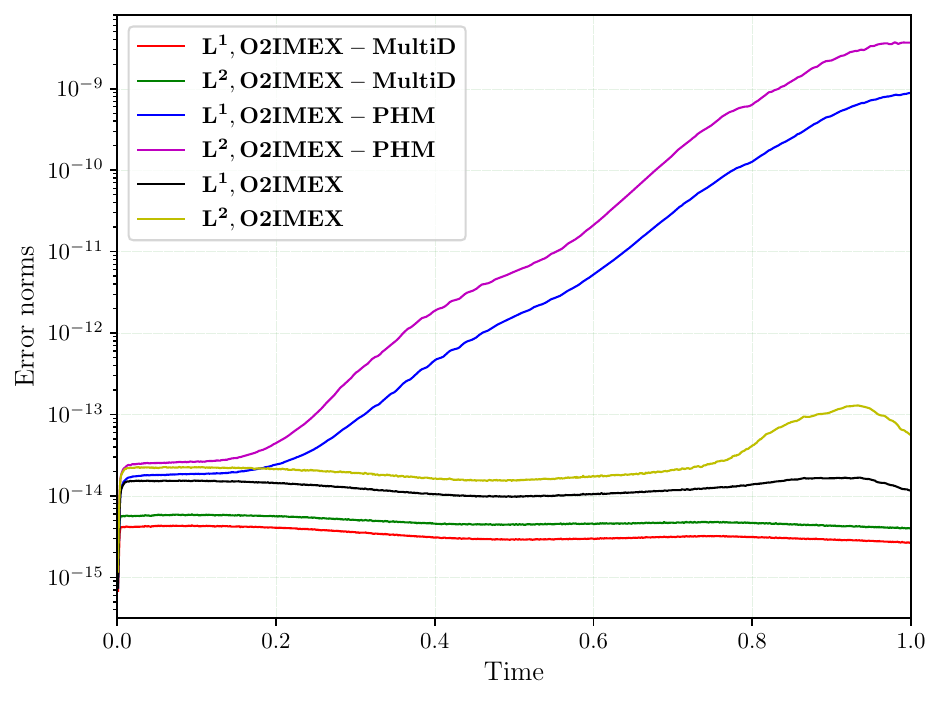}
	\label{fig:ot_E_norm_imp_es_Rus}}
		\caption{\nameref{test:2d_ot}: Figures \ref{fig:ot_o2_exp_multiD_rho}-\ref{fig:ot_o2_imp_multiD_MagB} contains plots of total density, total pressure, Ion Lorentz factor, and magnitude of the magnetic field $|\mathbf{B}|$ using the \textbf{O2EXP-MultiD} and \mdi~scheme with $200\times 200$ cells at time $t=1.0$. Figures \ref{fig:ot_B_norm_exp_es_Rus}-\ref{fig:ot_E_norm_imp_es_Rus} contains evolution of the divergence constraints errors of magnetic and electric fields for all the schemes.}
		\label{fig:ot}		
	\end{center}
\end{figure}

\subsubsection{Relativistic two-fluid blast problem} 
\label{test:2d_blast}
We consider the test case from \cite{Amano2016, Balsara2016, Bhoriya2023,bhoriya_entropydg_2023}, which is generalized from a similar RMHD test case from ~\cite{Komissarov1999}. The domain of the problem considered is $[-6,6]\times[-6,6]$ and we consider Neumann boundary conditions on all boundaries. Let us define, $\rho_{in}=10^{-2}$, $p_{in}=1.0$, $\rho_{out}=10^{-4}$ and $p_{out}=5 \times 10^{-4}$. Also, consider the radial distance $r=\sqrt{x^2+y^2}$ from the origin. Inside the disc $r<0.8$, we consider $\rho_i=\rho_e=0.5 \times \rho_{in}$ and  $p_i=p_e=0.5 \times p_{in}$. Outside the disc $r>1$, we consider $\rho_i=\rho_e=0.5 \times \rho_{out}$ and $p_i=p_e=0.5 \times p_{out}$. In the region $0.8 \le r \le 1.0$, the densities and pressures are defined using a linear profile, by imposing continuity at $r=0.8$ and $r=1.0$. The magnetic field has been initialized in the $x$-direction only, as $B_x=B_0$. All other variables are assumed to be vanishing. We consider the charge-to-mass ratios to be $r_i=-r_e=10^3$ and  $\gamma_i=\gamma_e=4/3$. We consider $B_0=0.1$ (weakly magnetized medium) and $B_0=1.0$ (strongly magnetized medium). The simulations are performed till the final time $t=4.0.$

The numerical solutions are presented in Figures \ref{fig:blast_o2_exp_0.1_multiD}, \ref{fig:blast_o2_imp_0.1_multiD}, \ref{fig:blast_o2_exp_1.0_multiD}, and  \ref{fig:blast_o2_imp_1.0_multiD}. We have plotted $\log_{10}(\rho_i+\rho_e)$, $\log_{10}(p_i+p_e)$, $\Gamma_i$ and $\dfrac{|\mathbf{B}|^2}{2}$ variables for the scheme \mde~ and \mdi~for weakly and strongly magnetized mediums. For both the mediums, both explicit and IMEX schemes are able to capture all the waves and have similar performance.

In Figures \ref{fig:blast_div_0p1} and \ref{fig:blast_div_1p0}, we have presented the error norms for the divergence constraints for weakly and strongly magnetized medium. In both cases, we note that the magnetic and electric field divergence constraints are satisfied by the proposed schemes up to the machine's precision. For the other schemes, the electric field divergence constraint is satisfied, up to the machine precision, but the divergence of the magnetic field is very large. This demonstrates the superiority of the proposed schemes in ensuring both constraints.

\begin{figure}[!htbp]
	\begin{center}
		\subfigure[$\log_{10}(\rho_i+\rho_e)$.]{
			\includegraphics[width=1.4in, height=1.25in]{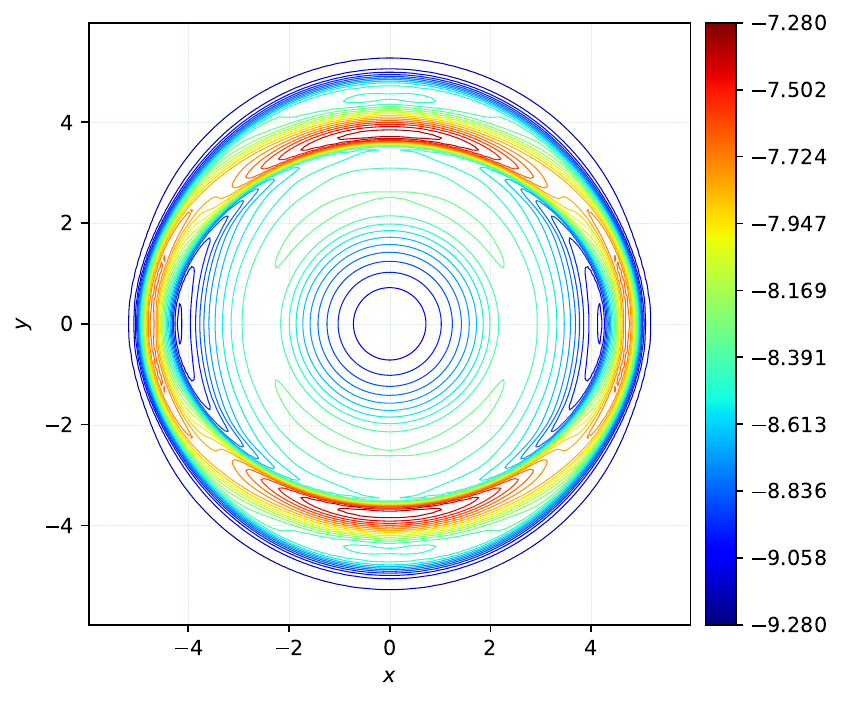}
			\label{fig:blast_o2_exp_0.1_multiD_logrho}}
		\subfigure[$\log_{10}(p_i+p_e)$.]{
			\includegraphics[width=1.4in, height=1.25in]{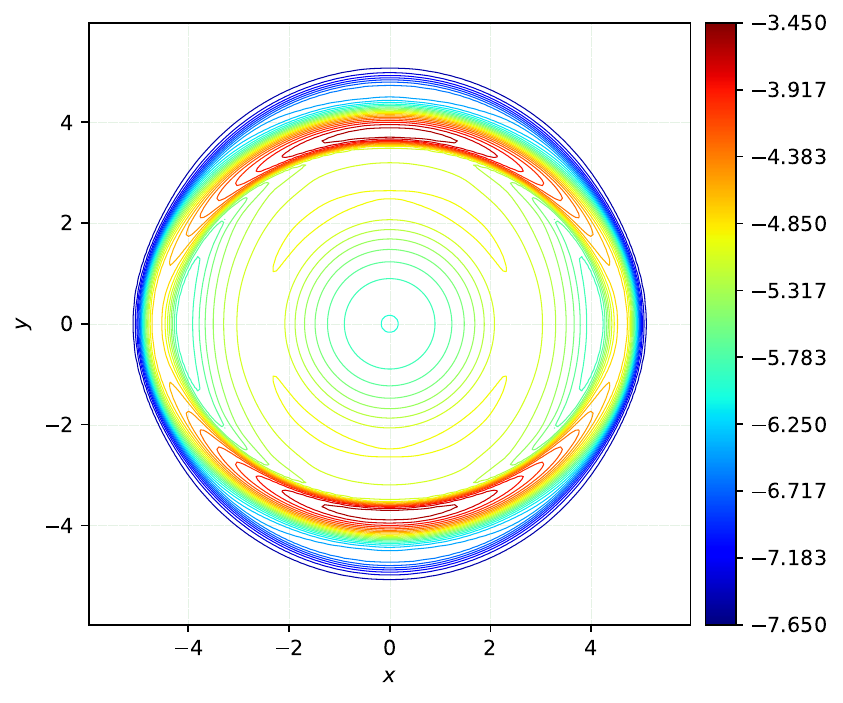}
			\label{fig:blast_o2_exp_0.1_multiD_logp}}
		\subfigure[$\Gamma_i$.]{
			\includegraphics[width=1.4in, height=1.25in]{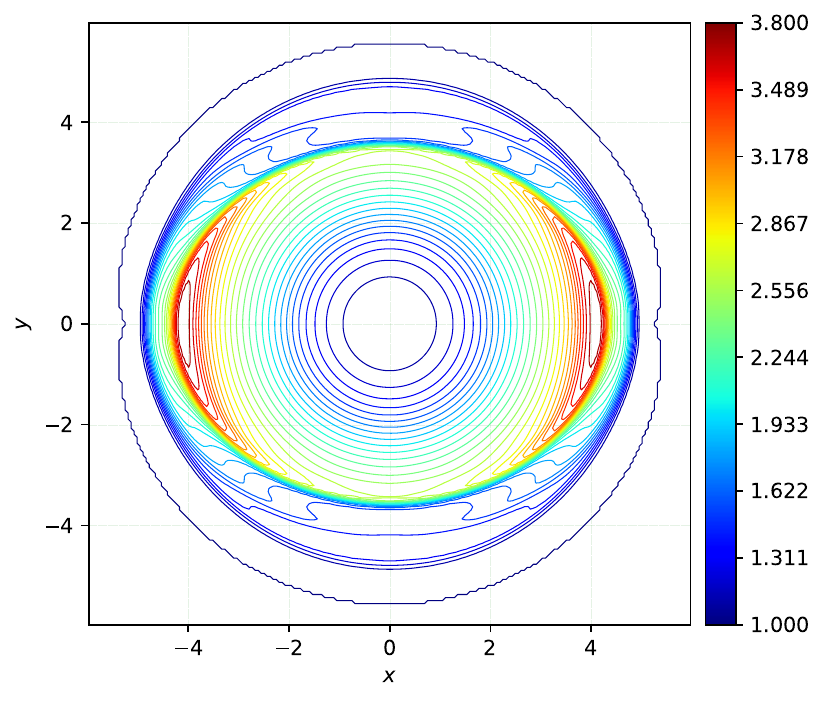}
			\label{fig:blast_o2_exp_0.1_multiD_lorentz}}
		\subfigure[$\dfrac{|\mathbf{B}|^2}{2}$]{
			\includegraphics[width=1.4in, height=1.25in]{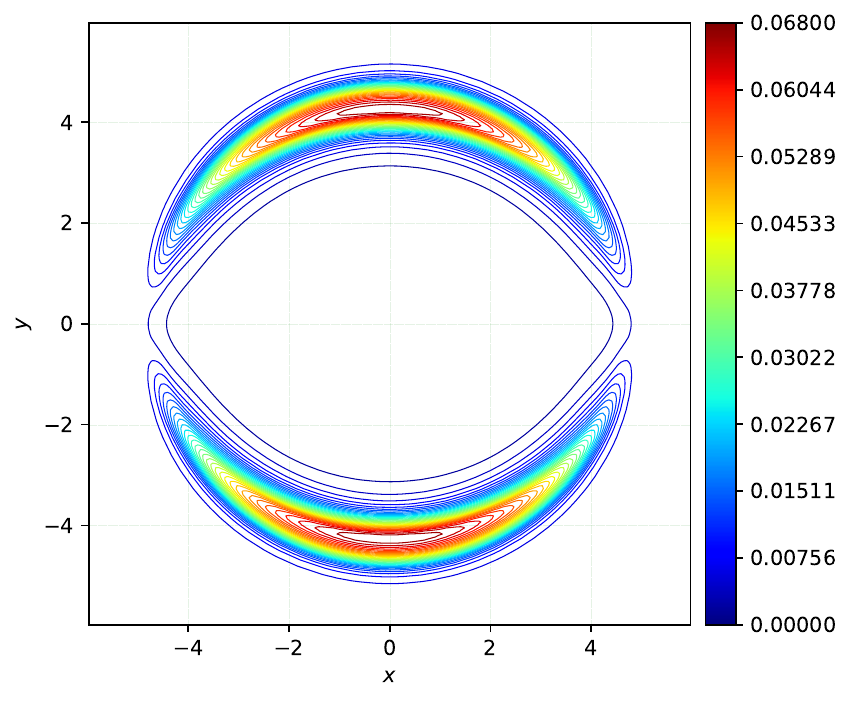}
			\label{fig:blast_o3_exp_0.1_multiD_MagBsqby2}}
		\caption{\nameref{test:2d_blast}: Plots of $\log_{10}(\rho_i+\rho_e)$, $\log_{10}(p_i+p_e)$, $\Gamma_i$ and  $\dfrac{|\mathbf{B}|^2}{2}$ for the weakly magnetized medium with $B_0=0.1$, using \textbf{O2EXP-MultiD} scheme on $200\times 200$ cells.}
		\label{fig:blast_o2_exp_0.1_multiD}
	\end{center}
\end{figure}

\begin{figure}[!htbp]
	\begin{center}
		\subfigure[Plot of $\log_{10}(\rho_i+\rho_e)$.]{
			\includegraphics[width=1.4in, height=1.25in]{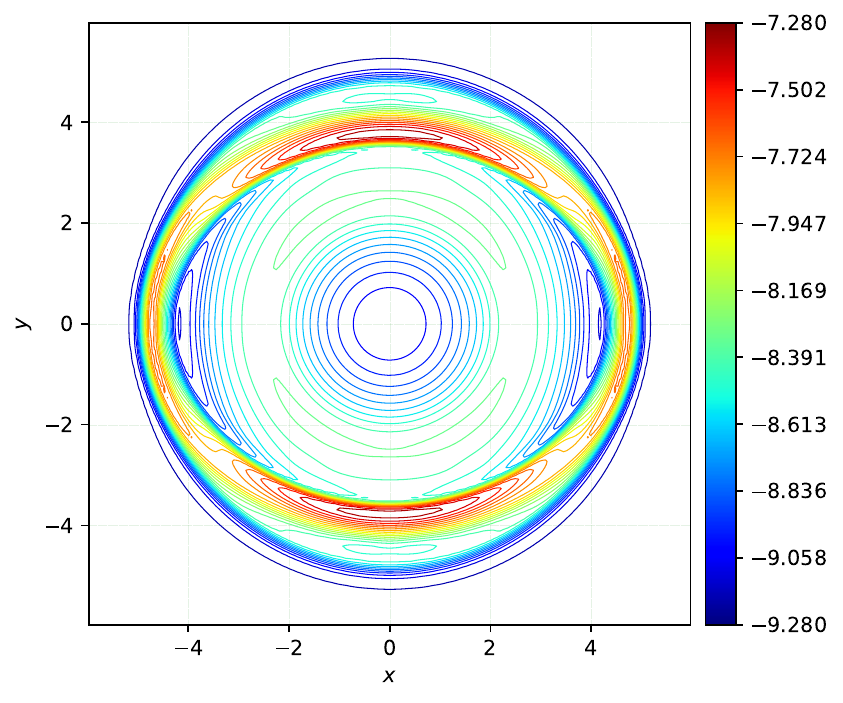}
			\label{fig:blast_o2_imp_0.1_multiD_logrho}}
		\subfigure[$\log_{10}(p_i+p_e)$.]{
			\includegraphics[width=1.4in, height=1.25in]{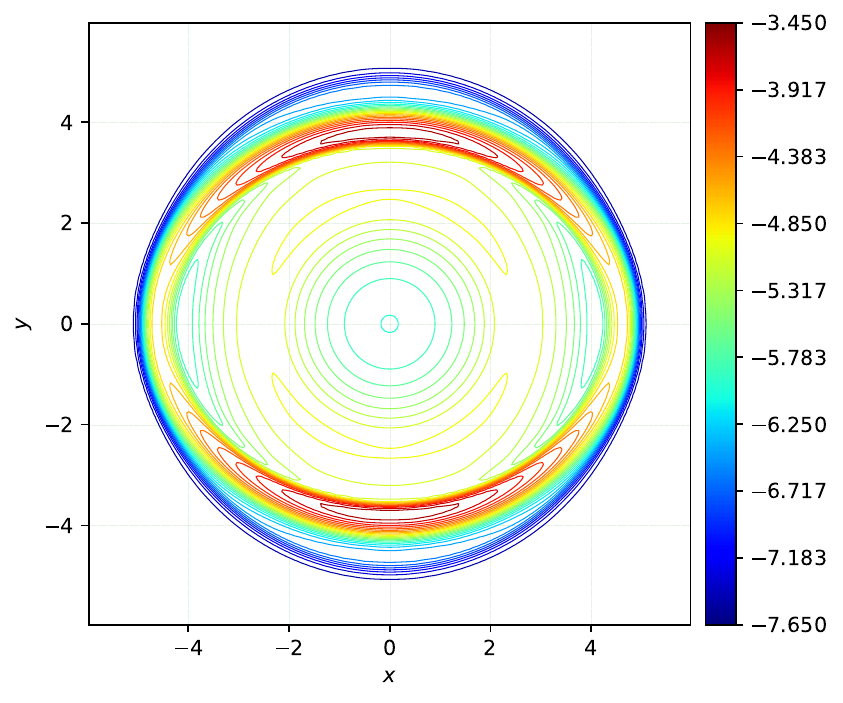}
			\label{fig:blast_o2_imp_0.1_multiD_logp}}
		\subfigure[$\Gamma_i$.]{
			\includegraphics[width=1.4in, height=1.25in]{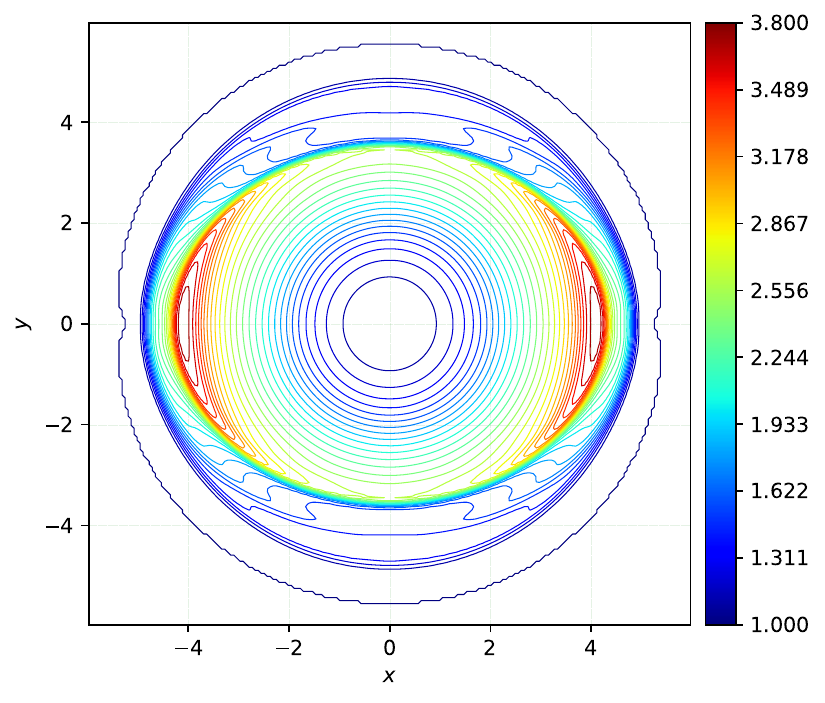}
			\label{fig:blast_o2_imp_0.1_multiD_lorentz}}
		\subfigure[$\dfrac{|\mathbf{B}|^2}{2}$.]{
			\includegraphics[width=1.4in, height=1.25in]{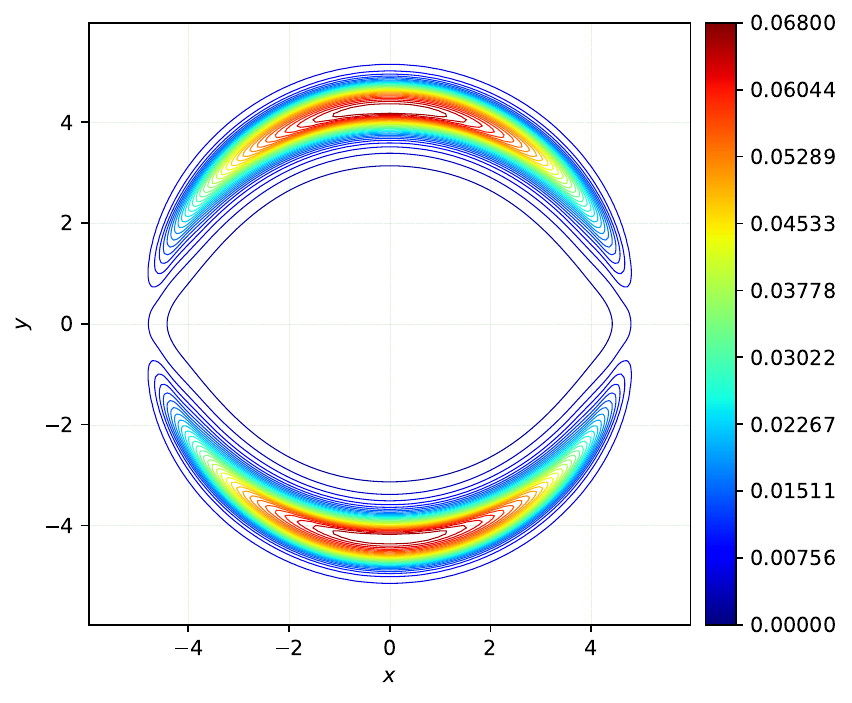}
			\label{fig:blast_o3_imp_0.1_multiD_MagBsqby2}}
		\caption{\nameref{test:2d_blast}: Plots of $\log_{10}(\rho_i+\rho_e)$, $\log_{10}(p_i+p_e)$, $\Gamma_i$ and  $\dfrac{|\mathbf{B}|^2}{2}$ for the weakly magnetized medium with $B_0=0.1$, using \textbf{O2IMP-MultiD} scheme on $200\times 200$ cells.}
		\label{fig:blast_o2_imp_0.1_multiD}
	\end{center}
\end{figure}
\begin{figure}[!htbp]
	\begin{center}
		\subfigure[$\log_{10}(\rho_i+\rho_e)$.]{
			\includegraphics[width=1.4in, height=1.25in]{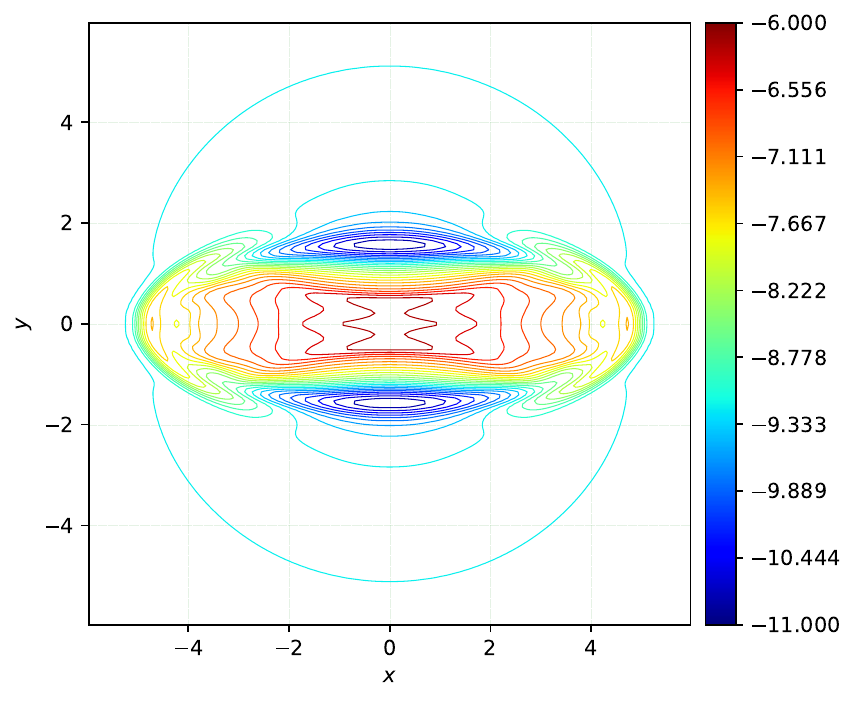}
			\label{fig:blast_o2_exp_1.0_multiD_logrho}}
		\subfigure[$\log_{10}(p_i+p_e)$.]{
			\includegraphics[width=1.4in, height=1.25in]{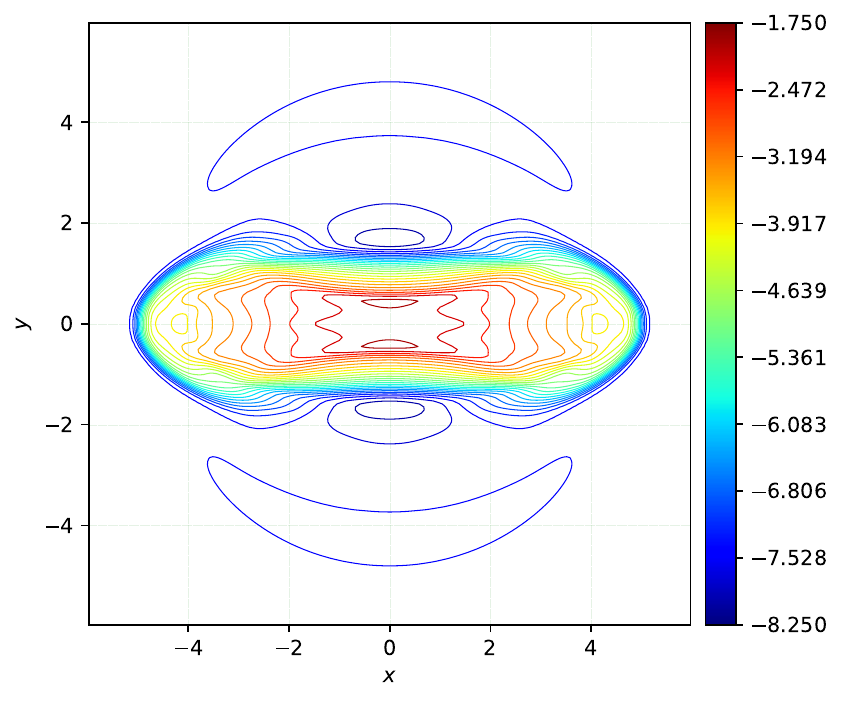}
			\label{fig:blast_o2_exp_1.0_multiD_logp}}
		\subfigure[$\Gamma_i$.]{
			\includegraphics[width=1.4in, height=1.25in]{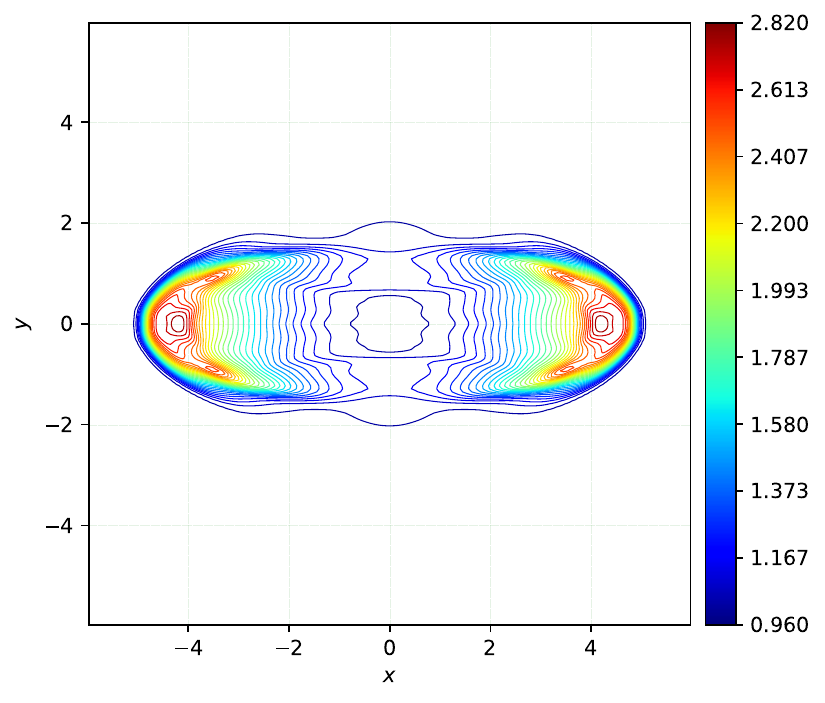}
			\label{fig:blast_o2_exp_1.0_multiD_lorentz}}
		\subfigure[$\dfrac{|\mathbf{B}|^2}{2}$.]{
			\includegraphics[width=1.4in, height=1.25in]{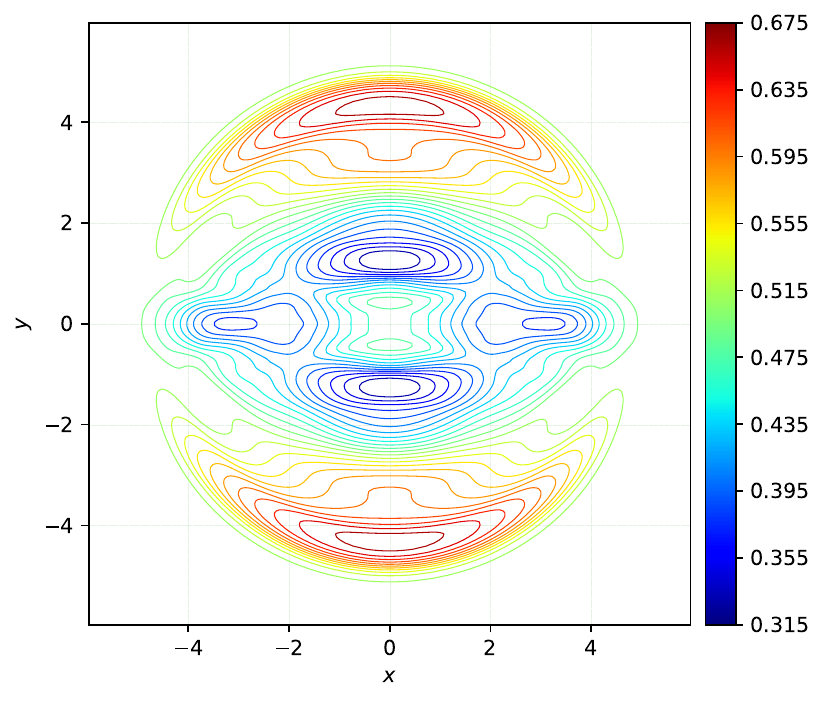}
			\label{fig:blast_o3_exp_1.0_multiD_MagBsqby2}}
		\caption{\nameref{test:2d_blast}: Plots of $\log_{10}(\rho_i+\rho_e)$, $\log_{10}(p_i+p_e)$, $\Gamma_i$ and  $\dfrac{|\mathbf{B}|^2}{2}$ for the strongly magnetized medium with $B_0=1.0$, using \textbf{O2EXP-MultiD} scheme on $200\times 200$ cells.}
		\label{fig:blast_o2_exp_1.0_multiD}
	\end{center}
\end{figure}

\begin{figure}[!htbp]
	\begin{center}
		\subfigure[Plot of $\log_{10}(\rho_i+\rho_e)$.]{
			\includegraphics[width=1.4in, height=1.25in]{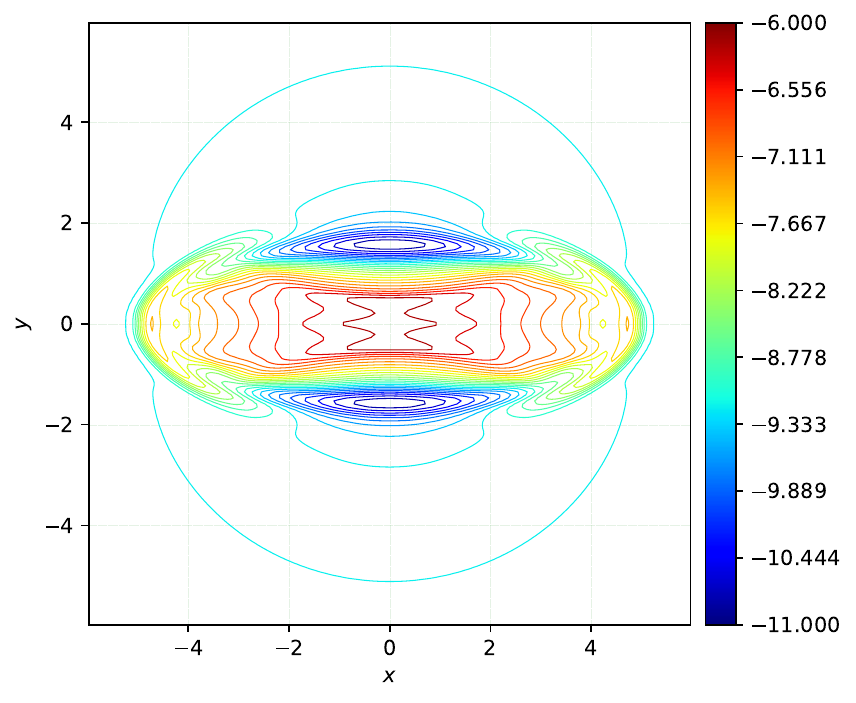}
			\label{fig:blast_o2_imp_1.0_multiD_logrho}}
		\subfigure[$\log_{10}(p_i+p_e)$.]{
			\includegraphics[width=1.4in, height=1.25in]{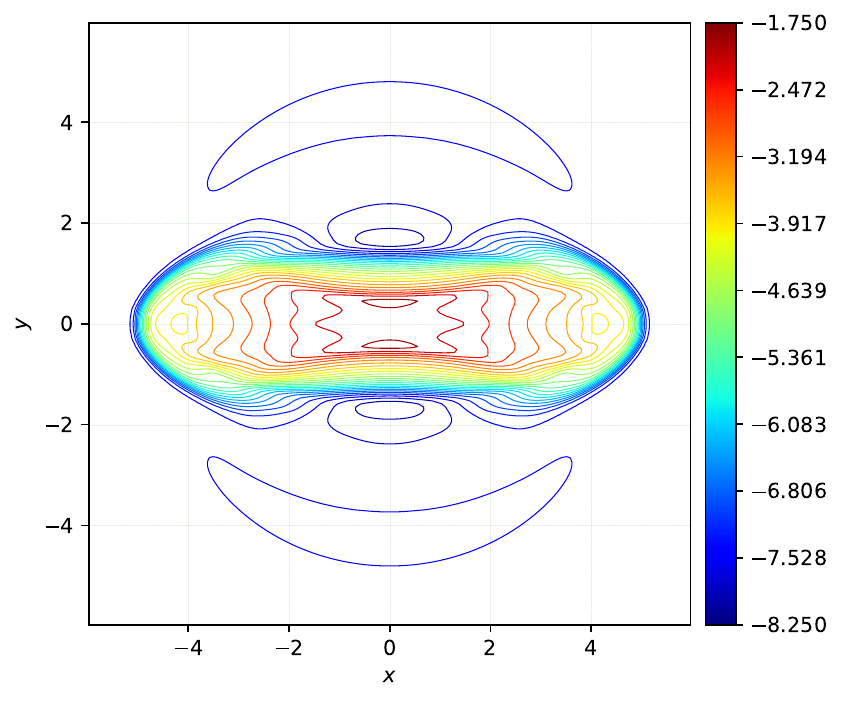}
			\label{fig:blast_o2_imp_1.0_multiD_logp}}
		\subfigure[$\Gamma_i$.]{
			\includegraphics[width=1.4in, height=1.25in]{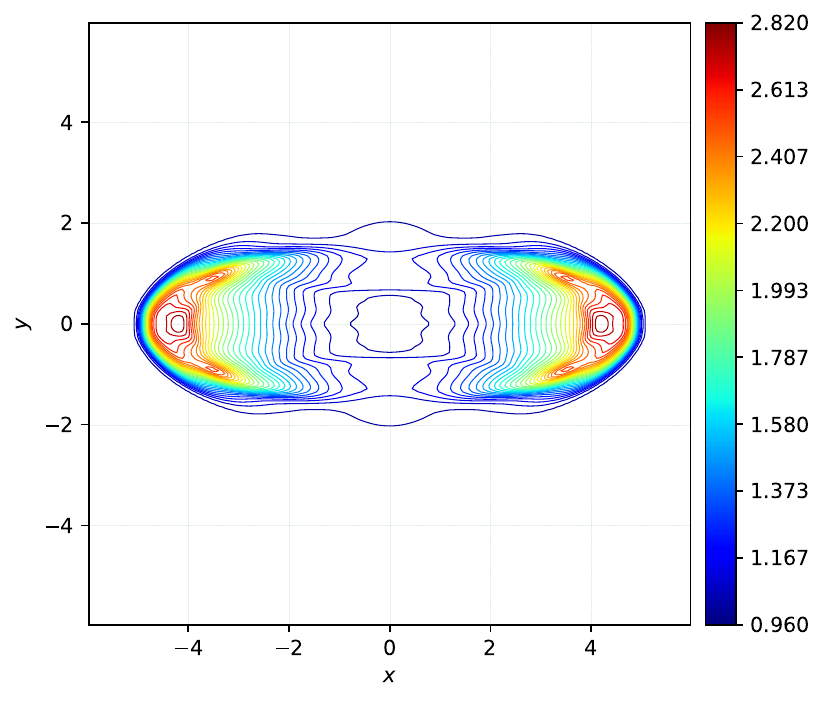}
			\label{fig:blast_o2_imp_1.0_multiD_lorentz}}
		\subfigure[$\dfrac{|\mathbf{B}|^2}{2}$.]{
			\includegraphics[width=1.4in, height=1.25in]{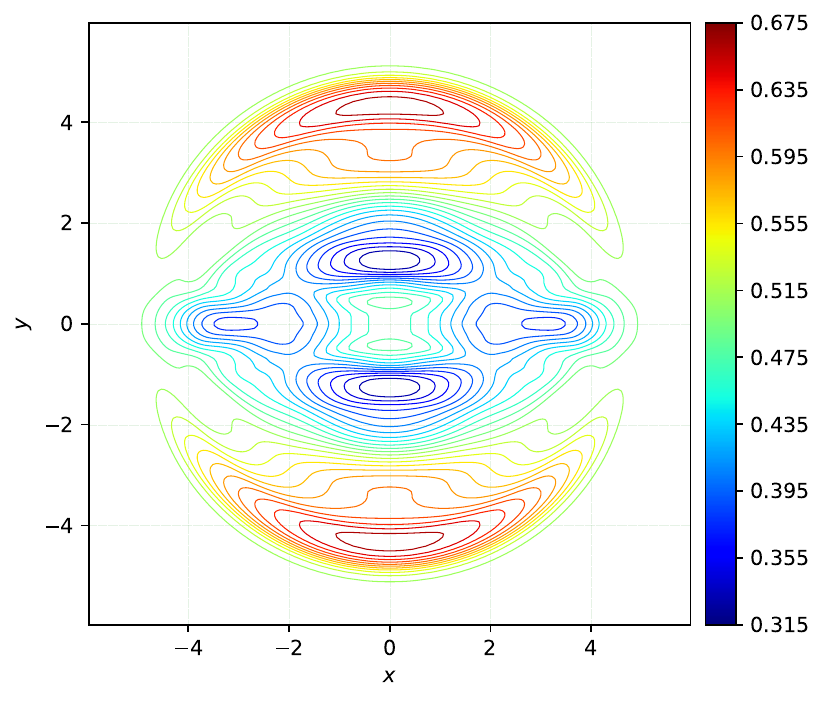}
			\label{fig:blast_o3_imp_1.0_multiD_MagBsqby2}}
		\caption{\nameref{test:2d_blast}: Plots of $\log_{10}(\rho_i+\rho_e)$, $\log_{10}(p_i+p_e)$, $\Gamma_i$ and  $\dfrac{|\mathbf{B}|^2}{2}$ for the strongly magnetized medium with $B_0=1.0$, using \textbf{O2IMP-MultiD} scheme on $200\times 200$ cells.}
		\label{fig:blast_o2_imp_1.0_multiD}
	\end{center}
\end{figure}

\begin{figure}[!htbp]
	\begin{center}
		\subfigure[$\|\na \cdot \mathbf{B}^n\|_1$,  and $\|\na \cdot \mathbf{B}^n\|_2$  errors for explicit schemes.]{
			\includegraphics[width=1.4in, height=1.25in]{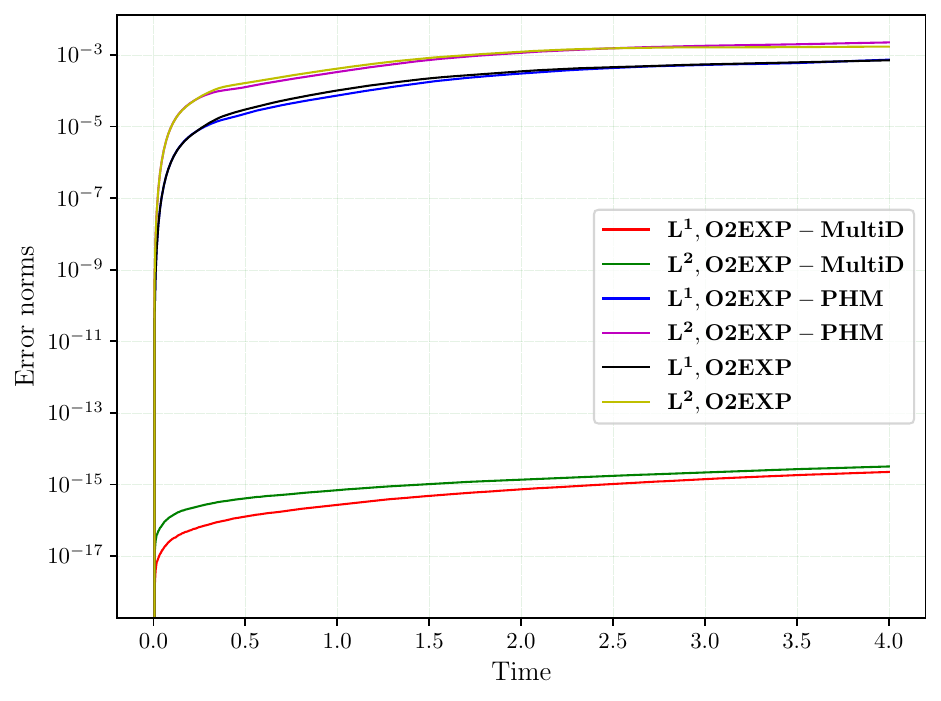}
			\label{fig:blast_B_norm_exp_es_Rus}}
		\subfigure[$\|\na \cdot \mathbf{B}^n\|_1$,  and $\|\na \cdot \mathbf{B}^n\|_2$  errors for IMEX schemes.]{
			\includegraphics[width=1.4in, height=1.25in]{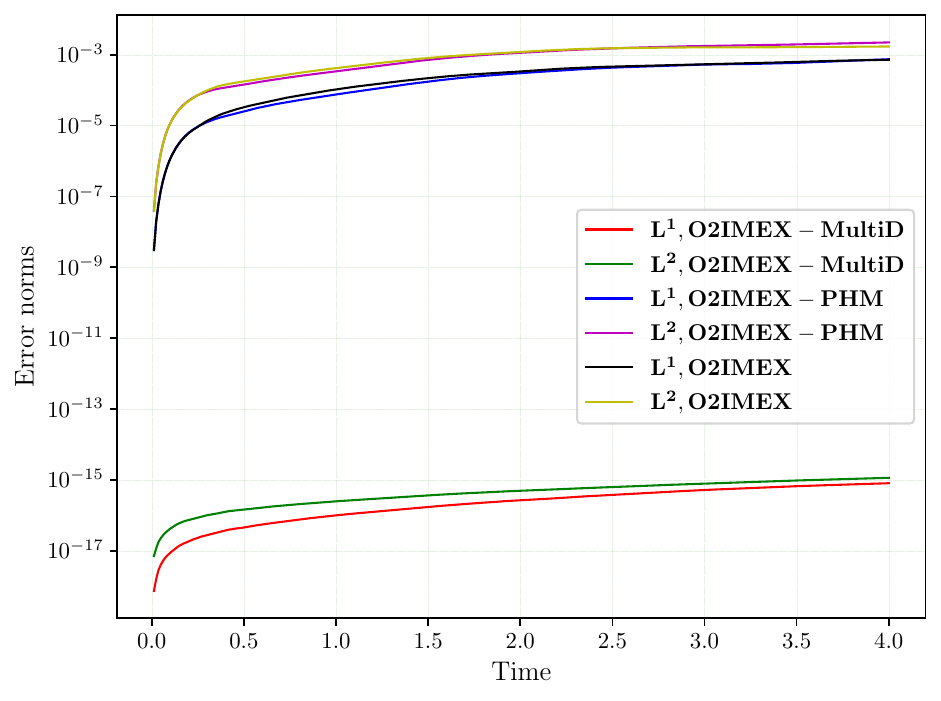}
			\label{fig:blast_B_norm_imp_es_Rus}}
			\subfigure[$\|\na\cdot\Eb\|_{1}^{E},$ and $\|\na\cdot\Eb\|_{2}^{E}$ errors, for explicit schemes.]{
			\includegraphics[width=1.4in, height=1.25in]{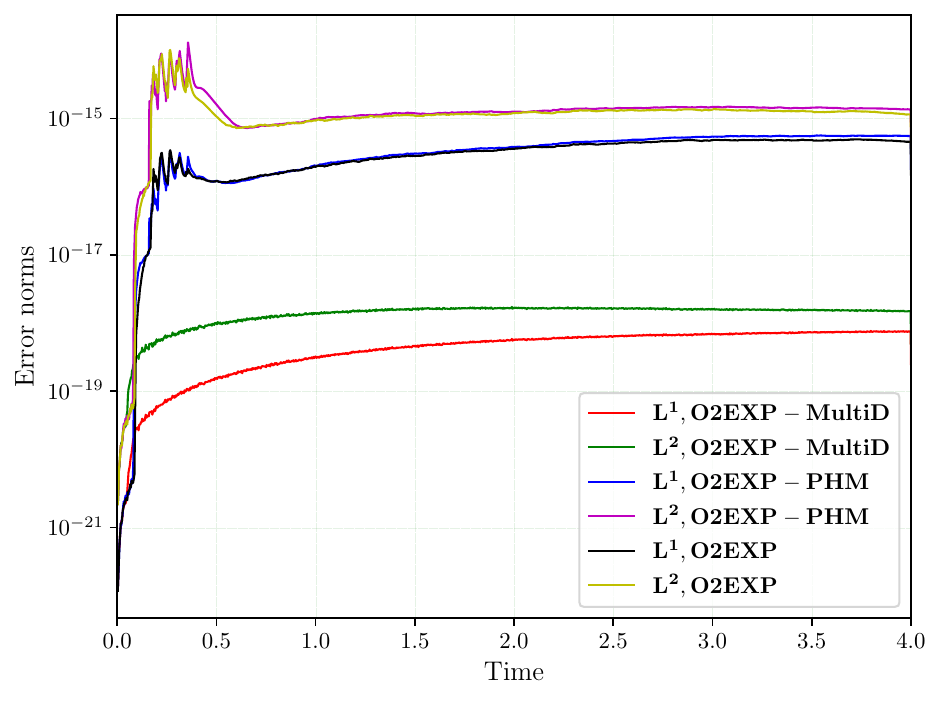}
			\label{fig:blast_E_norm_exp_es_Rus}}
		\subfigure[$\|\na\cdot\Eb\|_{1}^{I},$ and $\|\na\cdot\Eb\|_{2}^{I}$ errors, for IMEX schemes.]{
			\includegraphics[width=1.4in, height=1.25in]{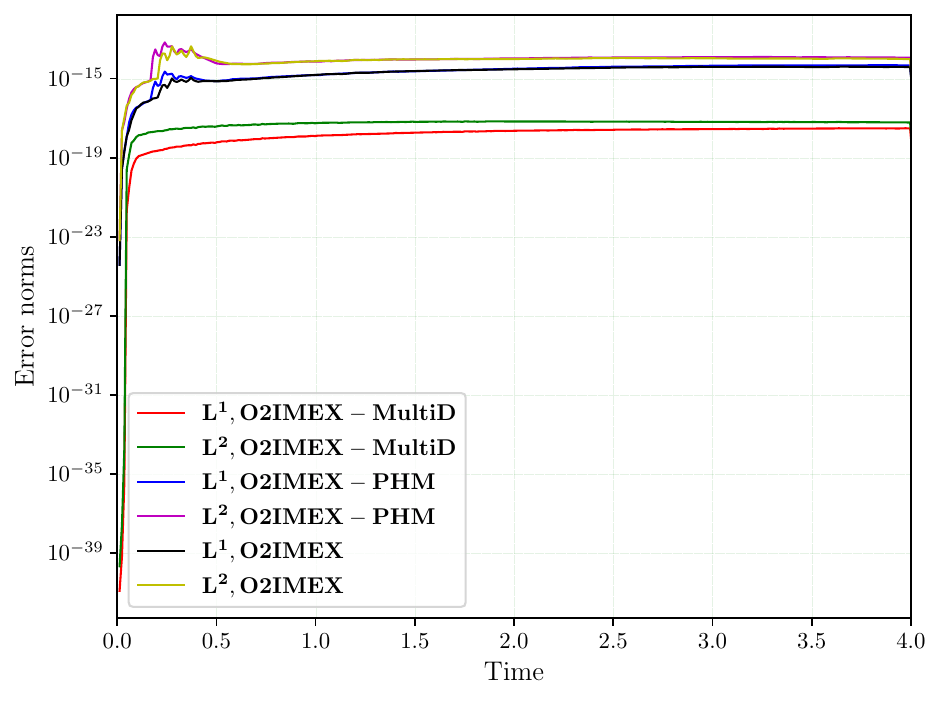}
			\label{fig:blast_E_norm_imp_es_Rus}}
		\caption{\nameref{test:2d_blast}  Evolution of the divergence constraints errors of magnetic and electric fields for explicit (\textbf{O2EXP-MultiD}, \textbf{O2EXP-PHM} and textbf{O2EXP}) and IMEX (\textbf{O2IMEX-MultiD}, \textbf{O2IMEX-PHM}, and \textbf{O2IMEX}) schemes using $200\times 200$ cells, for weakly magnetized medium, $B_0=0.1$.}
		\label{fig:blast_div_0p1}
	\end{center}
\end{figure}

\begin{figure}[!htbp]
	\begin{center}
		\subfigure[$\|\na \cdot \mathbf{B}^n\|_1$,  and $\|\na \cdot \mathbf{B}^n\|_2$  errors for explicit schemes.]{
			\includegraphics[width=1.4in, height=1.25in]{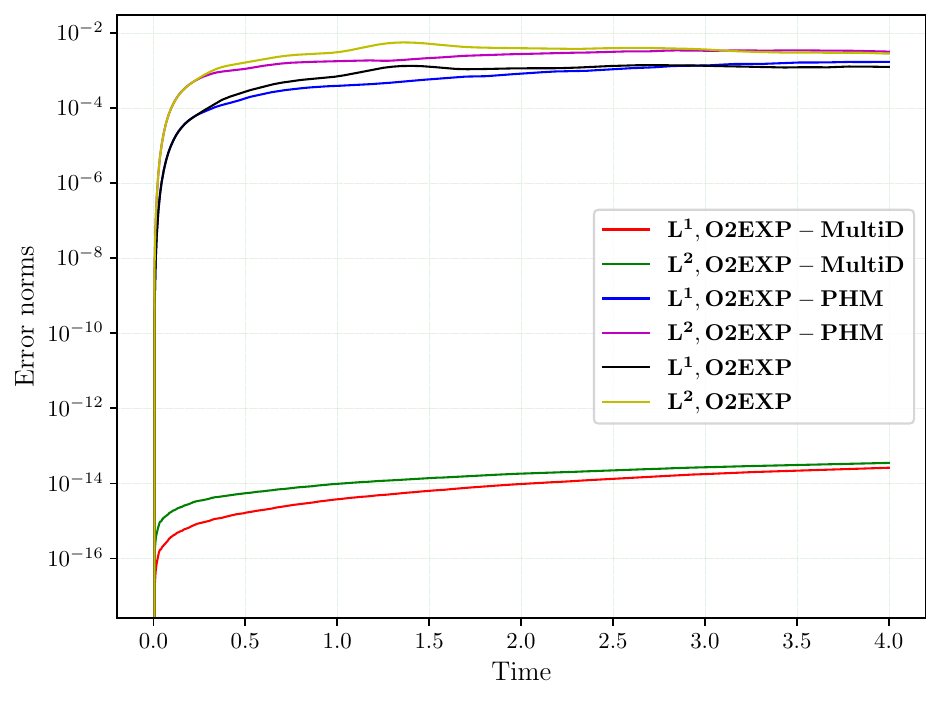}
		}
		\subfigure[$\|\na \cdot \mathbf{B}^n\|_1$,  and $\|\na \cdot \mathbf{B}^n\|_2$  errors for IMEX schemes.]{
			\includegraphics[width=1.4in, height=1.25in]{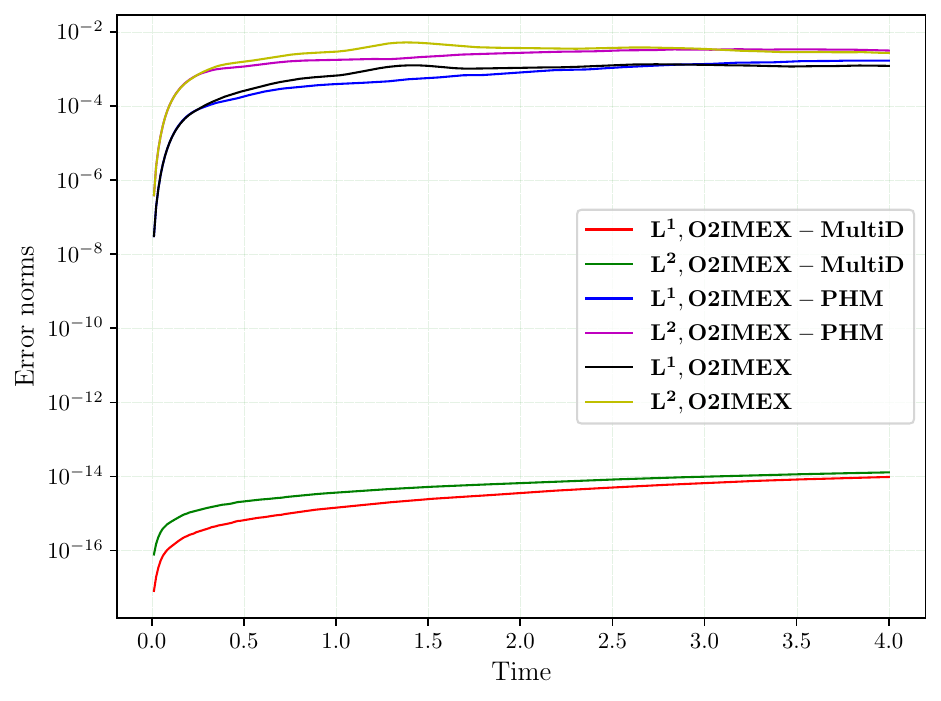}
		}
		\subfigure[$\|\na\cdot\Eb\|_{1}^{E},$ and $\|\na\cdot\Eb\|_{2}^{E}$ errors, for explicit schemes.]{
		\includegraphics[width=1.4in, height=1.25in]{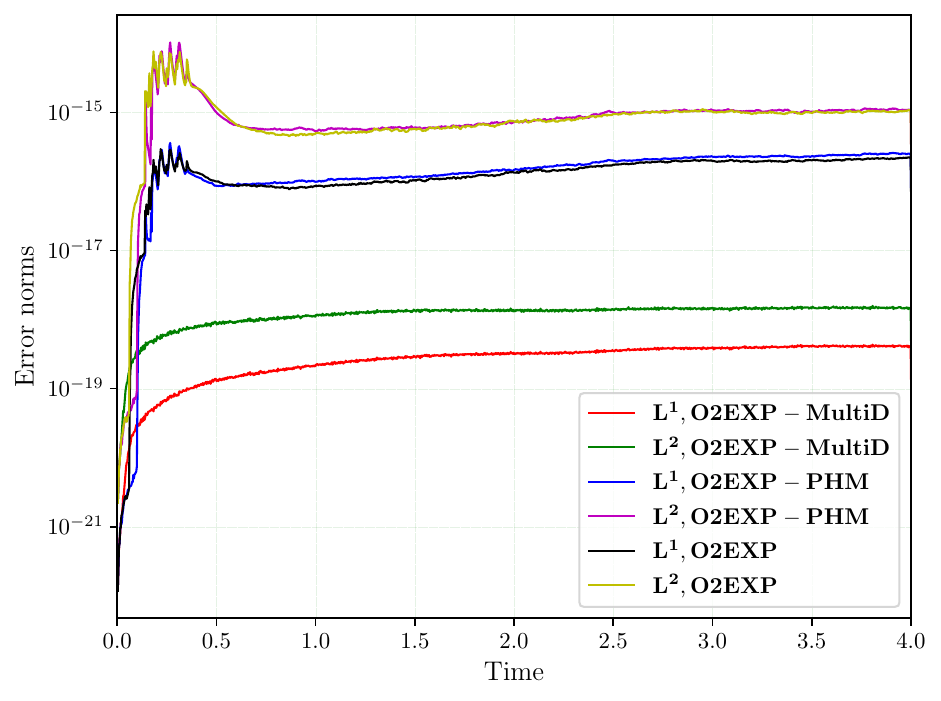}
	}
	\subfigure[$\|\na\cdot\Eb\|_{1}^{I},$ and $\|\na\cdot\Eb\|_{2}^{I}$ errors, for IMEX schemes.]{
		\includegraphics[width=1.4in, height=1.25in]{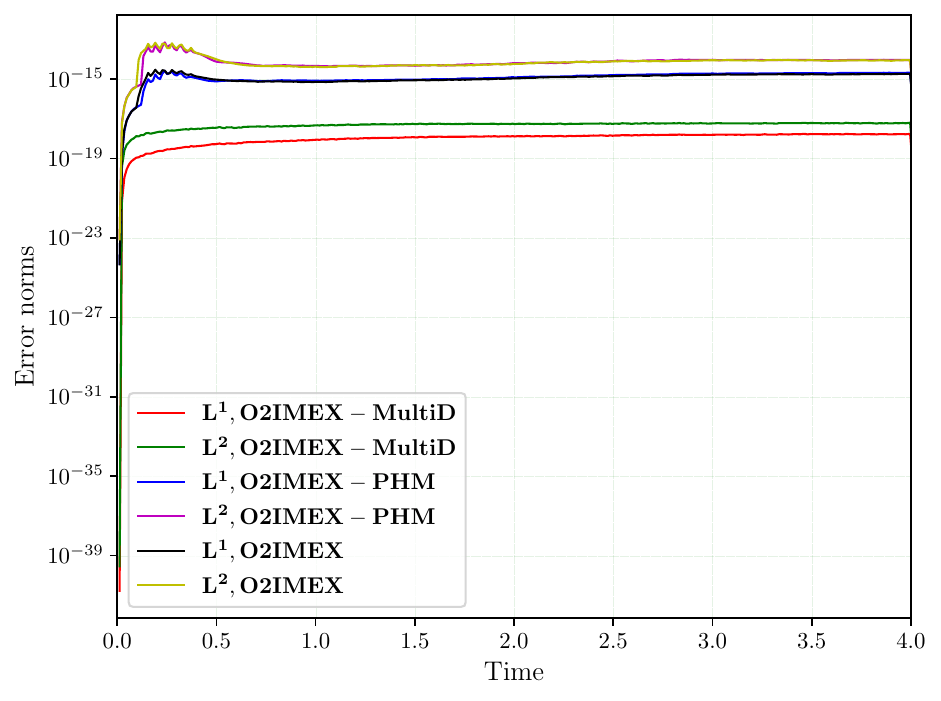}
	}
		\caption{\nameref{test:2d_blast}  Evolution of the divergence constraints errors of magnetic and electric fields for explicit (\textbf{O2EXP-MultiD}, \textbf{O2EXP-PHM} and \textbf{O2EXP}) and IMEX (\textbf{O2IMEX-MultiD}, \textbf{O2IMEX-PHM}, and \textbf{O2IMEX}) schemes using $200\times 200$ cells, for strongly magnetized medium, $B_0=1.0$.}
		\label{fig:blast_div_1p0}
	\end{center}
\end{figure}

\subsubsection{Relativistic two-fluid GEM challenge problem} 
\label{test:2d_gem}
This test case is motivated by the non-relativistic two-fluid relativistic Geospace Environment Modelling (GEM) magnetic reconnection problem from  \cite{Birn2001,wang2020}. The relativistic version is considered in \cite{Amano2016, Balsara2016, Bhoriya2023,bhoriya_entropydg_2023}. Here, we consider the additional resistive effects given by Eqn. \eqref{eq:resistive} with the resistivity constant as $\eta=0.01$. The domain of the problem is given by $[-L_x/2,L_x/2] \times [-L_y/2,L_y/2]$ where $L_x=8\pi$ and $L_y=4\pi$. We consider the periodic boundary conditions at $x=\pm L_x/2$ and conducting wall boundary at $y=\pm L_y/2$ boundaries. We take $r_i=1.0$ and $r_e=-25.0.$ We consider the unperturbed magnetic field of $(B_0 \tan (y/d),0,0)$ with $B_0=1.0$, where $d=1.0$ is the thickness of the current sheet. After perturbing the magnetic field the initial conditions are given by,
\begin{align*}
	\begin{pmatrix}
		\rho_i  \\
		u_{z}^i \\
		p_i     \\ 
		\rho_e  \\
		u_{z}^e \\
		p_e     \\ 
		B_x     \\
		B_y
	\end{pmatrix} =
	\begin{pmatrix}
		n                                                          \\
		\frac{1}{2d}\frac{B_0 \text{sech}^2(y/d)}{n}               \\
		0.2 + \frac{B_0^2  \text{sech}^2(y/d)}{4} \frac{5}{24 \pi}  \\
		\frac{m_e}{m_i} n                                          \\
		-u_i^z                                                   \\
		- p_i                                                      \\ 
		B_0 \tan (y/d) - B_0 \psi_0 \frac{\pi}{L_y} \cos(\frac{\pi x}{L_x}) \sin(\frac{\pi y}{L_y})
		\\
		B_0 \psi_0 \frac{\pi}{L_x} \sin(\frac{\pi x}{L_x}) \cos(\frac{\pi y}{L_y})
	\end{pmatrix}
\end{align*}
where $n=\mathrm{sech}^2(y/d)+0.2$. All other variables are set to zero. The adiabatic indices are taken as $\gamma_i =\gamma_e= 4.0/3.0$.

We use $512\times 256$ cells and compute till time $t=100.$ In Figures \ref{fig:gem_o2_80_exp_multiD} and \ref{fig:gem_o2_80_imp_multiD}, we plotted the results for \mde~and\mdi~schemes, respectively at time $t=80$. We observe that magnetic reconnection is underway. We also note that flow is indeed relativistic as $u_e^x$ reached the value of $0.6$. Both schemes have similar accuracy and performance.
 
In Figure \eqref{fig:reconn_rate}, we also plot the time evolution of the reconnected magnetic flux,
$$\psi(t) = \dfrac{1}{2 B_0} \int_{-L_x/2}^{L_x/2}
| B_y(x,y=0,t) | dx.
$$
for all the schemes. We note that all of them have similar reconnection rates and are comparable to the results from \cite{Amano2016} (Solid lines). In Figure~\ref {fig:gem_div_norms}, we plot the evolution of the errors in divergence constraints. We observe that the proposed schemes perform much better than the other schemes for both constraints. In particular, we note that Gauss's law is approximated much more accurately by the proposed schemes.

\begin{figure}[!htbp]
	\begin{center}
		\subfigure[$\rho_i + \rho_e$]{
			\includegraphics[width=2.9in, height=1.4in]{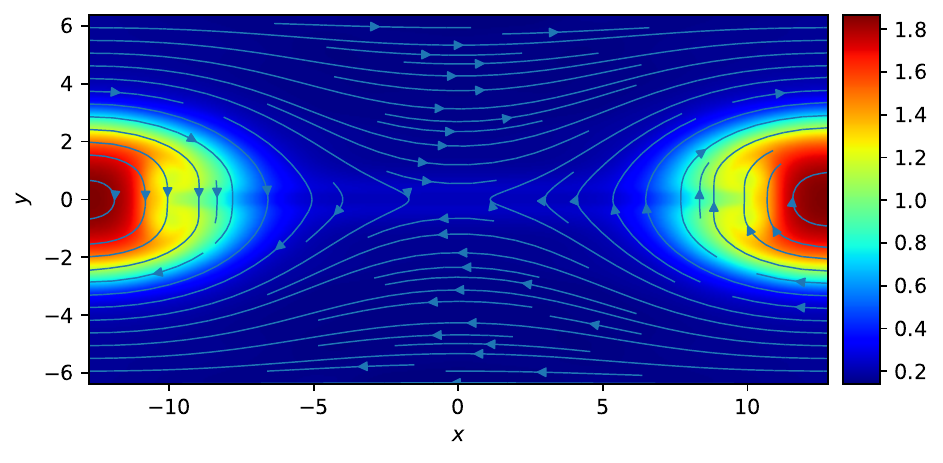}
			\label{fig:gem_exp_multiD_t80_512_rho}}
		\subfigure[$B_z$]{
			\includegraphics[width=2.9in, height=1.4in]{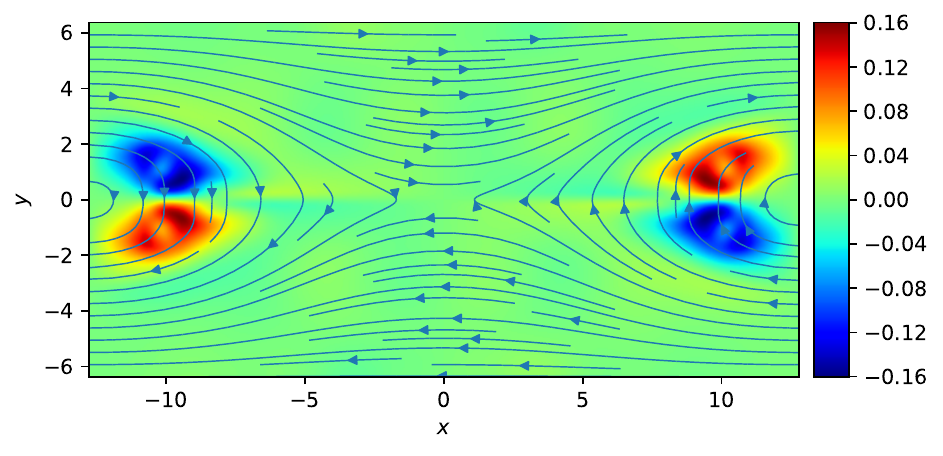}
			\label{fig:gem_exp_multiD_t80_512_Bz}}
		\subfigure[$u_i^x$]{
			\includegraphics[width=2.9in, height=1.4in]{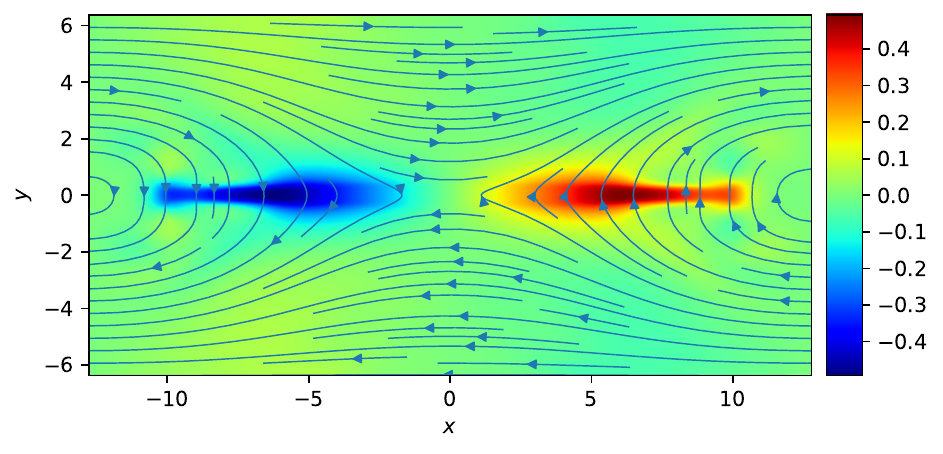}
			\label{fig:gem_exp_multiD_t80_512_uxi}}
		\subfigure[$u_e^x$]{
			\includegraphics[width=2.9in, height=1.4in]{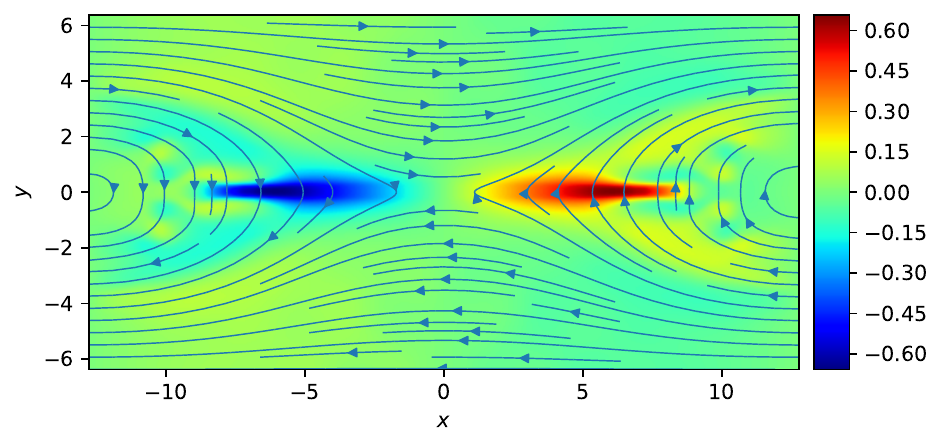}
			\label{fig:gem_exp_multiD_t80_512_uxe}}
		\caption{\nameref{test:2d_gem}: Plots of $\rho_i + \rho_e$, $B_z$-component, $u_i^x$ and Electron $u_e^x$ using $512 \times 256$ cells, for \textbf{O2EXP-MultiD} scheme, at time $t=80.0$. }
		\label{fig:gem_o2_80_exp_multiD}
	\end{center}
\end{figure}
\begin{figure}[!htbp]
	\begin{center}
		\subfigure[$\rho_i + \rho_e$]{
			\includegraphics[width=2.9in, height=1.4in]{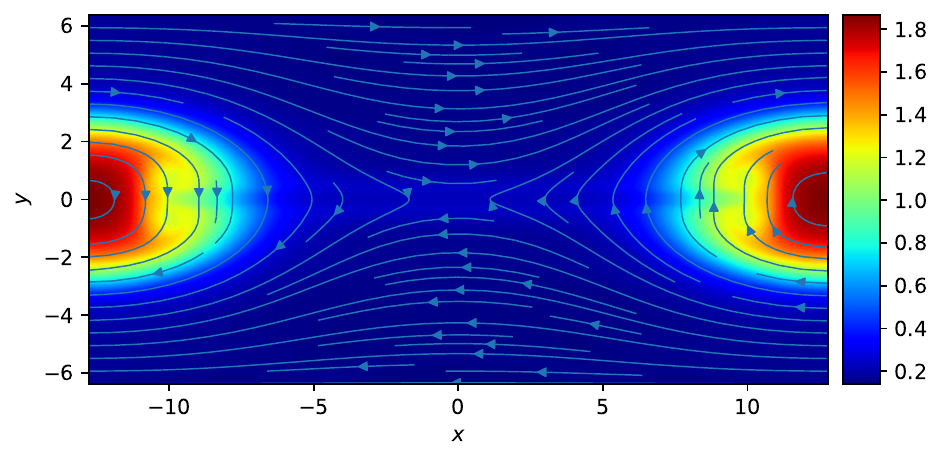}
			\label{fig:gem_imp_multiD_t80_512_rho}}
		\subfigure[$B_z$]{
			\includegraphics[width=2.9in, height=1.4in]{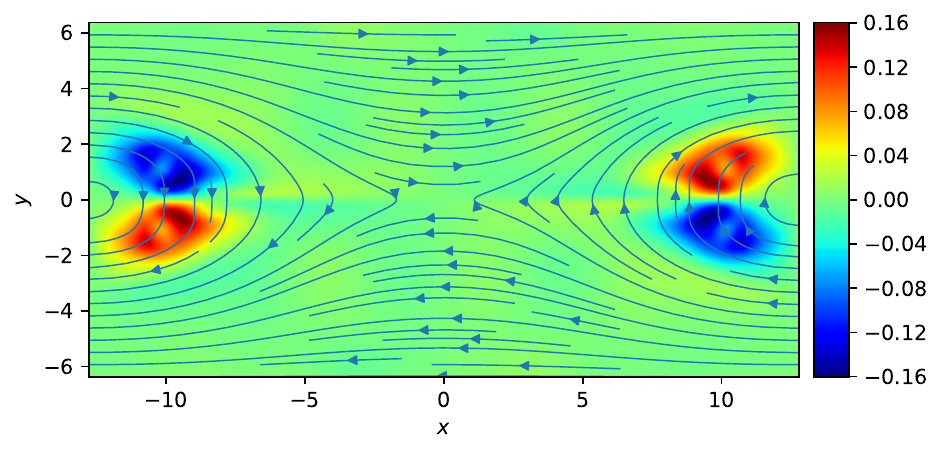}
			\label{fig:gem_imp_multiD_t80_512_Bz}}
		\subfigure[$u_i^x$]{
			\includegraphics[width=2.9in, height=1.4in]{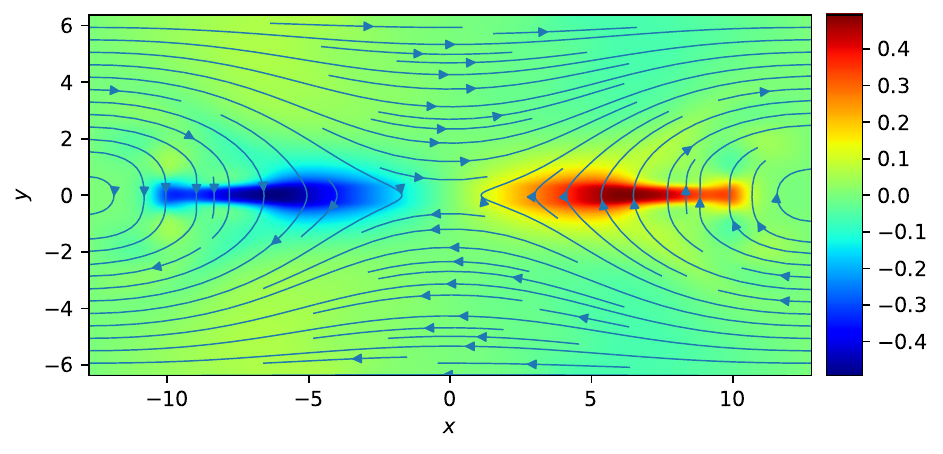}
			\label{fig:gem_imp_multiD_t80_512_uxi}}
		\subfigure[$u_e^x$]{
			\includegraphics[width=2.9in, height=1.4in]{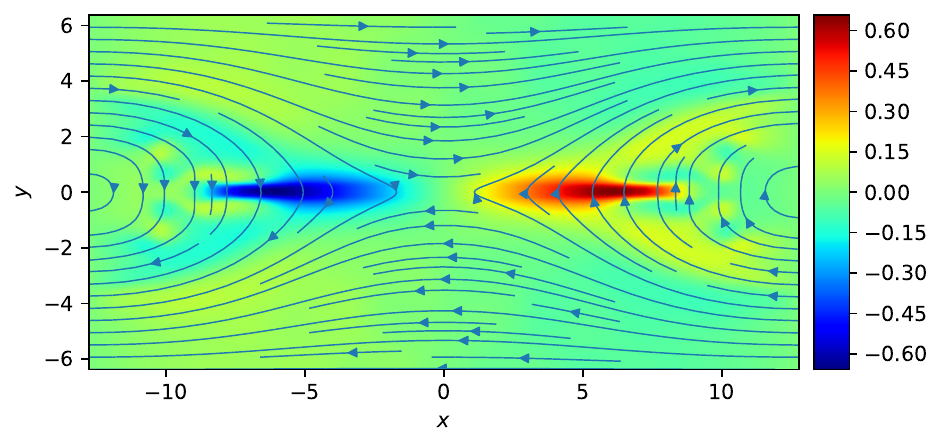}
			\label{fig:gem_imp_multiD_t80_512_uxe}}
		\caption{\nameref{test:2d_gem}: Plots of $\rho_i + \rho_e$, $B_z$-component, $u_i^x$ and Electron $u_e^x$ using $512 \times 256$ cells, for \textbf{O2EXP-MultiD} scheme, at time $t=80.0$. }
		\label{fig:gem_o2_80_imp_multiD}
	\end{center}
\end{figure}
\begin{figure}[h]
	\begin{center}
		\includegraphics[width=3.5in, height=2.9in]{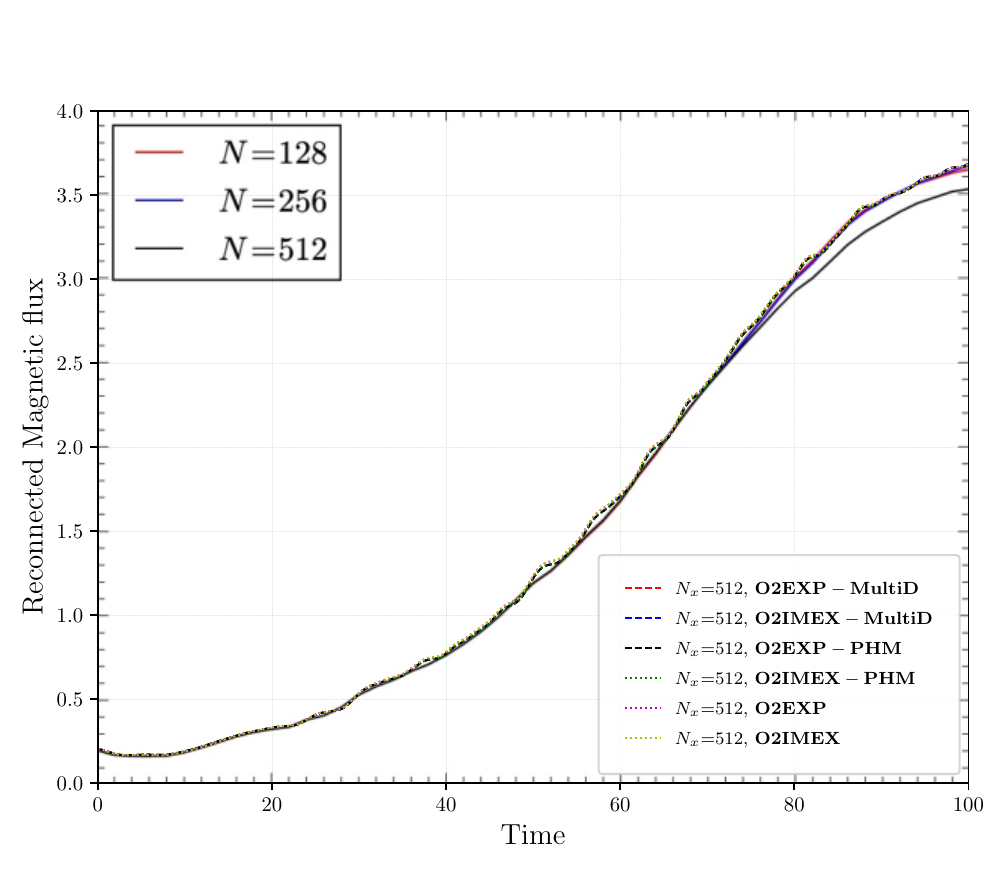}
		\caption{Time development of the reconnected magnetic flux using $512\times256$ cells, for different schemes. We overlay the plot on reconnection rates from \cite{Amano2016} (Solid lines). }
		\label{fig:reconn_rate}
	\end{center}
\end{figure}
\begin{figure}[!htbp]
	\begin{center}
		\subfigure[$\|\na \cdot \mathbf{B}^n\|_1$,  and $\|\na \cdot \mathbf{B}^n\|_2$  errors for explicit schemes.]{
			\includegraphics[width=1.4in, height=1.25in]{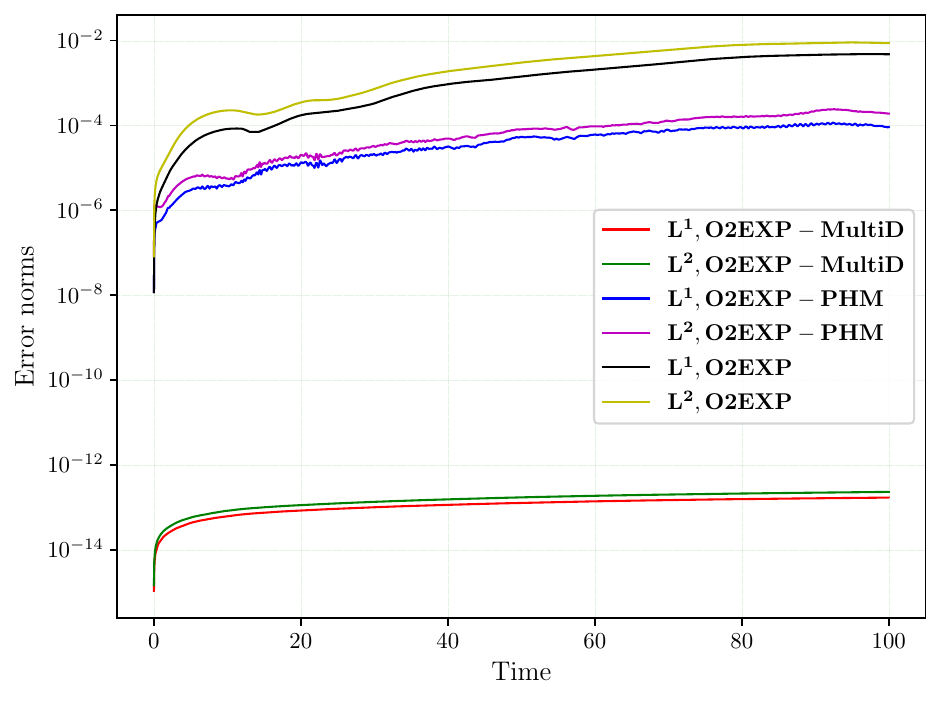}
			\label{fig:gem_divB_exp}}
		\subfigure[$\|\na \cdot \mathbf{B}^n\|_1$,  and $\|\na \cdot \mathbf{B}^n\|_2$  errors for IMEX schemes.]{
			\includegraphics[width=1.4in, height=1.25in]{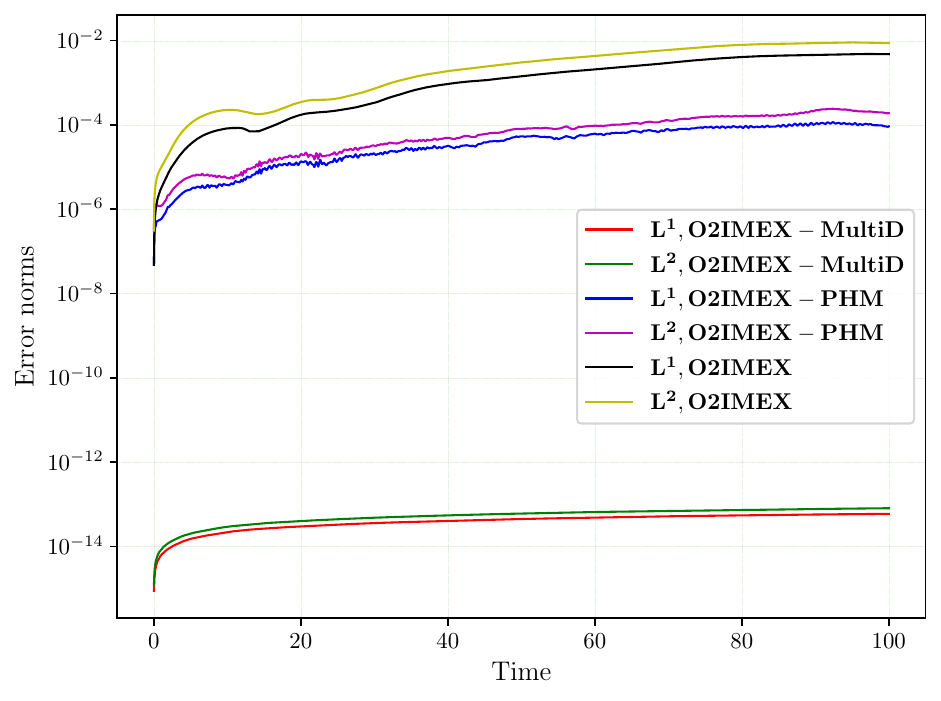}
			\label{fig:gem_divB}}
		\subfigure[$\|\na\cdot\Eb\|_{1}^{E},$ and $\|\na\cdot\Eb\|_{2}^{E}$ errors, for explicit schemes.]{
			\includegraphics[width=1.4in, height=1.25in]{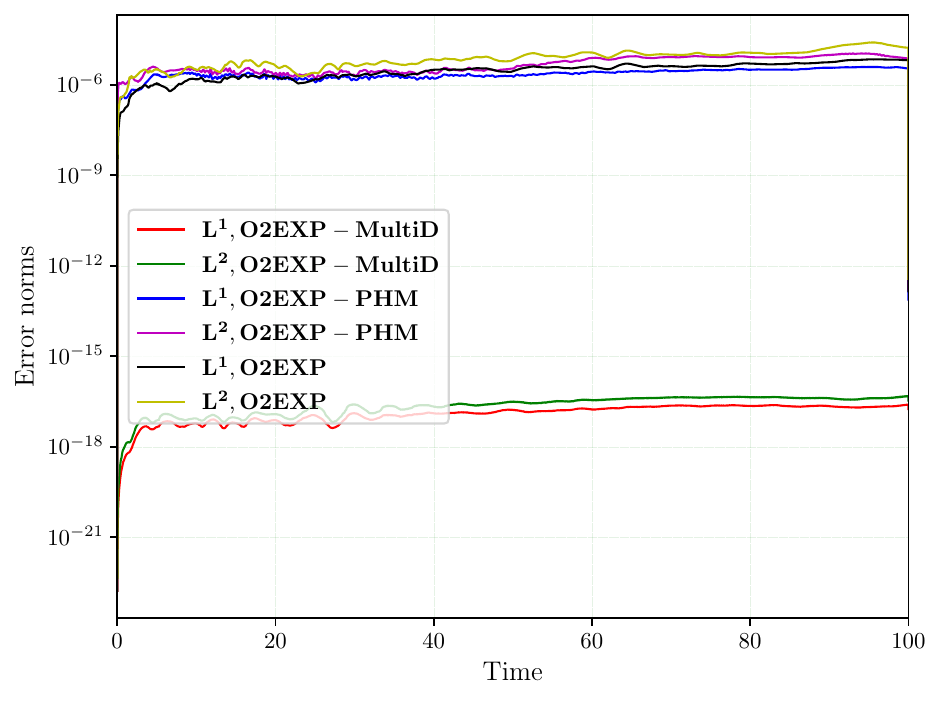}
			\label{fig:gem_divE_exp}}
		\subfigure[$\|\na\cdot\Eb\|_{1}^{I},$ and $\|\na\cdot\Eb\|_{2}^{I}$ errors, for IMEX schemes.]{
			\includegraphics[width=1.4in, height=1.25in]{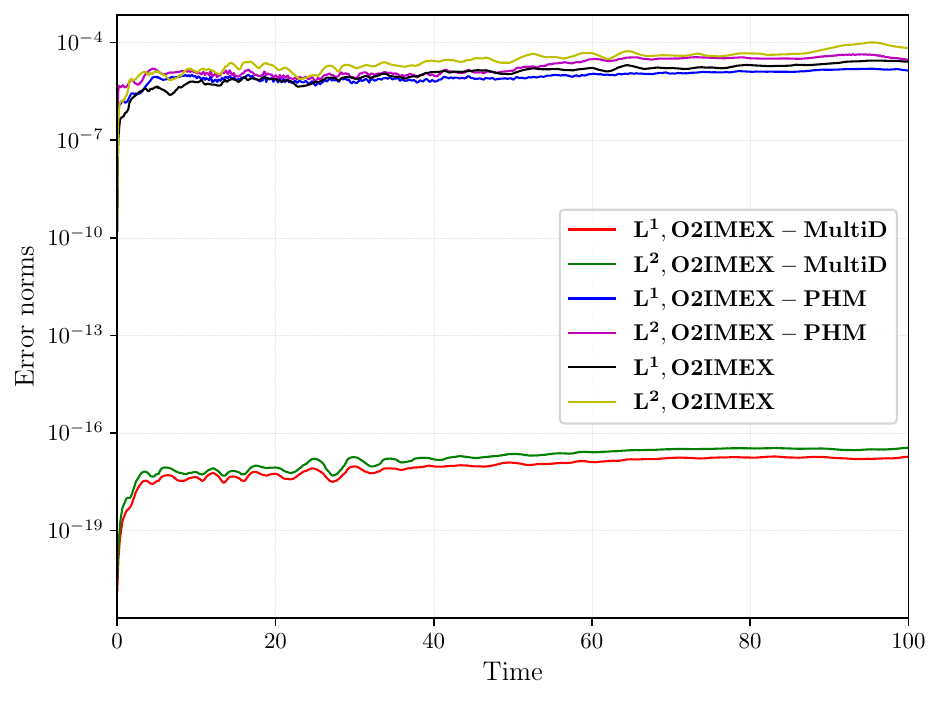}
			\label{fig:gem_divE}}
		\caption{\nameref{test:2d_gem}: Evolution of the divergence constraint errors for explicit (\textbf{O2EXP-MultiD}, \textbf{O2EXP-PHM} and \textbf{O2EXP}) and IMEX (\textbf{O2IMEX-MultiD}, \textbf{O2IMEX-PHM}, and \textbf{O2IMEX}) schemes using $512\times 256$ cells.}
		\label{fig:gem_div_norms}
	\end{center}
\end{figure}	

\section{Conclusions}
\label{sec:con}
In this article, we have presented second-order, co-located, entropy-stable schemes for two-fluid relativistic plasma flow equations. The proposed schemes are based on using a multidimensional Riemann solver at the cell edges to define the edge numerical flux for Maxwell's equations. Higher order accuracy is achieved using MinMod-based reconstruction in the diagonal directions since the state at vertices is required for the Riemann solver at the vertices. The schemes are shown to preserve the fully discrete version of the divergence-free magnetic field and Gauss law for the electric field. The schemes are then tested on several challenging test problems. In particular, for the two-dimensional test cases, we demonstrate that the proposed schemes preserve the divergence constraint up to the machine's precision.

\section*{Acknowledgements}
The work of Harish Kumar is supported in parts by VAJRA grant No. VJR/2018/000129 by the Dept. of Science and Technology, Govt. of India. Harish Kumar and Jaya Agnihotri acknowledge the support of FIST Grant Ref No. SR/FST/MS-1/2019/45 by the Dept. of Science and Technology, Govt. of India. The work of Praveen Chandrashekar is supported by the Department of Atomic Energy, Government of India, under project no. 12-R\&D-TFR-5.01-0520. 
\section*{Conflict of interest}
The authors declare that they have no conflict of interest.
\section*{Data Availability Declaration}
Data will be made available on reasonable request.

\bibliographystyle{elsarticle-num}
	
\bibliography{main}

\begin{thebibliography}{10}
\expandafter\ifx\csname url\endcsname\relax
  \def\url#1{\texttt{#1}}\fi
\expandafter\ifx\csname urlprefix\endcsname\relax\def\urlprefix{URL }\fi
\expandafter\ifx\csname href\endcsname\relax
  \def\href#1#2{#2} \def\path#1{#1}\fi

\bibitem{Komissarov1999}
S.~S. Komissarov, {A Godunov-type scheme for relativistic magnetohydrodynamics}, Mon. Not. R. Astron. Soc. 303~(2) (1999) 343--366.
\newblock \href {https://doi.org/10.1046/j.1365-8711.1999.02244.x} {\path{doi:10.1046/j.1365-8711.1999.02244.x}}.

\bibitem{Balsara2001}
D.~Balsara, {Total Variation Diminishing Scheme for Relativistic Magnetohydrodynamics}, The Astrophysical Journal Supplement Series 132~(1) (2001) 83--101.
\newblock \href {https://doi.org/10.1086/318941} {\path{doi:10.1086/318941}}.

\bibitem{DelZanna2003}
L.~{Del Zanna}, N.~Bucciantini, P.~Londrillo, {An efficient shock-capturing central-type scheme for multidimensional relativistic flows II. Magnetohydrodynamics}, Astronomy and Astrophysics 400~(2) (2003) 397--413.
\newblock \href {https://doi.org/10.1051/0004-6361:20021641} {\path{doi:10.1051/0004-6361:20021641}}.

\bibitem{Mignone2006}
A.~Mignone, G.~Bodo, {An HLLC Riemann solver for relativistic flows – II. Magnetohydrodynamics}, Monthly Notices of the Royal Astronomical Society 368~(3) (2006) 1040--1054.
\newblock \href {https://doi.org/10.1111/J.1365-2966.2006.10162.X} {\path{doi:10.1111/J.1365-2966.2006.10162.X}}.

\bibitem{Komissarov2007}
S.~S. Komissarov, {Multidimensional numerical scheme for resistive relativistic magnetohydrodynamics}, Monthly Notices of the Royal Astronomical Society 382~(3) (2007) 995--1004.
\newblock \href {https://doi.org/10.1111/j.1365-2966.2007.12448.x} {\path{doi:10.1111/j.1365-2966.2007.12448.x}}.

\bibitem{Amano2016}
T.~Amano, {A Second-order divergence-constrained multidimensional numerical scheme for relativistic two-fluid electrodynamics}, The Astrophysical Journal 831~(1) (2016).
\newblock \href {https://doi.org/10.3847/0004-637x/831/1/100} {\path{doi:10.3847/0004-637x/831/1/100}}.

\bibitem{Balsara2016}
D.~S. Balsara, T.~Amano, S.~Garain, J.~Kim, {A high-order relativistic two-fluid electrodynamic scheme with consistent reconstruction of electromagnetic fields and a multidimensional Riemann solver for electromagnetism}, J. Comput. Phys. 318 (2016) 169--200.
\newblock \href {https://doi.org/10.1016/j.jcp.2016.05.006} {\path{doi:10.1016/j.jcp.2016.05.006}}.

\bibitem{Hakim2006}
A.~Hakim, J.~Loverich, U.~Shumlak, {A high resolution wave propagation scheme for ideal Two-Fluid plasma equations}, Journal of Computational Physics 219~(1) (2006) 418--442.
\newblock \href {https://doi.org/10.1016/j.jcp.2006.03.036} {\path{doi:10.1016/j.jcp.2006.03.036}}.

\bibitem{loverich2011}
J.~Loverich, A.~Hakim, U.~Shumlak, {A discontinuous Galerkin method for ideal two-fluid plasma equations}, Communications in Computational Physics 9~(2) (2011) 240--268.
\newblock \href {https://doi.org/10.4208/cicp.250509.210610a} {\path{doi:10.4208/cicp.250509.210610a}}.

\bibitem{wang2020}
L.~Wang, A.~H. Hakim, J.~Ng, C.~Dong, K.~Germaschewski, Exact and locally implicit source term solvers for multifluid-maxwell systems, Journal of Computational Physics 415 (2020) 109510.
\newblock \href {https://doi.org/10.1016/j.jcp.2020.109510} {\path{doi:10.1016/j.jcp.2020.109510}}.

\bibitem{kumar2012entropy}
H.~Kumar, S.~Mishra, {Entropy stable numerical schemes for two-fluid plasma equations}, Journal of scientific computing 52~(2) (2012) 401--425.
\newblock \href {https://doi.org/10.1007/s10915-011-9554-7} {\path{doi:10.1007/s10915-011-9554-7}}.

\bibitem{Abgrall2014}
R.~Abgrall, H.~Kumar, {Robust Finite Volume Schemes for Two-Fluid Plasma Equations}, Journal of Scientific Computing 60~(3) (2014) 584--611.
\newblock \href {https://doi.org/10.1007/s10915-013-9809-6} {\path{doi:10.1007/s10915-013-9809-6}}.

\bibitem{Meena2019}
A.~K. Meena, H.~Kumar, {Robust numerical schemes for Two-Fluid Ten-Moment plasma flow equations}, Zeitschrift fur Angewandte Mathematik und Physik 70~(1) (2019) 1--30.
\newblock \href {https://doi.org/10.1007/s00033-018-1061-3} {\path{doi:10.1007/s00033-018-1061-3}}.

\bibitem{zenitani2009relativistic}
S.~Zenitani, M.~Hesse, A.~Klimas, {Relativistic two-fluid simulations of guide field magnetic reconnection}, Astrophysical Journal 705~(1) (2009) 907--913.
\newblock \href {https://doi.org/10.1088/0004-637X/705/1/907} {\path{doi:10.1088/0004-637X/705/1/907}}.

\bibitem{Zenitani2010}
S.~Zenitani, M.~Hesse, A.~Klimas, {Resistive magnetohydrodynamic simulations of relativistic magnetic reconnection}, Astrophysical Journal Letters 716~(2) (2010) L214----L218.
\newblock \href {https://doi.org/10.1088/2041-8205/716/2/L214} {\path{doi:10.1088/2041-8205/716/2/L214}}.

\bibitem{Amano2013}
T.~Amano, J.~G. Kirk, {The role of superluminal electromagnetic waves in pulsar wind termination shocks}, Astrophysical Journal 770~(1) (2013) 18.
\newblock \href {https://doi.org/10.1088/0004-637X/770/1/18} {\path{doi:10.1088/0004-637X/770/1/18}}.

\bibitem{Barkov2014}
M.~Barkov, S.~S. Komissarov, V.~Korolev, A.~Zankovich, {A multidimensional numerical scheme for two-fluid relativistic magnetohydrodynamics}, Monthly Notices of the Royal Astronomical Society 438~(1) (2014) 704--716.
\newblock \href {https://doi.org/10.1093/mnras/stt2247} {\path{doi:10.1093/mnras/stt2247}}.

\bibitem{Bhoriya2023}
D.~Bhoriya, H.~Kumar, P.~Chandrashekar, High-order finite-difference entropy stable schemes for two-fluid relativistic plasma flow equations, Journal of Computational Physics 488 (2023) 112207.

\bibitem{bhoriya_entropydg_2023}
D.~Bhoriya, B.~Biswas, H.~Kumar, P.~Chandrashekhar, \href{https://doi.org/10.1007/s10915-023-02387-z}{Entropy {Stable} {Discontinuous} {Galerkin} {Schemes} for {Two}-{Fluid} {Relativistic} {Plasma} {Flow} {Equations}}, Journal of Scientific Computing 97~(3) (2023) 72.
\newblock \href {https://doi.org/10.1007/s10915-023-02387-z} {\path{doi:10.1007/s10915-023-02387-z}}.
\newline\urlprefix\url{https://doi.org/10.1007/s10915-023-02387-z}

\bibitem{Zenitani2009a}
S.~Zenitani, M.~Hesse, A.~Klimas, {Two-fluid magnetohydrodynamic simulations of relativistic magnetic reconnection}, Astrophysical Journal 696~(2) (2009) 1385--1401.
\newblock \href {https://doi.org/10.1088/0004-637X/696/2/1385} {\path{doi:10.1088/0004-637X/696/2/1385}}.

\bibitem{jaya2024}
J.~Agnihotri, D.~Bhoriya, H.~Kumar, P.~Chandrashekhar, D.~S. Balsara, Second order divergence constraint preserving entropy stable finite difference schemes for ideal two-fluid plasma flow equations, Submitted (2024).

\bibitem{balsara2014}
D.~S. Balsara, M.~Dumbser, R.~Abgrall, {Multidimensional HLLC Riemann solver for unstructured meshes--with application to Euler and MHD flows}, Journal of Computational Physics 261 (2014) 172--208.
\newblock \href {https://doi.org/10.1016/j.jcp.2013.12.029} {\path{doi:10.1016/j.jcp.2013.12.029}}.

\bibitem{chandrashekar2020}
P.~Chandrashekar, R.~Kumar, {Constraint preserving discontinuous Galerkin method for ideal compressible MHD on 2-D Cartesian grids}, Journal of Scientific Computing 84~(2) (2020) 1--43.
\newblock \href {https://doi.org/10.1007/s10915-020-01289-8} {\path{doi:10.1007/s10915-020-01289-8}}.

\bibitem{Bhoriya2020}
D.~Bhoriya, H.~Kumar, {Entropy-stable schemes for relativistic hydrodynamics equations}, Zeitschrift fur Angewandte Mathematik und Physik 71~(1) (2020).
\newblock \href {https://doi.org/10.1007/s00033-020-1250-8} {\path{doi:10.1007/s00033-020-1250-8}}.

\bibitem{Schneider1993}
V.~Schneider, U.~Katscher, D.~H. Rischke, B.~Waldhauser, J.~A. Maruhn, C.~D. Munz, {New Algorithms for Ultra-relativistic Numerical Hydrodynamics}, Journal of Computational Physics 105~(1) (1993) 92--107.
\newblock \href {https://doi.org/10.1006/jcph.1993.1056} {\path{doi:10.1006/jcph.1993.1056}}.

\bibitem{jiang1996origin}
B.-n. Jiang, J.~Wu, L.~A. Povinelli, The origin of spurious solutions in computational electromagnetics, Journal of computational physics 125~(1) (1996) 104--123.

\bibitem{bond2016plasma}
D.~M. Bond, V.~Wheatley, R.~Samtaney, {Plasma flow simulation using the two-fluid model}, in: Proceedings of the 20th Australasian Fluid Mechanics Conference, 2016.

\bibitem{munz2000}
C.-D. Munz, P.~Omnes, R.~Schneider, E.~Sonnendr{\"{u}}cker, U.~Voss, {Divergence correction techniques for Maxwell solvers based on a hyperbolic model}, Journal of Computational Physics 161~(2) (2000) 484--511.
\newblock \href {https://doi.org/10.1006/jcph.2000.6507} {\path{doi:10.1006/jcph.2000.6507}}.

\bibitem{Fjordholm2012}
U.~S. Fjordholm, S.~Mishra, E.~Tadmor, {Arbitrarily High-order Accurate Entropy Stable Essentially Nonoscillatory Schemes for Systems of Conservation Laws}, SIAM Journal on Numerical Analysis 50~(2) (2012) 544--573.
\newblock \href {https://doi.org/10.1137/110836961} {\path{doi:10.1137/110836961}}.

\bibitem{Tadmor1987}
E.~Tadmor, {The Numerical Viscosity of Entropy Stable Schemes for Systems of Conservation Laws. I}, Mathematics of Computation 49~(179) (1987) 91.
\newblock \href {https://doi.org/10.2307/2008251} {\path{doi:10.2307/2008251}}.

\bibitem{balsara2010}
D.~S. Balsara, {Multidimensional HLLE Riemann solver: application to Euler and magnetohydrodynamic flows}, Journal of Computational Physics 229~(6) (2010) 1970--1993.
\newblock \href {https://doi.org/10.1016/j.jcp.2009.11.018} {\path{doi:10.1016/j.jcp.2009.11.018}}.

\bibitem{Balsara2016aderweno}
D.~S. Balsara, J.~Kim, {A subluminal relativistic magnetohydrodynamics scheme with ADER-WENO predictor and multidimensional Riemann solver-based corrector}, Journal of Computational Physics 312 (2016) 357--384.
\newblock \href {https://doi.org/10.1016/j.jcp.2016.02.001} {\path{doi:10.1016/j.jcp.2016.02.001}}.

\bibitem{Gottlieb2001}
S.~Gottlieb, C.~W. Shu, E.~Tadmor, {Strong stability-preserving high-order time discretization methods}, SIAM Rev. 43~(1) (2001) 89--112.
\newblock \href {https://doi.org/10.1137/S003614450036757X} {\path{doi:10.1137/S003614450036757X}}.

\bibitem{Pareschi2005}
L.~Pareschi, G.~Russo, {Implicit-explicit Runge-Kutta schemes and applications to hyperbolic systems with relaxation}, J. Sci. Comput. 25~(1) (2005) 129--155.
\newblock \href {https://doi.org/10.1007/s10915-004-4636-4} {\path{doi:10.1007/s10915-004-4636-4}}.

\bibitem{Dennis1996}
J.~E. Dennis, R.~B. Schnabel, {Numerical Methods for Unconstrained Optimization and Nonlinear Equations}, Society for Industrial and Applied Mathematics, 1996.
\newblock \href {https://doi.org/10.1137/1.9781611971200} {\path{doi:10.1137/1.9781611971200}}.

\bibitem{brio1988}
M.~Brio, C.~C. Wu, An upwind differencing scheme for the equations of ideal magnetohydrodynamics, Journal of computational physics 75~(2) (1988) 400--422.

\bibitem{Orszag1979}
S.~A. Orszag, C.~M. Tang, {Small-scale structure of two-dimensional magnetohydrodynamic turbulence}, J. Fluid Mech. 90~(1) (1979) 129--143.
\newblock \href {https://doi.org/10.1017/S002211207900210X} {\path{doi:10.1017/S002211207900210X}}.

\bibitem{Birn2001}
J.~Birn, J.~F. Drake, M.~A. Shay, B.~N. Rogers, R.~E. Denton, M.~Hesse, M.~Kuznetsova, Z.~W. Ma, A.~Bhattacharjee, A.~Otto, P.~L. Pritchett, {Geospace Environmental Modeling (GEM) Magnetic Reconnection Challenge}, Journal of Geophysical Research: Space Physics 106~(A3) (2001) 3715--3719.
\newblock \href {https://doi.org/10.1029/1999JA900449} {\path{doi:10.1029/1999JA900449}}.

\end{thebibliography}

\end{document}